\documentclass[12pt]{amsart}
\usepackage{amsmath,amssymb,amsfonts,latexsym}
\usepackage{mathrsfs}%pour les caracteres en calligraphie
\usepackage[dvips]{geometry}
\usepackage{array}
\DeclareFontFamily{T1}{calligra}{}
\DeclareFontShape{T1}{calligra}{m}{n} {<-> callig15}{}

\newtheorem*{thm}{Theorem}

\newtheorem*{lem}{Lemma}

\newtheorem*{prop}{Proposition}
\theoremstyle{definition}% Pour ne pas avoir les \'enonc\'es en italique.

\theoremstyle{remark}

\newtheorem*{rem}{Remark}

\numberwithin{equation}{subsection}
\newcommand{\codim}{{\rm codim\,}}

\newcommand{\ov}{\overline}
\newcommand{\B}{\mathbf{B}}

\newcommand{\X}{\mathcal{X}}
\newcommand{\W}{\mathcal{W}}

\newcommand{\Or}{\mathcal{O}}

\newcommand{\G}{\mathfrak{g}}
\newcommand{\Sl}{\mathfrak{s}\mathfrak{l}}

\newcommand{\gh}{\mathfrak{h}}

\newcommand{\bo}{\mathfrak{b}}

\newcommand{\nil}{\mathfrak{n}}

\newcommand{\Sp}{\mathfrak{s}\mathfrak{p}}

\newcommand{\Co}{\mathbb{C}}

\newcommand{\vb}{\vrule height 20pt depth 16pt}
\newcommand{\vbs}{\vrule height 10pt depth 10pt}

\newcommand{\vsa}{\noalign{\vskip -3pt}}
\newcommand{\vsal}{\noalign{\vskip -7pt}}
\newcommand{\ssa}{\noalign{\vskip -2pt}}
\newcommand{\ssal}{\noalign{\vskip -4pt}}
\newcommand{\ssod}{\noalign{\vskip -1pt}}

\begin{document}
\title[ B-orbits of square zero]{$\B-$orbits of square zero in nilradical of the symplectic algebra}
\author[N. Barnea]{Nurit Barnea}
\thanks{N. Barnea: partially supported by Israel Scientific Foundation grant 797/14}
\address{Department of Mathematics,
University of Haifa, Haifa 31905, Israel}
\email{barnea.nurit@gmail.com}
\author[A. Melnikov]{Anna Melnikov}
\address{Department of Mathematics,
University of Haifa, Haifa 31905, Israel}
\email{melnikov@math.haifa.ac.il}

%-------- abstract --------
\begin{abstract}
Let $SP_n(\Co)$ be the symplectic group and $\mathfrak{sp}_n(\Co)$ its Lie algebra. 
Let $\B$ be a Borel subgroup of $SP_n(\Co)$, $\mathfrak{b}={\rm Lie}(\B)$ and 
$\mathfrak n$ its nilradical. Let $\mathcal X$ be a subvariety of elements of square 0 in $\mathfrak n.$ 
$\B$ acts adjointly on $\mathcal X$.
In this paper we describe topology of orbits $\mathcal X/\B$ in terms of symmetric link patterns.

Further we apply this description to the computations of the
closures of orbital varieties of nilpotency order 2 and to their intersections. 
In particular we show that all the intersections of codimension 1 are irreducible.
\end{abstract}
\maketitle
%------- Introduction --------
\section{Introduction}
\subsection{Nilpotent orbits in a semi-simple Lie algebra}\label{1.1}
Let $G$ be a connected, reductive, linear algebraic group
over $\mathbb{C}$ and let $\mathfrak{g}$ be its Lie algebra.
The group $G$ acts on $\mathfrak{g}$ via the adjoint action
$(g,x)\mapsto g.x=gxg^{-1}$. For $x\in \mathfrak g$ let 
$\mathcal O_x:=G.x$ denote the orbit of $x$ under this action.

Fix a Borel subgroup $\B\subset G$ and let $\mathfrak b$ and
$\mathfrak{n}\subset \mathfrak{b}$
be the corresponding Borel subalgebra and its nilradical.

An adjoint orbit ${\mathcal O}\subset\mathfrak{g}$ is called {\em nilpotent}
if ${\mathcal O}\cap\mathfrak{n}\not=\emptyset$. Obviously, ${\mathcal O}$ is nilpotent iff 
$\mathcal O=\mathcal O_x$ for some $x\in \mathfrak n.$ 

Nilpotent orbits in a semi-simple Lie algebras and their topology are defined completely by the 
nilpotent orbits in the corresponding simple Lie algebras. So from now on we speak only about simple classical Lie algebras. 

For simple Lie algebras of types $A_n,\ B_n,\ C_n$ and $D_n$
a nilpotent orbit $\mathcal O_x$ is defined by Jordan form of $x$ (apart from very even cases in $D_n$). 
In turn, since $x$ is nilpotent so that 0 is its unique eigenvalue, Jordan form of $x$ can be  regarded as a list of Jordan block lengths, 
written in a non-increasing order, which in turn can be regarded as a partition of $n+1$ in case of $A_n$ or as an admissible partition of $2n+1$ 
in case of $B_n$ or of $2n$ in cases of $C_n$ and $D_n.$ In such a way we get a bijection between nilpotent orbits and corresponding set of partitions. 
Gerstenhaber (\cite{Ge}) provided the expression for the dimension of a nilpotent orbit in terms of partitions and defined the partial order on partitions corresponding to the inclusion relations on nilpotent orbit closures in case of $A_n$. 
Further Kraft and Procesi (\cite{Kr-Pr}) provided the expression for the dimension of a nilpotent orbit in terms of admissible partitions in cases $B_n,\ C_n,\ D_n$ and showed that the restriction of the order on partitions
mentioned above on $2n+1$ (resp. $2n$) to the admissible ones defines inclusions of nilpotent orbit closures.

\subsection{$\B-$orbits in $\mathfrak n$ and spherical orbits}\label{1.2} Consider an adjoint action of $\B$ on $\mathfrak n.$ 
For $x\in \mathfrak n$ let $\mathcal B_x:=\B.x$. Obviously, $\mathcal B_x\subset\mathfrak n$, however, there is no  elegant theory of 
$\B-$orbits in this case in the spirit of the theory of nilpotent orbits from \ref{1.1}. 
The first obstacle is that the number of $\B-$orbits is infinite.

A $G-$orbit $\Or$ in $\mathfrak g$ is called {\em spherical} if
it admits a dense $\B-$orbit. By the result of Brion and Vinberg \cite{Br,Vi} $\Or$ is spherical if and only if 
$\mathcal O$ is a union of finite number of $\B-$orbits. 
Panyushev \cite{Pa1} classified all the spherical nilpotent orbits in $A_n,\ B_n,\ C_n$ and $D_n.$

The description of $\mathcal B_x\subset \Or$ for a spherical nilpotent $\Or$ is much simpler then for a generic nilpotent $\mathcal O$. 
In particular, the description of $\B-$orbits in
$\Or\cap\mathfrak n$ for a spherical orbit is relatively simple. 
The aim of this paper is to describe such $\B-$orbits for type $C_n.$
  
\subsection{$\B-$orbits of square zero in $\mathfrak n$ for $\mathfrak{sl}_n(\Co)$ and link patterns}\label{1.3} 
Let us first recall the results for type $A_n$.
	
In $\mathfrak{sl}_n(\Co)$ by Panyushev $\Or_x$ is spherical iff $x^2=0.$ 
In this case the description of the topology of $\B-$orbits
in $\mathcal O_x\cap\nil$ is made in terms of link patterns (\cite{Me1, Me2}). 
To explain it we need to introduce more notation.

Let $\B_{SL_n}$ be a fixed Borel subgroup of $SL_n$ and 
let $\mathfrak b_{\Sl_n}={\rm Lie} (\B_{SL_n})=\mathfrak h_{\Sl_n}\oplus\mathfrak n_{\Sl_n}$.

Recall that the root system, positive roots and base of simple roots  in $\Sl_n$ (with respect to $\mathfrak h_{\Sl_n}^*$) are defined in a standard way as
\begin{itemize}
\item{} $R:=\{e_j-e_i\}_{1\leq i\ne j\leq n}$;
\item{} $R^+:=\{e_j-e_i\}_{1\leq i<j\leq n}$;
\item{} $\Pi:=\{e_{i+1}-e_i\}_{1\leq i\leq n-1}$.
\end{itemize}
For $1\leq i\ne j\leq n$ put $Y_{(i,j)}:=Y_{e_j-e_i}$ to be the corresponding root vector of $\Sl_n$ so that 
$\mathfrak n_{\Sl_n}=\bigoplus\limits_{1\leq i<j\leq n}{\Co} Y_{(i,j)}$.

The roots $\alpha,\beta\in R$ are called {\it strongly orthogonal} if $\alpha\pm \beta\not\in R$. 
In $\Sl_n$ the roots $e_j-e_i,\ e_l-e_k$ are strongly orthogonal iff $\{i,j\}\cap\{k,l\}=\emptyset$.

We call a subset $\{\alpha_s\}_{s=1}^t\subset R^+$ {\em strongly orthogonal} if for any 
$1\leq s_1<s_2\leq t$ the roots $\alpha_{s_1},\alpha_{s_2}$ are strongly orthogonal.

A {\it link pattern} $L$ on $n$ points with $k$ arcs is an unoriented graph with  $n$
linearly ordered vertices labelled by integers from $1$ to $n$ (from left to right)
and $k$ edges connecting $k$ disjoint pairs of points $(i_1,j_1),\ldots,(i_k,j_k)$ (where $i_s<j_s$).
We draw the edges as arcs over the line of points. Put $l(L)$ to be the number of arcs in $L.$ 
Put ${\bf LP}_n$ to be the set of all link patterns on $n$ points and
${\bf LP}_n(k):=\{L\in {\bf LP}_n\ :\ l(L)=k\}$. 

Given a strongly orthogonal subset $S=\{e_{j_s}-e_{i_s}\}_{s=1}^k\in R^+$ we attach to it a link pattern $L_S=(i_1,j_1)\ldots(i_k,j_k)\in {\bf LP}_n(k)$. 
Obviously, in such a way we get a bijection between the sets of strongly orthogonal roots in $R^+$ and ${\bf LP}_n.$ 

For $L=(i_1,j_1)\ldots (i_k,j_k)\in {\bf LP}_n(k)$ let
$Y_L:=\sum\limits_{s=1}^k Y_{(i_s,j_s)}$ denote the corresponding sum of root vectors and let $\mathfrak B(Y_L)$ denote $\B_{SL_n}-$orbit of $Y_L.$

We define a (rank) matrix $R_L$ for $L\in{\bf LP}_n$ as follows. 
Set $(s,t)\in L$ if $(s,t)$ is an arc of $L$. Let $|A|$ denote the cardinality of a set $A$. 
For $i,j\:\ 1\leq i<j\leq n$ put
$$l(L,i,j):=|\{(s,t)\in L\ :\ i\leq s<t\leq j\}|.$$

$R_L$ is an upper triangular integer $n\times n$ matrix defined by:
$$(R_L)_{i,j}:=\left\{\begin{array}{ll}
							l(L,i,j)	& {\rm if}\ \ 1\leq i< j\leq n;\\
							0 				& {\rm otherwise.}\\
\end{array}\right.$$

Define a partial order on $M_{n\times n}(\mathbb{R})$ as follows. 
For $A,B\in M_{n\times n}(\mathbb{R})$ put $A\preceq B$ if for any $1\leq i,j\leq n$ one has $(A)_{i,j}\leq (B)_{i,j}.$

This order induces a partial order on ${\bf LP}_n$ by putting $L'\preceq L$ for $L,L'\in \bf{LP}_n$ 
if $R_{L'}\preceq R_L.$

Let $\mathcal Y_n^{(2)}:=\{y\in \mathfrak n_{\Sl_n}\ :\ y^2=0\}$ denote  the subvariety of elements of square zero in $\mathfrak n_{\Sl_n}$. For a variety $\mathcal A\subset \mathfrak g $ let $\overline {\mathcal A}$ denote its Zariski closure in $\mathfrak g$.
By \cite{Me2} one has
\begin{thm}
	(i) $\mathcal Y_n^{(2)}=\bigsqcup\limits_{L\in {\bf LP}_n}\mathfrak B(Y_L)$.\\
	(ii) For $L\in{\bf LP}_n$ one has $\overline{\mathfrak B}(Y_L)=\bigsqcup\limits_{L'\preceq L}\mathfrak B(Y_{L'})$.
\end{thm}

\subsection{$\B-$orbits of square zero in $\mathfrak n$ for $\mathfrak{sp}_n(\Co)$ and symmetric link patterns}\label{1.4}
In $\mathfrak{sp}_n(\Co)$ by Panyushev a nilpotent orbit $\mathcal O_x$ is spherical iff $x^2=0$ exactly as in type $A_n.$ 
In this paper we describe $\B-$orbits of $\mathcal O_x\cap\nil$ in this case in terms of symmetric link patterns.
	
Consider $SP_{2n}(\mathbb{C})$ -- the symplectic group of degree
$2n$ over $\mathbb{C}$, and let $\G=\Sp_{2n}$ be its Lie algebra
with the (standard) Cartan subalgebra $\gh$.
We define roots system, positive roots and a base
of simple roots, respectively, in (one of two) standard ways, namely:
\begin{itemize}
\item{} $\Phi:=\left\{\pm e_i\pm e_j,\ \pm 2e_k\ |\ 1\leq i,j,k \leq n,\ i\neq j\right\}$;
\item{} $\Phi^+:=\left\{e_j\pm e_i,\ 2e_k\ |\ 1\leq i<j \leq n,\ 1\leq k \leq n\right\}$;
\item{} $\Delta:=\left\{e_{i+1}-e_{i},\ 2e_1\ |\ 1\leq i<n\ \right\}$.
\end{itemize}

Recall that a pair of roots $\{\pm(e_i\pm e_j),\ \pm(e_k\pm e_l)\}$ 
(resp. $\{\pm(e_i\pm e_j),\pm 2e_k\}$ or $\{\pm 2e_i,\pm 2e_k\}$)
is strongly orthogonal iff $\{i,j\}\cap \{k,l\}=\emptyset$ (resp. $k\not\in\{i,j\}$ or $i\ne k$).

Given a strongly orthogonal subset $S=\{\alpha_s\}_{s=1}^k\subset \Phi^+$ we associate to it a symmetric link pattern 
$L_S$ on $2n$ points $-n,\ldots,-1,1\ldots,\ldots,n$
by $(-i_s,i_s)\in L_S$ if $\alpha_s=2e_{i_s}$;  $(i_s,j_s),(-i_s,-j_s)\in L_S$ if $\alpha_s=e_{j_s}-e_{i_s}$ and 
$(-i_s,j_s),(i_s,-j_s)\in L_S$ if $\alpha_s=e_{j_s}+e_{i_s}$. We put $\alpha\in L$ if the corresponding arcs are in $L.$ 
In such a way we get a bijection between the subsets of strongly orthogonal roots in $\Phi^+$ and the symmetric link patterns on $2n$ points. 
Put ${\bf SLP}_{2n}$ to be the set of all symmetric link patterns on points $-n,\ldots,-1,1,\ldots,n$ 
and ${\bf SLP}_{2n}(k)$ to be the subset of the symmetric link patterns with $k$ arcs.
  
The corresponding Borel subgroup (resp. subalgebra) is denoted by $\B_{2n}$ (resp. $\bo_{2n}$). 
For $\alpha\in\Phi$ let $X_\alpha$ denote its root vector, so that the nilradical
%The Borel subgroup is denoted by $\B_{2n}$. Its corresponding Borel subalgebra is denoted by $\bo_{2n}$,
$\nil_{2n}=\bigoplus\limits_{\alpha\in\Phi^+}{\mathbb C}X_\alpha$. 

Given a strongly orthogonal subset $S\subset \Phi^+$ set $X_S=\sum\limits_{\alpha\in S}X_\alpha$. 
Equivalently, for 
$L=\{(i_s,\pm j_s),(-i_s,\mp j_s)\}_{s=1}^m\cup\{(-k_t,k_t)\}_{t=1}^l\in {\bf SLP}_n$ let
$$X_L:=\sum\limits_{s=1}^m X_{e_{j_s}\mp e_{i_s}}+\sum\limits_{s=1}^lX_{2e_{k_s}}$$
denote the corresponding sum of root vectors.

For example for $L=(1,2)(-1,-2)(3,-6)(-3,6)(-4,4)\in {\bf SLP}_{12}(5)$ 
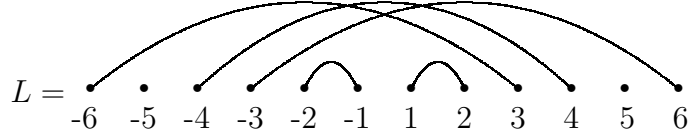
\begin{center}
\begin{picture}(130,60)(20,20)
\put(-40,35){$L=$}
 \multiput(-10,40)(20,0){12} {\circle*{3}} % points on
 \put(-17,25){-6}
 \put(5,25){-5}
 \put(25,25){-4}
 \put(45,25){-3}
 \put(65,25){-2}
 \put(85,25){-1}
 \put(107,25){1}
 \put(127,25){2}
 \put(147,25){3}
 \put(167,25){4}
 \put(188,25){5}
 \put(208,25){6}
 \qbezier(-10,40)(70,105)(150,40)
 \qbezier(50,40)(130,105)(210,40)
 %\qbezier(30,40)(50,70)(70,40)
 \qbezier(30,40)(100,105)(170,40)
 \qbezier(70,40)(80,60)(90,40)
 \qbezier(110,40)(120,60)(130,40)
\end{picture}
\end{center}
one has $X_L=X_{e_2-e_1}+X_{e_6+e_3}+X_{2e_4}$.
  
Let $\mathcal X_{2n}^{(2)}:=\{x\in\nil_{2n}\ :\ x^2=0\}$ denote the subvariety of elements of square zero in $\nil_{2n}$.
For $L\in {\bf SLP}_{2n}$ let $\mathcal B_L:=\B_{2n}.X_L.$  
As we show in \ref{4.1}
$$\mathcal X_{2n}^{(2)}=\bigsqcup\limits_{L\in {\bf SLP}_{2n}}\mathcal B_L$$

Further in \ref{4.2} we provide the expression for $\dim \mathcal B_L$ in terms of $L.$

In Section \ref{5} we show that the inclusion of $\B-$orbit closures in $\mathcal X_{2n}^{(2)}$ is defined by the restriction of order
$\prec$ on ${\bf LP}_{2n}$ to the subset ${\bf SLP}_{2n}$, namely
$$\overline{\mathcal B}_L=\bigsqcup\limits_{\genfrac{}{}{0pt}{}{L'\in {\bf SLP}_{2n}}{L'\preceq L}}\mathcal B_{L'}$$

\subsection{Orbital varieties of square zero in $C_n$}\label{1.5} 
Further in Section \ref{6} we apply the results on 
$\{\mathcal B_L\}_{L\in{\bf SLP}_{2n}}$ to orbital varieties of square zero in $C_n$. 

Let $\mathcal O$ be a nilpotent orbit in a semi-simple $\mathfrak g$. 
Then the intersection ${\mathcal O}\cap\mathfrak{n}$ is
a quasi-affine algebraic variety, in general reducible
and by Spaltenstein \cite{Sp},
it is equidimensional of dimension $\frac{1}{2}\dim{\mathcal O}$.
The irreducible components of ${\mathcal O}\cap \mathfrak{n}$
are called {\em orbital varieties}. 

In $\Sp_{2n}$ an orbital variety $\mathcal V\subset\mathcal O_x\cap\nil_{2n}$ is labeled by a standard domino tableaux $T$ of shape 
$\lambda\vdash 2n$
where $\lambda$ is the corresponding list of the blocks in Jordan form of $x$ (cf. \ref{a2}). We denote it by $\mathcal V_T$.

If $\mathcal O_x$ is a nilpotent orbit of square zero, then $\mathcal V\in \mathcal O_x\cap\nil_{2n}$ is a union of finite number of $\B_{2n}-$orbits so that $\mathcal V$ admits a dense $\B_{2n}-$orbit. Moreover, $\mathcal B_L$ is dense in some orbital variety if and only if $\dim \mathcal B_L=\frac{1}{2}\dim\mathcal O_{X_L}.$ 
Such $\mathcal B_L$ are called {\it maximal} in $\{\mathcal B_{L'}\}_{L'\in {\bf SLP}_{2n}(k)}$ 
(where $k$ is the number of arcs in $L.$) We define this byjection in \ref{a3}. 

Further we consider the first two questions on orbital varieties of square zero:
\begin{itemize}
\item[(i)] Description of an orbital variety closure. Let $\mathcal V_T$ be an orbital variety in 
$\mathcal O_x$ and let $\mathcal O'\subset \overline{\mathcal O}_x$. What are the components of 
$\overline{\mathcal V}_T\cap\mathcal O'$? How to define whether $\mathcal V_S\subset \overline{\mathcal V}_T$? We answer both these questions in \ref{a5}. We show that in general $\overline{\mathcal V}_T\cap\mathcal O'$ contains components of different dimensions and in general it can happen that there are no orbital varieties in the intersection. This picture is very different from the closure of an orbital variety of square 0 in $\Sl_n(\Co)$, where the closure of an orbital variety is a union of orbital varieties only. This means that although the results on $\B-$orbits of square zero in types $A_n$ and $C_n$ look very similar formally, the general pictures are very different.  
\item[(ii)] Another important question (it can be translated straightforwardly to the components of a Springer fiber) is whether the intersection of two orbital varieties of square zero of codimension 1 is always irreducible? Here the answer provided in \ref{a6} is positive, exactly as in $A_n$ case.
\end{itemize}
In a subsequent paper we use the results on $\B_{2n}-$orbits in order to prove a Panushev's conjecture \cite{Pa2} in types $B_n,C_n$ and $D_n$.

We also use them in another subsequent paper in order to classify orbital varieties of square 0 in $C_n$ according to their singularity.

\subsection{The structure of the paper}
In Section 2 we provide all the preliminaries and notation essential in what follows. 
In Section 3 we provide the classification of $\B_{2n}-$orbits of square zero in $\nil_{2n}$ and give the expression for their dimensions.

Section 4 is devoted to the construction of $\mathcal B_L$ closure. The results here rely heavily on $A_n$ case and are combinatorially involved. However, all the computations are similar. So we provide in detail only the computation in the most involved cases. 
Further, in Appendix A we provide a short but full construction of the boundary of $\mathcal B_L$.
Finally, in Section 5 we apply the results to orbital varieties. The reader can find the index of  notation used intesively in the paper at the end of the paper.
   
%------- Preliminaries and notation --------
\section{Preliminaries and notation}
\subsection{Root vectors}\label{2.3}
Recall the standard root system $\Phi$ of $C_n$ defined in \ref{1.4}.
We consider a standard idetification of root vectors, namely:
\begin{itemize}
\item{} $X_{e_j-e_i}:=E_{n+1-j,n+1-i}-E_{2n+1-i,2n+1-j}$
\item{} $X_{e_j+e_i}:=E_{n+1-i,2n+1-j}+E_{n+1-j,2n+1-i}$
\item{} $X_{-e_j-e_i}:=E_{2n+1-j,n+1-i}+E_{2n+1-i,n+1-j}$
\item{} $X_{2e_i}:=E_{n+1-i,2n+1-i}$
\item{} $X_{-2e_i}:=E_{2n+1-i,n+1-i}$
\end{itemize}
where $E_{i,j}$ is the elementary matrix having a "1" as its
$(i,j)$-entry and zeros elsewhere.

The corresponding Borel subgroup is a semiprime product $\B_{2n}={\bf T}_{2n}\ltimes {\bf U}_{2n}$ of the (torus) subgroup ${\bf T}_{2n}$ 
of diagonal matrices and the subgroup  ${\bf U}_{2n}$ of unipotent matrices.
Put $T_i(a)$ to be a diagonal matrix with $a \in \mathbb{C}$ on place $n+1-i$, $a^{-1}$ on place $2n+1-i$ and 1 on the other places of the main diagonal.
Put $U_{\alpha}(a):=I_{2n}+aX_{\alpha}$ for $\alpha\in\Phi^+$. 
One has ${\bf T}_{2n}=\langle T_i(a)\ : 1\leq i\leq n,\ a\in \Co \rangle$ and ${\bf U}_{2n}=\langle U_{\alpha}(a)\ :\alpha\in\Phi^+,\ a\in \Co\rangle.$ 
%Of course in both cases this is not the minimal set of generators.

For future references we provide the table of actions of these generators on positive root vectors:\\
For $i,j: 1\leq i<j\leq n$ and $l,k: 1\leq k<l\leq n$ one has
$$ T_i(a). X_{e_l-e_k}=\left\{\begin{array}{ll}aX_{e_l-e_k}&{\rm if}\ i=l;\\
									      a^{-1}X_{e_l-e_k}&{\rm if}\ i=k;\\
												X_{e_l-e_k}&{\rm otherwise.}\\
												\end{array}\right. \,  \
 T_i(a). X_{e_l+e_k}=\left\{\begin{array}{ll}aX_{e_l+e_k}&{\rm if}\ i\in\{k,l\};\\
																						 X_{e_l+e_k}&{\rm otherwise.}\\
																						 \end{array}\right.$$
$$ T_i(a). X_{2e_j}=\left\{\begin{array}{ll}a^2X_{2e_j}&{\rm if}\ i=j;\\
																						X_{2e_j}&{\rm otherwise.}\\
																						\end{array}\right.$$
For $i,j : 1\leq i<j\leq n$; $k,l: 1\leq k<l\leq n$ and $m: 1\leq m\leq n$ one has: 
\begin{itemize}	
\item[]$U_{e_j-e_i}(a).X_{e_l-e_k}=\left\{\begin{array}{ll}X_{e_l-e_k}+aX_{e_j-e_k}&{\rm if}\ i=l;\\
																													 X_{e_l-e_k}-aX_{e_l-e_i}&{\rm if}\ j=k;\\
																													 X_{e_l-e_k}&{\rm otherwise}.\\
																													\end{array}\right.$
\item[]$U_{e_j-e_i}(a).X_{e_l+e_k}=\left\{\begin{array}{ll}X_{e_l+e_k}+aX_{e_j+e_k}&{\rm if}\ i=l\ \&\ j\neq k;\\
																													 X_{e_l+e_k}+aX_{e_j+e_l}&{\rm if}\ i=k\ \&\ j\neq l;\\
																													 X_{e_l+e_k}+2aX_{2e_l}&{\rm if}\ \left\{i,j\right\}=\left\{k,l\right\};\\
																													 X_{e_l+e_k}&{\rm otherwise}.\\
																													\end{array}\right.$
\item[]$U_{e_j-e_i}(a).X_{2e_m}=\left\{\begin{array}{ll}X_{2e_m}+aX_{e_j+e_m}+a^2X_{2e_j}&{\rm if}\ i=m;\\
																												X_{2e_m}&{\rm otherwise}.\\
																												\end{array}\right.$\\
\item[]$U_{e_j+e_i}(a).X_{e_l-e_k}=\left\{\begin{array}{ll}X_{e_l-e_k}-aX_{e_l+e_i}&{\rm if}\ j=k\ \&\ i\neq l;\\
																													 X_{e_l-e_k}-aX_{e_l+e_j}&{\rm if}\ i=k\ \&\ j\neq l;\\
																													 X_{e_l-e_k}-2aX_{2e_l}&{\rm if}\ \left\{i,j\right\}=\left\{k,l\right\};\\
																													 X_{e_l-e_k}&{\rm otherwise}.\\
																													\end{array}\right.$\\
\item[]$U_{e_j+e_i}(a).X_{e_l+e_k}=X_{e_l+e_k}; \quad U_{e_j+e_i}(a).X_{2e_m}=X_{2e_m}.$\\
%\item[] $  U_{e_i+e_j}(a).X_{2e_k}=X_{2e_k}.$\\
\item[]$U_{2e_m}(a).X_{e_l-e_k}=\left\{\begin{array}{ll}X_{e_l-e_k}-aX_{e_k+e_l}&{\rm if}\ m=k;\\
																												X_{e_l-e_k}&{\rm otherwise}.\\
																												\end{array}\right.$\\
\item[]$U_{2e_m}(a).X_{e_l+e_k}=X_{e_l+e_k};\quad U_{2e_m}(a).X_{2e_k}=X_{2e_k}.$\\
\end{itemize}

%-----------------
\subsection{Nilpotent orbits in $C_n$}\label{2.5}
Recall that a {\em partition} of a positive integer $n$ is a non-increasing sequence
$\lambda:=(\lambda_1,\geq \lambda_2\ldots,\geq \lambda_k)$ of positive integers
such that $\sum_{i=1}^k{\lambda_i}=n$. Let ${\mathcal P}(n)$ denote the set of all partitions.
Sometimes we will write partitions with multiplicities, that is $\lambda=(\lambda_1^{m_1},\ldots,\lambda_r^{m_r})$ where
$\lambda_i>\lambda_{i+1}$ and $m_i$ is its multiplicity. We will omit $m_i=1$ in these cases.

Set
${\mathcal P}_1(2n):=\{(\lambda_1^{m_1},\ldots,\lambda_r^{m_r})\in {\mathcal P}(2n)\ :\ \lambda_i\ {\rm is\ odd}\ \Rightarrow m_i\ {\rm is\ even}\}.$
Recall that the Jordan form of $x\in \nil_{2n}$ is $J(x)=\lambda$ (as a list of lengths of Jordan blocks) 
where $\lambda\in {\mathcal P}_1(2n).$ Set $\mathcal O_\lambda:=\mathcal O_x$ in this case.

Given a partition $\lambda=(\lambda_1,\ldots,\lambda_k)\in {\mathcal P}_1(2n)$, for $i\ :\ 1\leq i\leq \lambda_1$
put $r_i:=|\{j\ :\ \lambda_j=i\}|$ and $s_i=|\{j\ :\ \lambda_j\geq i\}|.$ One has (cf. \cite[6.1.4]{Co-Mc})
$$\dim(\Or_\lambda)=2n^2+n-\frac{1}{2}\sum\limits_{i=1}^{\lambda_1}s_i^2-\frac{1}{2}\sum\limits_{i\ {\rm odd}} r_i$$

To define the closure of a nilpotent orbit we have to introduce a dominance order on partitions. 
For $\lambda=(\lambda_1,\ldots,\lambda_k),\mu=(\mu_1,\ldots,\mu_m)\in {\mathcal P}(n)$ 
we put $\lambda\geq \mu$ if for any $i\ :\ 1\leq i\leq \min\{k,m\}$ one has
$\sum\limits_{j=1}^i\lambda_j\geq\sum\limits_{j=1}^i\mu_j.$ 
For ${\mathcal O}_\lambda\in \Sp_{2n}$ one has (cf. \cite[6.2.5]{Co-Mc})
$$\overline\Or_\lambda:=\coprod\limits_{\genfrac{}{}{0pt}{}{\mu\in{\mathcal P}_1(2n)}{\mu\leq \lambda}}\Or_\mu.$$

%----------- Spherical -------------
\subsection{Spherical nilpotent orbits in $C_n$}
Since in $C_n$ a nilpotent orbit is spherical if and only if it is of nilpotent order 2 we consider in what follows only $\Or_\lambda$
with $\lambda=(2^k,1^{2l})$ where $k+l=n.$ In this case $r_1=2l$ and $s_1=k+2l,\ s_2=k,$ so that
$$\dim(\Or_{(2^k,1^{2l})})=k(k+2l+1).$$

Note that a partial dominance order on partitions becomes linear when restricted to partitions of type $(2^i,1^j).$
Note also that for any $k$ one has $(2^k,1^{2n-2k})\in {\mathcal P}_1(2n)$ so that
$$\overline\Or_{(2^k,1^{2n-2k})}=\coprod\limits_{i=0}^k\Or_{(2^{k-i},1^{2(n-k+i)})}$$
%-----------
\subsection{$\B_{2n}-$orbits in a spherical orbit}\label{2.7}
Since a spherical orbit is a disjoint finite union of $\B_{2n}-$orbits, in particular,  $\Or_{(2^k,1^{2(n-k)})}\cap\nil_{2n}$ is a disjoint
finite union of $\B_{2n}-$orbits. Moreover, by \ref{1.5} for ${\mathcal B}\subset\Or_{(2^k,1^{2(n-k)})}\cap\nil_{2n}$ 
one has $\overline{\mathcal B}\cap\Or_{(2^k,1^{2(n-k)})}$ is an irreducible component of $\Or_{(2^k,1^{2(n-k)})}\cap\nil_{2n}$ 
if and only if $\dim \mathcal B=\frac{1}{2}k(k+2(n-k)+1)$. 
We will call such orbits {\it maximal} in $\Or_{(2^k,1^{2(n-k)})}\cap\nil_{2n}$.

By a result of Timashev \cite[2.13]{Tim} for any $\B-$orbit $\mathcal B$ one has $\overline{\mathcal B}\setminus \mathcal B$ 
is an equidimensional variety of codimension 1 in $\overline{\mathcal B}$.

Applying this to $\Or_{(2^k,1^{2(n-k)})}\cap\nil_{2n}$ we get that for any ${\mathcal B}\in \Or_{(2^k,1^{2(n-k)})}\cap\nil_{2n}$ 
such that $\overline{\mathcal B}\cap \Or_{(2^k,1^{2(n-k)})}\ne \mathcal B$ there exist 
$\B_{2n}-$orbits ${\mathcal B}_1,\ldots,{\mathcal B}_s\subset  \Or_{(2^k,1^{2(n-k)})}\cap\nil_{2n}$ such that
$$\left(\overline{\mathcal B}\setminus {\mathcal B}\right) \cap \Or_{(2^k,1^{2(n-k)})}=\bigcup\limits_{i=1}^s{\overline{\mathcal B}_i}\cap \Or_{(2^k,1^{2(n-k)})},\quad {\rm codim}_{\overline{\mathcal B}}{\mathcal B}_i=1$$
We will call $\mathcal{B}_i$ the {\it component of the boundary (inside $\Or_{(2^k,1^{2(n-k)})}$)} of $\mathcal B.$

\subsection{Notation on link patterns }\label{2.8}
For $L\in {\bf LP}_n$ we always write $(s,t)\in L$ where $s<t$.
For $L=(i_1,j_1),\ldots,(i_k,j_k)\in {\bf LP}_n$  the set of points $Ep(L):=\{i_1,j_1,\ldots,i_k,j_k\}$ is called {\it end points} of $L$ 
and the set of all the other points $Fp(L):=\{i\}_{i=1}^n\setminus Ep(L)$
is called {\it fixed points} of $L$.

We say that arcs $(i_s,j_s),\ (i_t,j_t)$ where $i_s<i_t$ have a {\it crossing} if
$i_s<i_t<j_s<j_t.$ We say that an arc $(i_r,j_r)$ is a {\it bridge} over a fixed point $f\in Fp(L)$ if $i_r<f<j_r.$

For $L\in {\bf SLP}_{2n}$ we change the notation a little bit. 
We write the ends of an arc either from left to right, then we denote it as $\left\langle i,j\right\rangle$ and 
$-n\leq i<j\leq n$ or as  $(\pm i,\pm j)$ where $1\leq i\leq j\leq n$. 
The symmetric arcs we write as $(-i,i).$

To simplify the notation we will denote $X_{(\pm i, j)}:=X_{e_j\mp e_i}$ for a pair of arcs\\ 
$\{(\pm i, j),\ (\mp i,-j)\}\in L$ and $X_{(-i,i)}:=X_{2e_i}.$
For $L\in{\bf SLP}_{2n}$ set $Ep^+(L)$ to be the set of its positive end points and $Fp^+(L)$ to be the set of its positive fixed points. 
Obviously the full sets $Ep(L)$ and $Fp(L)$ are obtained by adding the same points with negative sign.

\begin{rem}
Note that in the cases $A_n$ and $C_n$ there is a bijection between the set of subsets of strongly orthogonal roots in $R^+$ (or resp. $\Phi^+$) 
and the set of involutions in the corresponding Weyl groups.
\end{rem}

\subsection{$\B_{SL_{2n}}$ and $\mathfrak n_{\Sl_{2n}}$ corresponding to $\B_{2n}$ and $\mathfrak n_{2n}$}\label{3.2}
Since in all our proofs we use intensively Theorem \ref{1.3} we construct root vectors $\{Y_{e_j-e_i}\}_{1\leq i<j\leq 2n}\in \Sl_{2n}$ 
such that $X_\alpha$ for $\alpha\in \Phi^+$ is the sum of root vectors (up to signs) for corresponding (strongly orthogonal) 
roots in $\mathfrak n_{\Sl_{2n}}$ and $\B_{2n}$ is a subgroup of $\B_{SL_{2n}}$ so that 
for $L\in {\bf SLP}_{2n}$ one has $X_L=Y_L$ (up to signs of root vectors) and $\mathcal B_L\subset \mathfrak B(Y_L)$. 
To do this we put:
$$\begin{array}{rl} Y_{(i,j)}&=\left\{\begin{array}{ll} E_{i,j}&{\rm if}\ 1\leq i<j\leq n;\\
                    E_{i,3n+1-j}&{\rm if}\ i<n,\ n+1\leq j\leq 2n;\\
                    E_{3n+1-i,3n+1-j}&{\rm if}\ n<i<j\leq 2n;\\
                    \end{array}\right.\\
                    \end{array}
                    $$
so that for $1\leq i\leq j\leq n$ one has
$$X_{(\pm i,j)}=\left\{\begin{array}{ll} X_{e_j-e_i}=Y_{(n+1-j,n+1-i)}-Y_{(n+i,n+j)}  &{\rm if}\ (\pm i,j)=(i,j); \\
                                     X_{2e_i}=Y_{(n+1-i,n+i)}&{\rm if}\ (\pm i,j)=(-i,i);\\
                                     X_{e_j+e_i}=Y_{(n+1-i,n-j)}+Y_{(n+1+j,n+i)}&{\rm if}\ (\pm i,j)=(-i,j);\\
\end{array}\right.$$

Note that to read $Y_L$ from $L\in {\bf SLP}_{2n}$ we must move from $\{-n,\ldots,-1,1,\ldots,n\}$ to $\{1,\ldots,2n\}$ 
and then $X_L=Y_L$ (up to signs of root vectors) as matrices.

\subsection{Projection $\pi$ from $C_n$ to $C_{n-1}$}\label{2.9a} In our proofs we use intensively $\pi: {\bf B}_{2n}\rightarrow{\bf B}_{2(n-1)}$ and
$\pi:\nil_{2n}\rightarrow\nil_{2(n-1)}.$ Let us define it. One can regard $X_{(\pm i,j)}$ for $j<n$ as a root vector of $\nil_{2(n-1)}.$ Respectively,
changing $I_{2n}$ to $I_{2(n-1)}$ one can regard $U_{e_j\pm e_i}$, for $1\leq i<j<n$, and $U_{2e_j}$, for $j\neq n$, as elements of ${\bf U}_{2(n-1)}.$
Note also that $T_i(a)$ for $i\ne n$ can be regarded as
an element of ${\bf T}_{2(n-1)}$ simply by changing $n$ to $n-1$ in the definition.
We put
$$\pi(X_{(\pm i,j)})=\left\{\begin{array}{ll} X_{(\pm i,j)} &{\rm if}\ j\ne n;\\
                              0&{\rm otherwise};\\
                              \end{array}\right.$$
This defines the projection $\pi:\nil_{2n}\rightarrow \nil_{2(n-1)}$.

Respectively put
$$\begin{array}{ll}\pi(U_{\alpha}(a))&=\left\{\begin{array}{ll} U_{\alpha}(a) &{\rm if}\ \alpha= e_j\pm e_i \ :\ 1\leq i< j<n\ {\rm or}\ \alpha=2e_i\ :\ i<n;\\
																			I_{2(n-1)}&{\rm otherwise};\\
																			\end{array}\right.\\
 \pi(T_i(a))&=\left\{\begin{array}{ll} T_i(a)&{\rm if}\ i<n;\\
                                       I_{2(n-1)}&{\rm otherwise};\\
																			\end{array}\right.\\
 \end{array}$$
Taking into account that 
$$\pi(U_\alpha(a) U_\beta(b))=\pi(U_\alpha(a))\pi(U_\beta(b)),\ \pi(T_i(a) T_j(b))=\pi(T_i(a))\pi(T_i(b))$$ 
we get that $\pi$ is well defined on both ${\bf T}_{2n}$ and ${\bf U}_{2n}.$ 
Further since $\pi (TU)=\pi(T)\pi(U)$ for $T\in {\bf T}_{2n}$, $U\in {\bf U}_{2n}$
it is well defined on ${\bf B}_{2n}$.

Exactly in the same way $\pi(A.X)=\pi(A).\pi(X)$ for $A\in {\bf B}_{2n},\ X\in\nil_{2n}$ so that the projection of ${\bf B}_{2n}-$orbit of $X\in \nil_{2n}$ is
${\bf B}_{2(n-1)}-$orbit of $\pi(X)$. In particular $\pi({\mathcal B}_L)$ is well defined. 
We define correspondingly $\pi:{\bf SLP}_{2n}\rightarrow {\bf SLP}_{2(n-1)}$ by
$$\pi(\left\langle  i,j\right\rangle)=\left\{\begin{array}{ll}\left\langle  i,j\right\rangle&{\rm if}\ -n<i< j<n;\\
                                  \emptyset&{\rm otherwise};\\
                                  \end{array}\right.$$
We get $\pi({\mathcal B}_L)=\widehat {\mathcal B}_{\pi(L)}$ where $\widehat{\mathcal B}$ is an orbit in $\nil_{2(n-1)}.$

Note also that in the way described above we can consider $A\in {\bf B}_{2(n-1)}$ as an element of ${\bf B}_{2n}$,
$X\in\nil_{2(n-1)}$ as an element of $\nil_{2n}$ and $L\in{\bf SLP}_{2(n-1)}$ as an element of ${\bf SLP}_{2n}.$

%------- Nilpotent $\B-$orbits of square 0 classification --------
\section{Nilpotent $\B-$orbits of square 0 classification}
All the $\B_{2n}-$orbits of square 0 in $\nil_{2n}$ are parameterized by
elements of ${\bf SLP}_{2n}$. In this section we provide this bijection and
compute the dimension of $\B_{2n}-$orbit of square 0.
\subsection{$\B_{2n}$-orbits of square $0$ in $\nil_{2n}$ }\label{4.1}
One has
\begin{thm}
$\mathcal X_{2n}^{(2)}=\bigsqcup\limits_{L\in {\bf SLP}_{2n}}{\mathcal B}_L$.
\end{thm}
\begin{proof}
For $L\ne L'\in {\bf SLP}_{2n}$ one has $\mathcal B_L\subset \mathfrak B(Y_L)$ and $\mathcal B_{L'}\subset \mathfrak B(Y_{L'})$ 
as we have explained in \ref{3.2}, so that Theorem \ref{1.3} provides $\mathcal B_L\cap\mathcal B_{L'}=\emptyset$ 
and $\bigsqcup\limits_{L\in {\bf SLP}_{2n}}{\mathcal B}_L\subset \mathcal X_{2n}^{(2)}$.

It remains to show that for $X\in \mathcal X_{2n}^{(2)}$ there exists $L\in {\bf SLP}_{2n}$ such that $X\in \mathcal B_L.$
The proof is by induction on $n$. It holds trivially for $n=1$.
Assume it holds for $n-1$ and show for $n.$

Obviously, $\pi:\mathcal X_{2n}^{(2)}\rightarrow\mathcal X_{2(n-1)}^{(2)}$. Given $X \in \mathcal X_{2n}^{(2)}$,
by the induction hypothesis there exists $\widehat L\in {\bf SLP}_{2(n-1)}$
such that $\pi(X) \in \widehat{\mathcal B}_{\widehat L}$. In other words, $\exists \widehat{M} \in \B_{2(n-1)}$ 
such that $\widehat{M}.\pi(X) = \widehat X_{\widehat L}$.

Let $\widehat L=(\pm i_1,\pm j_1)\ldots(\pm i_p,\pm j_p)(\pm l_1,\mp m_1)\ldots(\pm l_q,\mp m_q)(-k_1,k_1)\ldots(-k_r,k_r)$ so 
that $Ep^+(\widehat L)=\{i_s,j_s\}_{s=1}^p\sqcup\{l_s,m_s\}_{s=1}^q\sqcup\{k_s\}_{s=1}^r$.

Let $M$ be $\widehat M$ considered as an element of $\B_{2n}$ and $X_{\widehat L}$ be $\widehat X_{\widehat L}$ considered as an element of $\nil_{2n}.$ 
One has
$$M.X = X_{\widehat L} + \sum_{i=1}^{n-1}f_i X_{(i,n)} + \sum_{i=1}^{n-1}g_i X_{(-i,n)} + h X_{(-n,n)}.$$

Taking into account that $(M.X)^2=0$ and the structure of $X_{\widehat L}$ we get $f_{j_s}=0$ for $1\leq s\leq p$, 
$f_{m_s}=f_{l_s}=0$ for $1\leq s\leq q$
and $f_{k_s}=0$ for $1\leq s\leq r$. We also get $g_{j_s}=0$ for $1\leq s\leq p$.

$$\begin{array}{l}M':=\prod\limits_{s=1}^p U_{e_n-e_{j_s}}(-f_{i_s})U_{e_n+e_{j_s}}(g_{i_s})\prod\limits_{s=1}^qU_{e_n-e_{m_s}}(-g_{l_s})U_{e_n-e_{l_s}}(-g_{m_s})\prod\limits_{s=1}^rU_{e_n-e_{k_s}}(-g_{k_s})\\

= I_{2n} + \sum\limits_{s=1}^p \left(-f_{i_s} X_{(j_s,n)}
							+ g_{i_s} X_{(-j_s,n)} \right)
							- \sum\limits_{s=1}^q \left( g_{l_s} X_{(m_s,n)}
							+ g_{m_s} X_{(l_s,n)} \right)
							- \sum\limits_{s=1}^r g_{k_s} X_{(k_s,n)}\\
\end{array}$$
one has
$$X':= M'.(M.X) =X_{\widehat L}+\sum\limits_{i\in Fp^+(\widehat L)}f_i X_{(i,n)} + \sum\limits_{i\in Fp^+(\widehat L)}g_i X_{(-i,n)} + h' X_{(-n,n)}$$

To complete the proof we have to consider 4 cases:
\begin{itemize}
\item[(i)] If $f_i=g_i=h'=0$ for all $1\leq i\leq n$ then obviously, $X'=X_{\widehat L}$ so that $X\in \mathcal B_{\widehat L}$.
\item[(ii)] If $f_i=g_i=0$ and $h'\ne 0$ then  $T_n({h'}^{-0.5}).X'=X_{\widehat L}+X_{(-n,n)}=X_L$ where $L=(n,-n)\widehat L$ so that
$X\in \mathcal B_L.$
\item[(iii)] If $f_i=0$ for all $1\leq i\leq n$ and there exists $g_s\ne 0$ then let\\
 $i:=\min\{ s\ :\ g_s\ne 0\}$. 
One has 
$$X^{\prime\prime}:=T_i(g_i^{-1}). X'=X_{\widehat L}+X_{(-i,n)}+\sum\limits_{j=i+1}^n g_j X_{(-j,n)}+h'X_{(-n,n)}.$$
Take
$$M^{\prime\prime}:=U_{e_n-e_i}(-0.5 h')\,\prod\limits_{j=i+1}^{n-1}U_{e_j-e_i}(-g_j)= I_{2n}-\sum\limits_{j=i+1}^{n-1} g_jX_{(i,j)}-0.5 h'X_{(i,n)}.$$
Recall that $g_j\ne 0$ implies $j\in Fp^+(L)$. Thus,
$M^{\prime\prime}.X^{\prime\prime}=X_{\widehat L}+X_{(-i,n)}=X_L$ where $L=(-i,n)(i,-n)\widehat L$ so that
$X\in \mathcal B_L.$
\item[(iv)] Finally, assume $f_s\ne 0$ for some $s$ and let $j:=\max\{s\ :\ f_s\ne 0\}$. One has
$$X^{\prime\prime}:=T_j(f_j^{-1}).X'=
X_{\widehat L}+X_{(j,n)}+\sum\limits_{i=1}^{j-1} f_i X_{(i,n)}+\sum\limits_{i=1,\ i\ne j}^n g_iX_{(-i,n)}+g_jf_j^{-1}X_{(-j,n)}+
h'X_{(-n,n)}.$$

Further we proceed in a similar way, however the computations are a little bit more complex. We make them in two steps. First, take
$$M^{\prime\prime}:=U_{e_n+e_j}(0.5 h')\cdot\prod\limits_{i=1}^{j-1}U_{e_j-e_i}(f_i)=I_{2n}+\sum\limits_{i=1}^{j-1} f_iX_{(i,j)}+0.5 h'X_{(-j,n)}.$$
One has
$$X^{\prime\prime\prime}:=M^{\prime\prime}.X^{\prime\prime}=X_{\widehat L}+X_{(j,n)}+\sum\limits_{i=1, i\ne j}^n g_iX_{(-i,n)}$$
since the coefficient of $X_{(-j,n)}$ in $X^{\prime\prime\prime}$ (which is equal to $r=\sum\limits_{i=1}^{j-1} f_ig_i +g_jf_j^{-1}$) 
must be zero since $(X^{\prime\prime\prime})^2=rX_{(-n,n)}=0$.\\
At the last step take
$$M^{\prime\prime\prime}:=\prod\limits_{i=1,i\ne j}^{n-1} U_{e_i+e_j}(g_i)=I_{2n}+\sum\limits_{i=1}^{j-1}g_iX_{(-i,j)}+\sum\limits_{i=j+1}^{n-1}g_iX_{(-j,i)}$$
and  get $M^{\prime\prime\prime}.X^{\prime\prime\prime}=X_{\widehat L}+X_{(j,n)}=X_L$ where $L=(j,n)(-j,-n)\widehat L$ so that $X\in \mathcal B_L.$
\end{itemize}
\end{proof}

%Given $L\in {\bf SLP}_{2n}$ we put $l(L)$ to be the number of arcs in it and exactly as in ${\bf LP}_n$, we put
%${\bf SLP}_{2n}(k)$ to be the subset of symmetric link patterns on $2n$ points with $k$ arcs.

By the construction of $X_L$ we get $Rank X_L=l(L)$. In turn the rank of a nilpotent matrix in ${\mathcal X}_{2n}^{(2)}$ 
is equal to the number of $2-$blocks in its Jordan form, so that ${\mathcal B}_L\subset {\mathcal O}_{(2^k,1^{2(n-k)})}$ 
iff $l(L)=k$ and exactly as in case $A_n$ one gets
$${\mathcal O}_{(2^k,1^{2(n-k)})}\cap\nil_{2n}=\bigsqcup\limits_{L\in {\bf SLP}_{2n}(k)}{\mathcal B}_L.$$
\rem Let $\mathfrak g={\rm Lie}(G)$ be a simple Lie algebra and let $\mathfrak a\subset \mathfrak b$ 
be an abelian nilradical in $\mathfrak b={\rm Lie}(\B)$ of Borel subgroup $\B\subset G$. 
Let $\Phi_{\mathfrak a}\subset \Phi^+$ be such that $\mathfrak a=\bigoplus\limits_{\alpha\in \Phi_{\mathfrak a}}\Co X_\alpha$ 
where $X_\alpha$ is a root vector. 
Panyushev \cite{Pa2} shows that there is a bijection between $\B-$orbits in $\mathfrak a$ and strongly orthogonal subsets of $\Phi_{\mathfrak a}$,
namely each $\B-$orbit has a unique representative which is the sum of root vectors for some subset of strongly orthogonal roots of $\Phi_{\mathfrak a}$.\\
In case of $C_n$ one has $\mathfrak a=\bigoplus\limits_{1\leq i\leq j\leq n}\Co X_{(-i,j)}$ and in particular
$\mathfrak a\subset \mathcal X_{2n}^{(2)}$ so that his result in this case can be obtained as a corollary of the proposition above.

%--------- dimensions ----------
\subsection{$\B-$orbit dimension}\label{4.2}
For the expression of $\dim {\mathcal B}_L$ we need the following notation connected to \ref{2.8}:
\begin{itemize}
\item[(i)]For a fixed point $f$ of $L$ let $b_L(f)$ be the number of bridges over $f$.
For $L\in{\bf SLP}_{2n}$ put $b(L):=\sum\limits_{f\in Fp^+(L)}b_L(f)$.

\item[(ii)] For $\left\langle i,j\right\rangle,\left\langle k,l\right\rangle\in L\in {\bf SLP}_{2n}$ where 
$i<k$ we say that $\left\langle i,j\right\rangle$ {\it crosses} $\left\langle k,l\right\rangle$ (on the right) {\it non-negatively} if $i<k<j<l$ and  $k\geq -j.$
  For $\left\langle i,j\right\rangle\in L$ put $c_L(\left\langle i,j\right\rangle):=|\{\left\langle k,l\right\rangle\in L\ :\ i<k<j<l\ {\rm and}\ k>-j\}|$ to be the number of such crossings.
Note that
\begin{itemize}
\item{} For $(-i,-j)(i,j)\in L$ one has $c_L((-i,-j))=0$ and\\
$c_L((i,j))=|\{(k,l)\in L\ :\ i<k<j<l\}|$;
\item{} For $(-i,i)\in L$ one has $c_L((-i,i))=|\{\left\langle k,l\right\rangle\in L\ :\ -i<k<i<l\}|$;
\item{} For $(i,-j)(-i,j)$ one has 
$c_L((i,-j))=|\{\left\langle k,l\right\rangle\in L\ :\-i<k<i<l\}|$ and
$c_L((-i,j))=|\{\left\langle k,l\right\rangle -i<k<j<l\}|$;
\end{itemize}
Put $c(L):=\sum\limits_{\left\langle i,j\right\rangle\in L} c_L(\left\langle i,j\right\rangle)$.
Note that if we draw arcs $\left\langle i,j\right\rangle$ symmetrically with the center over $0.5(i+j)$ 
then $c(L)$ is the number of crossings over non-negative part of $L.$
\end{itemize}

In the example from \ref{1.4} one has $b(L)=b_L(5)=1$ and $c(L)=c_L(\left\langle -6,3\right\rangle)+c_L((-4,4))=2$.

By \cite{Co-Mc,Sp}
$$d(k,n):=\dim \left(\Or_{(2^k,1^{2(n-k)})}\cap \nil_{2n}\right)=
{\frac 1 2}\dim \Or_{(2^k,1^{2n-2k})}=
nk-{\frac 1 2}k(k-1).$$
In these terms we get
\begin{thm} For $L\in{\bf SLP}_{2n}(k)$ one has $\dim {\mathcal B}_L=d(k,n)-c(L)-b(L).$
\end{thm}
\begin{proof}
Let $L=(\pm i_1,\pm j_1)\ldots(\pm i_p,\pm j_p)(\pm l_1,\mp m_1)\ldots(\pm l_q,\mp m_q)(-d_1,d_1),\ldots(-d_r,d_r)$
so that
$Ep^+(L)=\{i_s,j_s\}_{s=1}^p\sqcup\{l_s,m_s\}_{s=1}^q\sqcup\{d_s\}_{s=1}^r$ and $2p+2q+r=k.$
One has
$$X_L=\sum\limits_{s=1}^pX_{(i_s,j_s)}+\sum\limits_{s=1}^qX_{(-l_s,m_s)}+\sum\limits_{s=1}^rX_{(-d_s,d_s)}$$

Recall that $\B_{2n}={\bf T}_{2n}\ltimes {\bf U}_{2n}$.
Let ${\mathcal T}_L$ denote ${\bf T}_{2n}-$orbit of $X_L$ and ${\mathcal U}_L$ denote ${\mathbf U}_{2n}-$ orbit of $X_L.$
Since these are subsets for the corresponding orbits in $\Sl_{2n}$ and since their intersection in $\Sl_{2n}$ is $X_L$ (cf. \cite{Me0}) we get
${\mathcal T}_L\cap {\mathcal U}_L=\{X_L\}$ so that $\dim \mathcal B_L=\dim\mathcal T_L+\dim\mathcal U_L.$

One has immediately (by strong orthogonality of roots) that $\dim\mathcal T_L=p+q+r.$

As for $\mathcal U_L$ one has $\dim\mathcal U_L=\dim {\bf U}_{2n}-\dim Z(X_L)$
where $Z(X_L)=\{U\in {\bf U}_{2n}\ :\ UX_LU^{-1}=X_L\}=\{V\in \nil_{2n}\ :\ VX_L=X_LV\}$ (the last equality is obtained by taking $V=U-I_{2n}$).

Note further that taking
$$N=\sum\limits_{1\leq i<j\leq n}N_{(i,j)}+ \sum\limits_{1\leq i<j\leq n}N_{(i,-j)}+\sum\limits_{1\leq i\leq n}N_{(i,-i)}$$
where $N_{(i,j)}$ is a variable in $\Co X_{(i,j)}$ we get $\dim\mathcal U_L$ is equal to the rank of linear system $X_LN-NX_L=0.$
In what follows we compute exactly this rank.

The proof is by induction on $n$.
The proposition holds trivially for $n=1$. 
Assume that it holds for $n-1$ and show for $n$.

We need the following notation: let $\widehat \cdot$ denote the projection $\pi(\cdot)$.
Let $(N)_{n}=\sum\limits_{i=1}^{n-1}(N_{(i,n)}+N_{(i,-n)})+N_{(n,-n)}$. 
Let $N'$ be $\widehat N$ regarded as an element of $\nil_{2n}$ (in other words $N=N'+(N)_n$). 
Let $(X_L)_n=X_L-X_{\widehat L}.$ Note that $(X_L)_n$ is at most a one root vector.
Namely, there are exactly 4 options:
\begin{itemize}
\item[(i)] $X_L=X_{\widehat L}$ that is $\widehat L=L$ and $(X_L)_n=0$.
\item[(ii)] $X_L=X_{\widehat L}+X_{(-n,n)}$ that is  $L=(-n,n)\widehat L$ and $l(\widehat L)=k-1$.
\item[(iii)] $X_L=X_{\widehat L}+X_{(-i,n)}$, that is $(l_q,-m_q)=(i,-n)$ i.e. $L=(i,-n)(-i,n)\widehat L$ and $l(\widehat L)=k-2$.
\item[(iv)] $X_L=X_{\widehat L}+X_{(i,n)}$, that is $(i_p,j_p)=(i,n)$ i.e. $L=(i,n)(-i,-n)\widehat L$ and $l(\widehat L)=k-2$.
\end{itemize}
All four cases are proved in the same way. We consider below in detail only case (iv) which is more complex and leave the check of other cases to the reader.

First of all, note that $\dim \mathcal T_L=p+q+r=\dim \widehat{\mathcal T}_{\widehat L}+1$.
Consider $\widehat X_{\widehat L}\in \X_{2(n-1)}^{(2)}$.
By induction hypothesis $\dim \widehat{\mathcal B}_{\widehat L}$ is defined by the expression.
To show the truth of the statement we have to compute the difference $\dim {\mathcal U}_L-\dim \widehat{\mathcal U}_{\widehat L}$.
Note that $[X_L,N]=0$ as a system can be regarded as
$\{[\widehat X_{\widehat L},\widehat N]=0\}\cup\{[X_{\widehat L},(N)_n]+[(X_L)_n,N']+[(X_L)_n,(N)_n]=0\}$.
Let $Eq_{(i,j)}$ denote equation at place $(i,j)$ of the system. Then
$$\{[X_{\widehat L},(N)_n]+[(X_L)_n,N']+[(X_L)_n,(N)_{(n)}]=0\}=\{Eq_{(i,n)}\}_{i=1}^{n-1}\sqcup \{Eq_{(i,-n)}\}_{i=1}^{n}\eqno{(*)}.$$

To show the expression let us compute $b(L)$ and $c(L)$ versus $b(\widehat L)$ and $c(\widehat L).$ To do this
we need the following notation.
For a link pattern $L$ and $-n\leq s<t\leq n$ let $f_L([s,t])$ be the number of fixed points at the interval $[s,t]$ (including the end points $s$ and $t$).
Note that $i$ is not a fixed point anymore so that we have to subtract $b_{\widehat L}(i)$
when compare $b(L)$ and $b(\widehat L)$.
All the bridges of $\widehat L$ over other positive fixed points points are also bridges of $L$.
Further note that there is a new bridge  $(i,n)$ over any fixed point between $i$ and $n$.
Summarizing, $b(L)=b(\widehat L)-b_{\widehat L}(i)+f_L([i,n])$.

To compute the number of crosses note that all the crosses of $\widehat L$ are crosses of $L$. Note also that
$(i,n)$ crosses $\left\langle l,m\right\rangle$ on the right (always non-negatively) iff $\left\langle l,m\right\rangle$ is a bridge over $i$
in $\widehat L.$ Nothing crosses $(i,n)$ on the right and $c_L((-i,-n))=0$.
Thus, $c(L)=c(\widehat L)+b_{\widehat L}(i)$.

Thus, to show that the expression for the dimension is satisfied we have to show that $(*)$ provides
$$\begin{array}{l}d(n,k)-b(L)-c(L)-\dim \mathcal T_L-(d(n-1,k-2)-b(\widehat L)-c(\widehat L)-\dim\widehat{\mathcal T}_{\widehat L})=\\
										nk-0.5k(k-1)-(b(\widehat L)-b_{\widehat L}(i) + f_{L}([i,n]))-(c(\widehat L)+b_{\widehat L}(i))-k-\\
										\quad ((n-1)(k-2)-0.5(k-2)(k-3)-b(\widehat L)-c(\widehat L)-(k-1))=\\
										2n-k-f_L([i,n])=n+f_L([1,i])\\
										\end{array}$$
 new equations.

In this case $(*)$ has a form
$$\begin{array}{l}Eq_{(t, n)}=\left\{\begin{array}{lll}\ N_{(j_s,n)}-N_{(i_s,i)}& {\rm if}\ t=i_s,\ i>i_s;&\\
\ N_{(j_s,n)}& {\rm if}\ t=i_s,\ i<i_s;&\\
                                               -N_{(t,i)} &{\rm if}\ t<i\ {\rm and}\ t\not\in\{i_s\}_{s=1}^{p-1};& (**)\\
                                               0&{\rm otherwise;}&\\
                                               \end{array}\right. \\
Eq_{(t,-n)}=\left\{\begin{array}{lll}\quad N_{(i_s,-n)}+N_{(i,-j_s)}&{\rm if}\ t=j_s;&\\
																					-N_{(m_s,-n)}+N_{(i,-l_s)}&{\rm if}\ t=l_s;&\\
																					-N_{(l_s,-n)}+N_{(i,-m_s)}&{\rm if}\ t=m_s;&\\
																					-N_{(d_s,-n)}+N_{(i,-d_s)}&{\rm if}\ i=d_s;& (***)\\
																					-\sum\limits_{s=1}^r(N_{(d_s,-n)})^2+N_{(i,-i)}^2+2N_{(i,-n)}&{\rm if}\ i=n;&\\
																					N_{(i,-t)}&{\rm otherwise};&\\
																					\end{array}\right.\\
								\end{array}$$

Note that all the equations containing variable $N_{(s,\pm n)}$ for some $s$
are linearly independent of the system $\widehat Eq$ and
there are exactly $k-1$ such equations: out of them $k-2$ are
$\{Eq_{(i_s,n)}\}_{s=1}^{p-1}\cup\{Eq_{(j_s,-n)}\}_{s=1}^{p-1}\cup\{Eq_{(l_s,-n)}\}_{s=1}^{q}\cup\{Eq_{(m_s,-n)}\}_{s=1}^{q}\cup\{Eq_{(d_s,-n)}\}_{s=1}^r$
and one is $Eq_{(n,-n)}$.

There are also $i-1-|\{i_s<i\}|$ equations of type $(**)$ involving $N_{(t,i)}$ where $t<i$ and $t\ne i_s$.
Since $i$ is a fixed point of $\widehat L$ one has that the only equations of
$\widehat Eq$ involving $N_{(t,i)}$ where $t\not\in \{i_s\}_{i_s<i}$ are 
$$\begin{array}{ll} Eq_{(i_s,i)}=-N_{(j_s,i)}& {\rm if}\ j_s<i,\\
										Eq_{(i,-m_s)}=-N_{(l_s,i)}& {\rm if}\ l_s<i,\\
										Eq_{(l_s,-i)}=-N_{(m_s,i)}& {\rm if}\ m_s<i,\\
										Eq_{(d_s,-i)}=-N_{(d_s,i)}& {\rm if}\ d_s<i.\\
										\end{array}$$
So that one has that among $i-1-|\{i_s<i\}|$ equations of type $(**)$ there are
$$i-1-|\{i_s<i\}|-|\{j_s<i\}|-|\{l_s<i\}|-|\{m_s<i\}|-|\{d_s<i\}|=f_L([1,i])$$
new linearly independent equations.

There are also $n-(p+2q+r)=n-k+p$ equations of type $(***)$
%$\genfrac{(}{)}{0pt}{}{***}{}$
involving only $N_{(i,-t)}$ where $t\not\in\{j_s,l_s,m_s,d_s\}$ for all possible $s$.
And the only equations of $\widehat Eq$ involving $N_{(i,-t)}$ where
$t\not\in\{j_s,l_s,m_s,d_s\}$ for all possible $s$ are  $Eq_{(i,-j_s)}=N_{(i,-i_s)}$
where $1\leq s\leq p-1$ so that subsystem $(***)$
%$\genfrac{(}{)}{0pt}{}{***}{}$
provides us with $n-k+1$ new linearly independent equations.

Summarizing, we get that $(*)$ adds $k-1+f_L([1,i])+n-k+1=n+f_L([1,i])$ new linearly independent equations, exactly as demanded.
\end{proof}

%------- B_{2n}-orbit closures --------
\section{Description of $\B_{2n}-$orbit closures}\label{5}
%------- Partial order --------
\subsection{Partial order on $\B_{2n}-$orbits}\label{4.1a} 
%For $L\in {\bf SLP}_{2n}$ let
%and $\ov{\mathcal B}_L$ denote its closure with respect to Zariski topology.
%In particular, $\ov{\mathcal B}_L$ as a spherical variety, is a finite union of $\B-$orbits.

The aim of this subsection is to show that the restriction of partial order $\preceq$ on ${\bf LP}_{2n}$ defined in \ref{2.8}
to ${\bf SLP}_{2n}$ defines the closures of $\B_{2n}-$orbits of square 0 in $\nil_{2n}$, namely:
\begin{thm} For $L\in {\bf SLP}_{2n}$ one has
$$\ov{\mathcal B}_L=\bigsqcup\limits_{\genfrac{}{}{0pt}{}{L'\in {\bf SLP}_{2n}:} {L'\preceq L}}\mathcal B_{L'}$$
\end{thm}
\noindent
{\it Outline of the proof:}\\
Note that by \ref{3.2} for $L,L'\in {\bf SLP}_{2n}$ $X_L, X_{L'}\in \nil_{2n}$ can be regarded (up to signs) 
as $Y_L, Y_{L'}\in \nil_{\Sl_{2n}}$ and $\B_{2n}$ as a subgroup of $\B_{SL_{2n}}$ 
so that $X_{L'}\in \ov{\mathcal B}_L$ implies $Y_{L'}\in \ov{\mathcal{B}}(Y_L)$ and respectively 
$\mathcal B_{L'}\subset \ov{\mathcal B}_L$ implies $\mathfrak B(Y_{L'})\subset \ov{\mathfrak B}(Y_L)$. 
Thus, it is obvious that $\ov{\mathcal B}_L\subseteq\bigsqcup\limits_{\genfrac{}{}{0pt}{}{L'\in {\bf SLP}_{2n}:}{ L'\preceq L}}\mathcal B_{L'}$.

In order to prove the other inclusion let $A(L):=\{L'\in {\bf SLP}_{2n}\ :\ L'\prec L\}$ and set $C(L),D(L),N(L)\subset A(L)$ to be
$$\begin{array}{l}C(L):=\{L'\in A(L)\ :\ {\rm if}\ L^{\prime\prime}\in A(L)\ {\rm satisfies}\ L'\preceq L^{\prime\prime}\preceq L\ \Rightarrow\ L^{\prime\prime}=L\ 
												{\rm or} \ L^{\prime\prime}=L' \};\\
									D(L):= C(L)\cap {\bf SLP}_{2n}(l(L));\\
									N(L):=\{L'\in A(L)\ :\ l(L')<l(L)\ {\rm and \ if}\ L^{\prime\prime}\in A(L)\ {\rm holds}\ L'\preceq L^{\prime\prime}\preceq L\ \Rightarrow\ 
									L^{\prime\prime}=L' \};\\
									\end{array}$$
Obviously it is enough to determine $D(L)$ and $N(L)$ and to show that
for $L'\in D(L)\cup N(L)$
$$\mathcal B_{L'}\subset\ov{\mathcal B}_L.\eqno{(*)}$$

Indeed, for any $L^{\prime\prime}\in A(L)\setminus (D(L)\cup N(L))$
there exists $L'\in D(L)\cup N(L)$ such that $L'\succ L^{\prime\prime}$. 
Thus, proof of $(*)$ will complete the proof of the theorem since by \ref{2.7} 
the boundary of $\mathcal B_L$ is  $\bigcup\limits_{L'\in C'(L)}\mathcal B_{L'}$,  where
$$C'(L)=\{L'\in {\bf SLP}_{2n}\ :\ \mathcal B_{L'}\subset \ov{\mathcal B}_L\ {\rm and}\ {\rm codim}_{\ov{\mathcal B}_L}\mathcal B_{L'}=1\}$$
is a subset of $A(L)$. Thus, if $(*)$ is true $C'(L)\subset D(L)\cup N(L)$ and also $C'(L)\cap N(L)=C(L)\cap N(L)$ so that $C'(L)=C(L).$
Note that we get automatically that $D(L)\subset C'(L)$ and $L'\in N(L)\cap C(L)$ iff ${\rm codim}_{\ov{\mathcal B}_L}\mathcal B_{L'}=1.$

In order to define $D(L)$ and $N(L)$ we consider
$L\in {\bf SLP}_{2n}(k)$ as an element of ${\bf LP}_{2n}(k)$ and define $\widehat A(L):=\{L'\in {\bf LP}_{2n}\ :\ L'\prec L\}$ and respectively
$$\begin{array}{l}\widehat C(L):=\{L'\in \widehat A(L)\ :\ {\rm codim}_{\ov{\mathfrak B}(Y_L)}\mathfrak B(Y_{L'})=1  \};\\
									\widehat D(L):= \widehat C(L)\cap {\bf LP}_{2n}(k);\\
									\widehat N(L):=\{L'\in \widehat A(L)\, :\, l(L')<k\ {\rm and \ if}\ L^{\prime\prime}\in \widehat A(L)\ {\rm holds}\ L'\preceq L^{\prime\prime}\preceq L\ 																	\Rightarrow\ L^{\prime\prime}=L' \};\\
									\end{array}$$
									
Our strategy to construct $D(L),\ N(L)$ is as follows:\\
Note that by definitions of $D(L)$ and $\widehat D(L)$ (resp. $N(L)$ and $\widehat N(L)$) for any $L'\in D(L)$ (resp. $L'\in N(L)$) 
there exists $L^{\prime\prime}\in \widehat D(L)$ (resp. $L^{\prime\prime}\in \widehat N(L)$) 
such that $L^{\prime\prime}\succeq L'.$ Since $L'\in {\bf SLP}_{2n}$ one has
$R_{L'}$ is symmetric around antidiagonal, namely $(R_{L'})_{s,t}=(R_{L'})_{2n+1-t,2n+1-s}$.

If $L^{\prime\prime}\in {\bf SLP}_{2n}$ then $L'=L^{\prime\prime}$, otherwise let us define $SR_{L^{\prime\prime}}$ 
to be the "symmetrization" of $R_{L^{\prime\prime}}$, that is
$$(SR_{L^{\prime\prime}})_{s,t}:=\min\{(R_{L^{\prime\prime}})_{s,t},(R_{L^{\prime\prime}})_{2n+1-t,2n+1-s}\}.$$
Obviously, by the symmetry we get $R_{L'}\preceq SR_{L^{\prime\prime}}$. In particular,
\begin{itemize}
\item{} If there exists $\widehat L\in {\bf LP}_{2n}$ such that
$R_{\widehat L}=SR_{L^{\prime\prime}}$ then $L'=\widehat L\in D(L)$ (resp. in $N(L)$) 
and we are done. As we show in the next subsection this is exactly the case of $N(L)$.
\item{} If there is no such $\widehat L$ (and this happens in some cases  of $D(L)$) we have to find the set of 
maximal symmetric link patterns $\{L_i\}$ satisfying $R_{L_i}\prec SR_{L^{\prime\prime}}$.
Note that this subset can contain more than one element and that the elements $L_i$ not always in $D(L)$, 
but if we take  $D'(L)$ to be the union of $L^{\prime\prime}\in {\bf SLP}_{2n}$, 
$\widehat L$ and  $\{L_i\}$ obtained for all possible $L^{\prime\prime}\in \widehat D(L)$ we get that $D(L)\subset D'(L).$ 
Thus it is enough to define $D'(L)$ and then to extract $D(L)$ from it and to show $(*)$ for any $L'\in D(L).$
\end{itemize}

In the next subsection we construct $N(L)$ according to this plan and in the subsections \ref{5.3} we construct $D(L)$ according to this plan.

\begin{rem} In order to prove the theorem we need only to compute $N(L)\cap C(L)$ 
but for our further applications to orbital varieties we need all of $N(L).$
\end{rem}

We need some notation. Let $L\in {\bf LP}_{2n}$. For $\{(i_s,j_s)\}_{s=1}^k\in L$ let $L^-_{(i_1,j_1)\ldots(i_k,j_k)}$ 
be a new link pattern obtained from $L$ by deleting the arcs $\{(i_s,j_s)\}_{s=1}^k.$ 
For $i,j\in Fp(L)$ put $L(i,j)=(i,j)L$ to be obtained from $L$ by adding the arc $(i,j)$.

Note also that on one hand we consider link patterns of $-n,\ldots,-1,1,\ldots,n$ so that all the intervals are $[s,t]$ where $-n\leq s<t\leq n$ and $s,t\ne 0.$ 
On the other hand $R_L$ is a strictly upper triangular $2n\times 2n$ matrix so that its possible non-zero elements are with double indexes $1\leq s'<t'\leq 2n$.
Thus in what follows we use always put $s,t \ :\ -n\leq s<t\leq n,\ s,t\ne 0$ and resp. we define $s',t'\ :\ 1\leq s'<t'\leq 2n$ where
$$s'(resp.\ t')=\left\{\begin{array}{ll}s+n+1\ (resp.\ t+n+1) &{\rm if}\ s<0\ (resp.\ t<0);\\
                                             s+n\ (resp.\ t+n)&{\rm otherwise};\\
                                             \end{array}\right.$$
%------- Construction of N(L) --------
\subsection{Construction of $N(L)$}\label{5.2}
We start with a simpler $N(L)$.

For $L\in {\bf LP}_{2n}$ arc $(i,j)\in L$ is called {\it external} if there is no $(i',j')\in L$ 
such that $i'<i<j<j'$. Let $E(L)$ be the set of its external arcs.

By \cite{Me1} $\widehat N(L)=\{L^-_{(i,j)}\ :\ (i,j)\in E(L)\}$.

Further note that if $L\in{\bf SLP}_{2n}$ then for any $(\pm i, j)\in E(L)$ one has either $(-i,j)=(-i, i)$ 
(and if there exists such external arc it is unique) or both\\  
$(i,\pm j),(-i,\mp j)\in E(L)$.
Put $E'(L)=\{(i,\pm j)\in E(L)\}.$

By symmetry considerations it is enough to consider $L_{\left\langle i,j\right\rangle}^-\in \widehat N(L)$ where $\left\langle i,j\right\rangle\in E'(L).$
\begin{prop} Let $L\in {\bf SLP}_{2n}$. One has  $|N(L)|=|E'(L)|$ and for $(i,\pm j)\in E'(L)$ the corresponding element of $N(L)$ is defined as follows
$$L':=\left\{\begin{array}{ll}L^-_{(i,j)}&{\rm if}\ j= -i;\\
L^-_{(i,\pm j)(-i,\mp j)}(-j, j)&{\rm otherwise};\\
\end{array}\right.$$
For any $L'\in N(L)$ one has $L'\in \ov{\mathcal B}_L.$
\end{prop}
\begin{proof}
For $(-i,i)\in E'(L)$ we get that $L^-_{(-i,i)}$ is symmetric and it is in $\widehat N(L)$ thus it is in $N(L)$.
Put $B_m:=H_i(\frac{1}{m})$, then $ \lim\limits_{m\rightarrow\infty} {B_m.X_{L}}=X_{L_{(-i,i)}^-}$.

For $(i,\pm j)\in E'(L)$ such that $j\ne -i$ put $L^{\prime\prime}=L^-_{(i,\pm j)(-i,\mp j)}$ and put $\widehat j=\pm j$ so that $(i,\pm j)=(i,\widehat j)$. One has
$$R_L=R_{L^{\prime\prime}}+R_{(i,\widehat j)(-i,-\widehat j)}\quad{\rm and}\quad R_{L^-_{(i,\widehat j)}}=R_{L^{\prime\prime}}+R_{(-i,-\widehat j)}$$
so that
$$\begin{array}{rl}
(R_{L^-_{(i,j)}})_{s',t'}&=\left\{ \begin{array}{ll}(R_{L^{\prime\prime}})_{s',t'}+1&{\rm if}\ s\leq \min\{-i,-\widehat j\} \ {\rm and}\  t\geq \max\{-i,-\widehat j\} ;\\
                                             (R_{L^{\prime\prime}})_{s',t'}&{\rm otherwise};\\
                                            \end{array}\right.\\
{\rm and\ respectively,}&\\
(SR_{L^-_{(i,\pm j)}})_{s',t'}&
=\left\{ \begin{array}{ll}(R_{L^{\prime\prime}})_{s',t'}+1&{\rm if}\ s\leq -j\ {\rm and}\  t\geq j;\\
                                             (R_{L^{\prime\prime}})_{s',t'}&{\rm otherwise};\\
                                            \end{array}\right.\\
                                            \end{array}$$
Further note that $SR_{L^-_{(i,\widehat j)}}=R_{L^{\prime\prime}(j,-j)}$ so that $L'=L^-_{(i,\widehat j)(-i,-\widehat j)}(-j,j).$
To complete the proof note that for $B_m:=T_i(m^{sign(\widehat j)})U_{e_j+sign(\widehat j) e_i}(-sign(\widehat j) 0.5)$ one has
$ \lim\limits_{m\rightarrow\infty} {B_m.X_{L}}=X_{L'}$.
\end{proof}
To finish with $N(L)$ note
\begin{lem}  For $(i,\pm j)\in E'(L)$ and $L'$ a corresponding element of $N(L)$ one has
$\dim{\mathcal B}_L-\dim\mathcal B_{L'}=1$ if and only if $(i,\pm j)=(i,-j)$ and $Fp(L)\subset [-i,i]$. In particular,
this is a necessary and sufficient condition for $L'$ to be in $C(L).$
\end{lem}
The proof is an absolutely straightforward computation, so we skip it.

%------- Construction of D(L) --------
\subsection{Construction of $D(L)$}\label{5.3}
{\it $\widehat D(L):$}
Recall that in the case of $L \in LP_{2n}$ there are exactly 3 types of elements in $\widehat D(L)$ obtained by three types of elementary minimal moves.
Let us recall these moves and the conditions for their minimality:
\begin{itemize}
	\item[(M1)] Let $(i,j)\in L$ and let $l_i$ be the maximal fixed point $l_i<i$ (resp. let $r_j$ be the minimal fixed point $r_j>j$).
The first elementary move is obtained by moving  $i$ to $l_i$,
that is $L'=L^-_{(i,j)}(l_i,j)$ (resp. by moving  $j$ to $r_j$, that is $L'=L^-_{(i,j)}(i,r_j)$.
The condition for minimality of such a move is that there are no $(s,t)\in L$ such that $l_i<s<i<j<t$ (resp. such that $s<i<j<t<r_j$).

	\item[(M2)] Let $(i,j),(k,l)\in L$, they are called consecutive if $i<j<k<l$.
The second elementary move is exchanging $j$ and $k$ that is $L'=L^-_{(i,j),(k,l)}(i,k)(j,l)$.
The condition for the minimality of such a move is that there are no fixed points on $[j,k]$
and there are no $(s,t)\in L$ such that either $s<i<j<t<k$ or $j<s<k<l<t$.

	\item[(M3)] Let $(i,j),(k,l)\in L$, they are called concentric if $i<k<l<j$.
The third elementary move is exchanging $i$ and $k$ that is $L'=L^-_{(i,j),(k,l)}(i,l)(k,j)$.
The condition for the minimality of such a move is that there are no $(s,t)\in L$ which is concentric between them,
that is no $(s,t)$ such that $i<s<k<l<t<j$.
\end{itemize}

Note that if $L'\in \widehat D(L)$ is not symmetric then ${\mathfrak B}(Y_{L_S})$ for a symmetric
$L_S \ : L_S \prec L'$ must be of at least codimension 1 in $\overline {\mathfrak B}(Y_{L'})$,
that is at least of codimension 2 in $\overline {\mathfrak B}(Y_L)$.
As we show by the construction for any $L_S\in D(L)$ one has
$\codim_{\overline {\mathfrak B}(Y_L)}{\mathfrak B}(Y_{L_S}) \leq 3$.
Moreover a straightforward computation shows that $\dim \mathcal B_L-\dim \mathcal B_{L_S}=1.$

As a result we get three types of elements in $D(L)$ obtained from (M1)-(M3)
by two mirror moves (or by the fact that $L'\in \widehat D(L)$ is symmetric);
we also get the fourth type of elements in $D(L)$ (connected to $(i,-j)(-i,j)$ 
and only these elements are such that
$\codim_{\overline {\mathfrak B}(Y_L)}{\mathfrak B}(Y_{L_S})=3$).

We illustrate each case by the picture of corresponding link patterns where we denote by small points
all points and their arcs which are not changed under the move.
We always assume that the considered arcs satisfy the minimality condition for the corresponding move (M1)-(M3).\\

We have to consider case by case $L^{\prime\prime}$ obtained by moves (M1)-(M3), to constract the corresponding maximal 
symmetric link patterns and to check for symmetric $L'$ obtained in this way that either $\dim \mathcal B_L-\dim\mathcal B_{L'}=1$ and $X_{L'}\in \overline{\mathcal B}_L$, or $\dim \mathcal B_L-\dim\mathcal B_{L'}>1$ and there exists $\widetilde{L}\in{\bf SLP}_{2n}$
such that $L\succ\widetilde L\succ L'$. Here we consider in detail only one (but the most complex) case. All the cases are described completely, however in short in Appendix A. The proofs are similar to the proof below, following the same ideas. 

\begin{prop}
Let $L\in {\bf SLP}_{2n}$ with $(i,j)(k,-l)\in L$ for $1\leq i<j<k<l \leq n$.
Denote $L_0:=L^-_{(i,j)(-i,-j)(k,-l)(-k,l)}$, then $(i,j)$ is under both $(k,-l)$ and $(-k,l)$ and 
\begin{itemize}
\item[(i)] If $(i,j),(-i,-j)\in L$ are (consecutive arcs) satisfying conditions of (M2) then
$L_{I}:=L^-_{(i,j)(-i,-j)}(i,-j)(-i,j)\in D(L);$\\
One has $\lim\limits_{m\rightarrow\infty} {T_j(m).U_{2e_j}(-\frac{1}{m}).X_{L}}=X_{L_{I}}.$

\item[(ii)] If $(i,j),(-k,l)\in L$ are (concentric arcs) satisfying conditions of (M3) then
 $L_{II}:=L_0(i,l)(-i,-l)(-j,j)(-k,k)\in D(L).$\\
Let $U_m:=U_{e_l-e_j}(1).U_{e_k-e_j}(1).U_{e_l-e_k}(0.5).U_{ e_l+e_i}(0.5).U_{e_k+e_i}(1).U_{e_j+e_i}(-0.5)$
and $T_m:=T_k(\sqrt{-1}).T_l(m).T_i(m)$ then $\lim\limits_{m\rightarrow\infty} {T_m.U_m.X_{L}}=X_{L_{II}}.$

\item[(iii)] If $(i,j),(k,-l)\in L$  satisfy conditions of (M3) then\\
$L_{III}:=L_0(i,k)(-i,-k)(j,-l)(-j,l)\in D(L);$ and 
$L_{II}$ from (ii) is obtained via this move as well. Moreover,  $L_{II}\in D(L)$ iff $(i,j),(-k,l)$ satisfy conditions of (M3).\\
One has $\lim\limits_{m\rightarrow\infty} {T_k(m).T_j(-1).T_i(m).U_{e_k-e_j}(1).U_{e_l+e_i}(1).X_{L}}=X_{L_{III}}.$
\end{itemize}
\end{prop}

The following picture represents the proposition above:
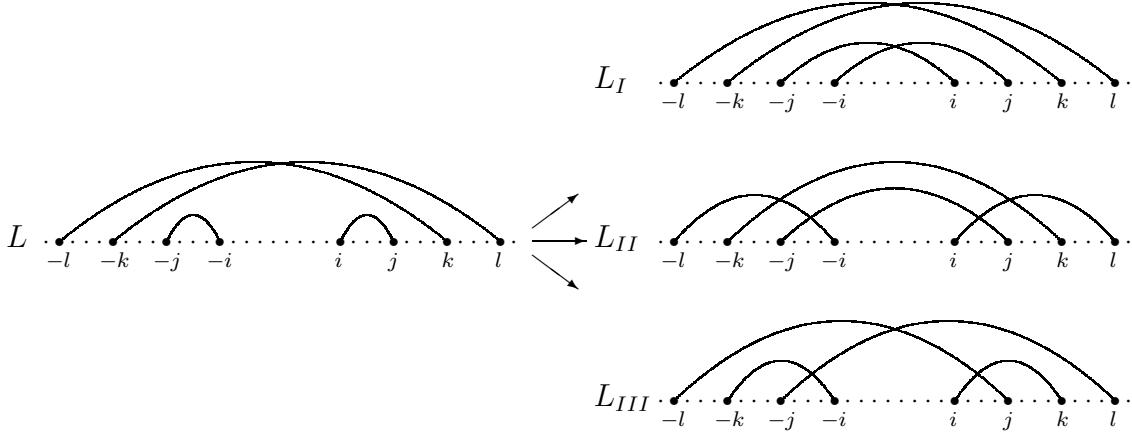
\begin{center}
\begin{picture}(100,170)(0,-60)
\put(-160,8){$L$}
\multiput(-145,10)(5,0){36}%
{\circle*{1}}
\put(-140,10){\circle*{3}}
 \put(-145,2){${}_{-l}$}
\put(-120,10){\circle*{3}}
 \put(-125,2){${}_{-k}$}
 \put(-100,10){\circle*{3}}
 \put(-105,2){${}_{-j}$}
 \put(-80,10){\circle*{3}}
 \put(-85,2){${}_{-i}$}
 \put(-35,10){\circle*{3}}
 \put(-37,2){${}_{i}$}
 \put(-15,10){\circle*{3}}
 \put(-17,2){${}_{j}$}
 \put(5,10){\circle*{3}}
 \put(3,2){${}_{k}$}
\put(25,10){\circle*{3}}
 \put(23,2){${}_{l}$}

 \qbezier(-140,10)(-67.5,70)(5,10)
 \qbezier(-120,10)(-47.5,70)(25,10)
 \qbezier(-100,10)(-90,30)(-80,10)
 \qbezier(-35,10)(-25,30)(-15,10)

\put(37,15){\vector(4,3){17}}
\put(60,68){$L_{I}$}
\multiput(85,70)(5,0){36}%
{\circle*{1}}
\put(90,70){\circle*{3}}
 \put(85,62){${}_{-l}$}
\put(110,70){\circle*{3}}
 \put(105,62){${}_{-k}$}
 \put(130,70){\circle*{3}}
 \put(125,62){${}_{-j}$}
 \put(150,70){\circle*{3}}
 \put(145,62){${}_{-i}$}
 \put(195,70){\circle*{3}}
 \put(193,62){${}_{i}$}
 \put(215,70){\circle*{3}}
 \put(213,62){${}_{j}$}
 \put(235,70){\circle*{3}}
 \put(233,62){${}_{k}$}
\put(255,70){\circle*{3}}
 \put(253,62){${}_{l}$}

 \qbezier(90,70)(162.5,130)(235,70)
 \qbezier(110,70)(182.5,130)(255,70)
 \qbezier(130,70)(162.5,100)(195,70)
 \qbezier(150,70)(182.5,100)(215,70)

\put(37,10){\vector(4,0){20}}
\put(60,8){$L_{II}$}
\multiput(85,10)(5,0){36}%
{\circle*{1}}
\put(90,10){\circle*{3}}
 \put(85,2){${}_{-l}$}
\put(110,10){\circle*{3}}
 \put(105,2){${}_{-k}$}
 \put(130,10){\circle*{3}}
 \put(125,2){${}_{-j}$}
 \put(150,10){\circle*{3}}
 \put(145,2){${}_{-i}$}
 \put(195,10){\circle*{3}}
 \put(193,2){${}_{i}$}
 \put(215,10){\circle*{3}}
 \put(213,2){${}_{j}$}
 \put(235,10){\circle*{3}}
 \put(233,2){${}_{k}$}
\put(255,10){\circle*{3}}
 \put(253,2){${}_{l}$}

\qbezier(90,10)(120,45)(150,10)
 \qbezier(110,10)(172.5,70)(235,10)
 \qbezier(130,10)(172.5,50)(215,10)
 \qbezier(195,10)(225,45)(255,10)

\put(37,5){\vector(4,-3){17}}

\put(60,-52){$L_{III}$}
\multiput(85,-50)(5,0){36}%
{\circle*{1}}
\put(90,-50){\circle*{3}}
 \put(85,-58){${}_{-l}$}
\put(110,-50){\circle*{3}}
 \put(105,-58){${}_{-k}$}
 \put(130,-50){\circle*{3}}
 \put(125,-58){${}_{-j}$}
 \put(150,-50){\circle*{3}}
 \put(145,-58){${}_{-i}$}
 \put(195,-50){\circle*{3}}
 \put(193,-58){${}_{i}$}
 \put(215,-50){\circle*{3}}
 \put(213,-58){${}_{j}$}
 \put(235,-50){\circle*{3}}
 \put(233,-58){${}_{k}$}
\put(255,-50){\circle*{3}}
 \put(253,-58){${}_{l}$}

\qbezier(90,-50)(152.5,10)(215,-50)
\qbezier(130,-50)(192.5,10)(255,-50)
\qbezier(110,-50)(130,-20)(150,-50)
\qbezier(195,-50)(215,-20)(235,-50)

\end{picture}
\end{center}

\begin{proof} The computations showing $X_{L_I}, X_{L_{II}},X_{L_{III}}\in\overline{\mathcal B}_L$ are 
immediate check of the limits written in the proposition. 
We have to show that (M2) in case (i) and (M3) in cases (ii) and (iii) provide us with $L_I,\ L_{II},\ L_{III}$
respectively and only with them.
\begin{itemize}
\item[(i)] Immediately, $L^{\prime\prime}=L^-_{(i,j)(-i,-j)}(i,-j)(-i,j)$ is symmetric and $L_{I}=L^{\prime\prime}\in D(L).$
In this case $\dim \mathfrak B(Y_L)-\dim\mathfrak B(Y_{L_I})=\dim\mathcal B_L-\dim\mathcal B_{L_I}=1.$

\item[(ii)]
For $L^{\prime\prime}=L^-_{(i,j)(-k,l)}(i,l)(j,-k)\in \widehat D(L)$ one can check straightforwardly
$SR_{L^{\prime\prime}}=R_{L_{II}}$ where $L_{II}=L_0(i,l)(-i,-l)(-j,j)(-k,k)$. Further, again straightforwardly,
one gets $\dim \mathcal B_L-\dim\mathcal B_{L_{II}}=1$. Note that in this case $\dim\mathfrak B(Y_L)-\dim\mathfrak B(Y_{L_{II}})=3.$ 

\item[(iii)] Let us compute the last case in more details.
Recall that\\ $L^{\prime\prime}=L^-_{(i,j)(k,-l)}(i,k)(j,-l)\in \widehat D(L)$. Thus,
$$(R_{L^{\prime\prime}})_{s',t'}=\left\{\begin{array}{ll} (R_{L_0})_{s',t'}+4	&{\rm if}\ s\leq -l\ {\rm and}\ t\geq l;\\
                                            (R_{L_0})_{s',t'}+3 &{\rm if}\ s\leq -l\ {\rm and}\ k\leq t<l\ {\rm or}\\
																																					&\quad -l<s\leq -k\ {\rm and}\ t\geq l;\\
                                            (R_{L_0})_{s',t'}+2 &{\rm if}\ s\leq -l\ {\rm and}\ j\leq t<k\ {\rm or}\\
																																					&\quad -l<s\leq -j\ {\rm and}\ k\leq t<l\ {\rm or}\\
																																					&\quad -k< s\leq -j\ {\rm and}\ t\geq k;\\
                                            (R_{L_0})_{s',t'}+1 &{\rm if}\ s\leq -j\ {\rm and}\ -i\leq t<j\ {\rm or}\\
																																					&\quad -l<s\leq -j\ {\rm and}\ -i\leq t< k\ {\rm or}\\
																																					&\quad -j< s\leq i\ {\rm and}\ t\geq k;\\
                                            (R_{L_0})_{s',t'}   &{\rm otherwise};\\
																						\end{array}\right.$$
respectively,
$$(SR_{L^{\prime\prime}})_{s',t'}=\left\{\begin{array}{ll}(R_{L_0})_{s',t'}+4	&{\rm if}\ s\leq -l\ {\rm and}\ t\geq l;\\
                                            (R_{L_0})_{s',t'}+3 &{\rm if}\ s\leq -l\ {\rm and}\ k\leq t<l\ {\rm or}\\
																																					&\quad -l<s\leq -k\ {\rm and}\ t\geq l;\\
                                            (R_{L_0})_{s',t'}+2 &{\rm if}\ s\leq -l\ {\rm and}\ j\leq t<k\ {\rm or}\\
																																					&\quad -l<s\leq -k\ {\rm and}\ k\leq t<l\ {\rm or}\\
																																					&\quad -k< s\leq -j\ {\rm and}\ t\geq l;\\
                                            (R_{L_0})_{s',t'}+1 &{\rm if}\ s\leq -k\ {\rm and}\ -i\leq t<j\ {\rm or}\\
																																					&\quad -l<s\leq -j\ {\rm and}\ j\leq t< k\ {\rm or}\\
																																					&\quad -k<s\leq -j\ {\rm and}\ j\leq t < l\ {\rm or}\\
																																					&\quad -j< s\leq i\ {\rm and}\ t\geq k;\\
                                            (R_{L_0})_{s',t'}   &{\rm otherwise};\\
																						\end{array}\right.$$

Note that $SR_{L^{\prime\prime}}$ is not $R_{L_S}$ for some $L_S\in {\bf SLP}_{2n}$.
Otherwise, if it would be then $(SR_{L^{\prime\prime}})_{n+1-k,n+1-j}=(R_{L_0})_{n+1-k,n+1-j}+1$ and
$(SR_{L^{\prime\prime}})_{n+j,n+1-j}=(R_{L_0})_{n+j,n+1-j}+1$
together with the properties of $R_{L_0}$ would imply\\
$(SR_{L^{\prime\prime}})_{n+1-k,n+j}=(R_{L_0})_{n+1-k,n+j}+2$ which does not happen.

However there are exactly two maximal symmetric link patterns:\\
$L_{III}=L_0(i,k)(-k,-i)(j,-l)(l,-j)$
and $L_{II}=L_0(i,l)(-l,-i)(j,-j)(k,-k)$ (from (ii))
such that
$R_{L_{II}}, R_{L_{III}}<SR_{L^{\prime\prime}}.$
Moreover, $R_{L_{III}} \in D(L)$ since 
$\dim \mathcal B_{L}-\dim \mathcal B_{L_{III}}=1$.
Note that $\dim{\mathfrak B}(Y_L)-\dim{\mathfrak B}(Y_{L_{III}})=2$.

As for $L_{II}$, if $(i,j),(-k,l)$ satisfy conditions of (M3) then $L_{II}\in D(L)$ by (ii). Otherwise there exists $\left\langle p,q\right\rangle \in L$ such that $-k<p<i<j<q<l$ and $L'=L_{(i,j)(-i,-j)\left\langle p,q\right\rangle\left\langle -q,-p\right\rangle}^-
(i,q)\left\langle p,j\right\rangle(-i,-q)\left\langle -j,-p\right\rangle$ satisfies $L\succ L'\succ L_{II}.$ One can check this directly.
\end{itemize}
\end{proof}

%-------- Applications ---------
\section{Applications to orbital varieties of square zero}\label{6}
\subsection{Orbital Varieties of square zero and maximal $\B_{2n}-$orbits}\label{a1}
Recall the notion of an orbital variety from \ref{1.5}.
A general construction of Steinberg (which we explain in the next subsection) provides a way to 
construct orbital varieties, but the construction is complex so it is not easy to get the information out of it. 
Further Mcgovern \cite{Mc1, Mc2} provided a way connecting orbital varieties in types $B_n,\ C_n$ and $D_n$ to standard domino tableaux (SDT), 
so that an orbital variety in ${\mathcal O}_\lambda$ is labelled by SDT of shape $\lambda$.

Orbital varieties are obviously $\B-$stable, so that in particular if $\lambda=(2^k,1^{2n-2k})$ 
each orbital variety of ${\mathcal O}_\lambda\subset\Sp_{2n}$ admits a dense $\B_{2n}-$orbit.

On the other hand by our computations above for $L\in{\bf SLP}_{2n}(k)$ 
one has $\dim \mathcal B_L=0.5 \dim{\mathcal O}_{(2^k,1^{2n-2k})}$ 
if and only if $L$ is  without crossings and without fixed points under the arcs. 
In such a way we get a bijection between maximal $\B_{2n}-$orbits in ${\mathcal O}_{(2^k,1^{2n-2k})}\cap{\mathfrak n}_{2n}$ 
and orbital varieties so that we can get the information on the latter from the theory developed above.

We would like to read the information on an orbital variety from its domino tableau, 
so first of all we have to connect a maximal link pattern with $k$ arcs to a domino tableau of shape $(2^k,1^{2n-2k}).$ 
We do this in the next two subsections.

\subsection{Steinberg's construction and SDT with two columns}\label{a2}
In \cite{St}  Steinberg provided a general construction of an orbital variety for 
a simple Lie algebra $\mathfrak g$ that we explain in short (only for $\Sp_{2n}$, although the general constraction is the same) in this subsection.

%Let $\mathfrak b ={\rm Lie}\, (\B)$ be a Borel subalgebra of $\mathfrak g$ and let 
%$\mathfrak n=\bigoplus\limits_{\alpha\in \Phi^+}X_\alpha$ be its nilradical. 
For $w$ in Weyl group $\W_{2n}$ put $w(\mathfrak n_{2n}):=\bigoplus\limits_{\alpha\in \Phi^+} X_{w(\alpha)}$ 
and consider $\mathfrak n_{2n}\cap w(\mathfrak n_{2n})$. 
This is a subspace in $\mathfrak n_{2n}$ so that there exists the unique nilpotent orbit which 
we denote ${\mathcal O}_w$ such that $({\mathfrak n_{2n}}\cap w(\mathfrak n_{2n}))\cap {\mathcal O_w}$ is dense in $\mathfrak n_{2n}\cap w(\mathfrak n_{2n})$.

By Steinberg's construction ${\mathcal V}_w=\overline {\B.(\mathfrak n_{2n}\cap w(\mathfrak n_{2n}))}\cap \mathcal O_w$ 
is an orbital variety of $\mathcal O_w$ and moreover each orbital variety is obtained in such a way.

One can partition $\W_{2n}$ into (right) cells according to the equivalence relation $w\sim y$ if $\mathcal V_w=\mathcal V_y$. 

Recall that $\W_{2n}$ can be represented as a group of 
$w=\left[\begin{array}{llll}1&2&\ldots&n\\
														a_1&a_2&\ldots&a_n\\ 
														\end{array}\right]$
where $-n\leq a_i\leq n$ and $\{|a_i|\}_{i=1}^n=\{i\}_{i=1}^n.$
In what follows we will write $w\in \W_{2n}$ in a word form, namely $w=[a_1,a_2,\ldots,a_n].$

Recall that $\lambda=(\lambda_1,\ldots,\lambda_k)\in\mathcal P_n$ can be visualized as an array with $k$ rows all starting at the same place on the left where $i$-th row contains $\lambda_i$ cells. Such an array is called a {\it Young diagram of shape $\lambda$}.

Recall that a {\it domino} $[i:i]$ is a Young diagram with two cells filled with number $i$. We call a domino {\it horisontal} if it is of shape $(2)$  and {\it vertical} if it is of shape $(1,1)$. 

A {\it standard domino tableau} (STD) of shape $\lambda=(\lambda_1,\ldots,\lambda _k)\in \mathcal P_{2n}$ is obtained by filling Young diagram of shape $\lambda$ with dominoes $[1:1],\ldots,[n:n]$ in such a way that they increase along a row from left to right and along a column from top to bottom.

Gurfinkel-McGovern algorithm provides a way to associate the ordered pair of SDTs of the same shape $(T(w),S(w))$ to $w\in \W_{2n}$.
One has ${\mathcal O}_w={\mathcal O}_\lambda$ where $\lambda$ is a shape of $T(w)$ and ${\mathcal V}_w={\mathcal V}_y$ if and only if $T(w)=T(y).$ 
Thus we can label orbital varieties by SDTs. In what follows we write $\mathcal V_T:= \mathcal V_w$ if $T=T(w).$
The procedure is involved so we do not describe it in general.

We consider orbital varieties of nilpotency order 2 so that they are labelled by SDT's with two columns. For such SDT we produce $w$ such that $T(w)=T.$

Given a two-column SDT $T$ we say $j\in T_2$ if either a vertical domino $[j:j]$ belongs to the second column 
or a horizontal domino $[j:j]$ belongs to both first and second columns. 
Put $ T_2=\{b_1<\ldots< b_k\}$ to be these entries. 
We say $i\in T_1$ if a vertical domino $[i:i]$ belongs to the first column. 
Put $T_1=\{a_1<\ldots< a_s\}$ to be these entries. 
Put $w_T=[b_k,\ldots,b_1,-a_1,\ldots,-a_s]$. 
Note that we always get a decreasing sequence so that \\
$X_{e_j-e_i}\in {\mathfrak n_{2n}}\cap w_T(\mathfrak n_{2n})$ iff $j\in T_2,\ i\in T_1$ and $i<j$ \\
$X_{2e_i}\in {\mathfrak n_{2n}}\cap w_T(\mathfrak n_{2n})$ iff $i\in T_2$\\
$X_{e_j+e_i} \in {\mathfrak n_{2n}}\cap w_T(\mathfrak n_{2n})$ iff $i, j\in T_2$\\

A straightforward easy check provides $T(w_T)=T.$

As an example for $n=7$, a SDT, its corresponding word and ${\mathfrak n_{14}}\cap w(\mathfrak n_{14})$ are:
$$T=\vcenter{\halign{& \hfill#\hfill\cr
\vsal
\multispan{4}{\hrulefill}\cr
  \ssod
	\vb &\ \,1\ &\vb &\ 2\ \, &\vb\cr
 \ssa
\multispan{4}{\hrulefill}\cr
 \ssal
\vbs & &\, 3 & & \vbs\cr
 \vsal\multispan{4}{\hrulefill}\cr
 \ssal\vbs & &\, 4 & & \vbs\cr
 \vsal\multispan{4}{\hrulefill}\cr
 \ssod\vb &\ \, 5 &\vb &\ 7 &\vb\cr
 \ssa\multispan{4}{\hrulefill}\cr
 \ssod\vb &\ \,6 &\vb\cr
 \ssod
\multispan{3}{\hrulefill}\cr
 }}
 \ \longrightarrow\ \begin{array}{c}w_T=[7,4,3,2,-1,-5,-6]\\ \ \\
 {\mathfrak n_{14}}\cap w_T(\mathfrak n_{14}={\mathrm Span}\left\{\begin{array}{l}X_{e_2-e_1},\, 
X_{e_3-e_1},\, X_{e_4-e_1},\, X_{e_7-e_1}, \\
 X_{e_7-e_5},\, X_{e_7-e_6},\, X_{2e_2},\, X_{2e_3},\, X_{2e_4},\\
 X_{2e_7},\,X_{e_3+e_2},\, X_{e_4+e_2},\, X_{e_7+e_2},\\
X_{e_4+e_3},\, X_{e_7+e_3},\, X_{e_7+e_4}\\
 \end{array}\right\}\\
 \end{array}
$$

\subsection{$SDT \longleftrightarrow link\ pattern$}\label{a3} We draw a symmetric link
pattern for any given two-column SDT. The dominoes in $T$ are
labelled by integers from $1$ to $n$, while the vertices in a
symmetric link pattern $L_T$ are labelled by integers  $-n,\ldots,-1,1,\ldots, n$. So, the identification between the
dominoes of $T$ and the positive part of the link pattern is straightforward.

The integers of $ T_2$ are the positive right ends
 of the arcs in $L_T$.  Starting from $b_1$ and continuing to $b_k$ we connect 
 each $b_i$ to the closest  still free vertex
on its left, and complete at each step the negative part symmetrically.

This algorithm produces a maximal link pattern, that is, without
crossings and without fix points under the arcs. 
Moreover, $(-i,i)\in L_T$  if and only if $[i:i]$ is a horizontal domino of $T$.

For example, for $T$ presented in the previous example, we get
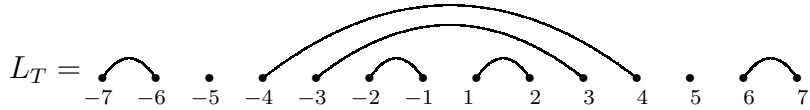
\begin{center}
\begin{picture}(200,50)
\put(-60,10){$L_T=$} \multiput(-25,10)(20,0){14} {\circle*{3}}
 \put(-32,2){${}_{-7}$} %-25
 \put(-12,2){${}_{-6}$} % -5
 \put(8,2){${}_{-5}$}   % 15
 \put(28,2){${}_{-4}$}  % 35
 \put(48,2){${}_{-3}$}  % 55
 \put(68,2){${}_{-2}$}  % 75
 \put(88,2){${}_{-1}$}  % 95
 \put(110,2){${}_{1}$}  %115
 \put(135,2){${}_{2}$}  %135
 \put(155,2){${}_{3}$}  %155
 \put(175,2){${}_{4}$}  %175
 \put(195,2){${}_5$}        %195
 \put(215,2){${}_6$}        %215
 \put(235,2){${}_7$}        %235

 \qbezier(75,10)(85,25)(95,10)      %-2->-1
 \qbezier(115,10)(125,25)(135,10) %1->2
 \qbezier(-25,10)(-15,25)(-5,10)    %-7->-6
 \qbezier(215,10)(225,25)(235,10)   %6->7
 \qbezier(55,10)(105,50)(155,10)    %-3->3
 \qbezier(35,10)(105,65)(175,10)    % -4->4

\end{picture}
\end{center}
Note that this algorithm provides a bijection between maximal link patterns and STD with two colmns. Indeed, given a maximal 
$L\in {\bf SLP}_{2n}(k)$ we construct STD $T(L)$ of shape $(2^{Ep^+(L)},1^{2n-2(Ep^+(L)})$ by taking the right positive ends of $L$ as elements of $(T(L))_2$. We take $[j:j]$ as a vertical domino if $(i,j)\in L$ and as a horisontal domino if $(-j,j)\in L.$ We complete $T(L)$ by filling still empty cells of $T_1$ with vertical dominoes containing the left positive ends and $Fp^+(L)$ in increasing from top to bottom order. Obviously, $L_{T(L)}=L.$ 

\subsection{A dense $\B_{2n}-$orbit in an orbital variety}\label{a4} Recall from \ref{1.4} that  $L\in {\bf SLP}_{2n}$ defines  
$X_L=\sum\limits_{\alpha\in L}X_\alpha$.
For $L_T$ constructed in \ref{a3} put $X_T:=X_{L_T}$ and let $\mathcal B_T$ be its $\B_{2n}-$orbit.
The connection between $w_T$ defined in \ref{a2} and $L_T$ defined
in \ref{a3} is as follows:
\begin{prop} Let $T$ be a SDT of shape $(2^k,1^{2n-2k})$. Then $\mathcal B_T$ is dense in ${\mathcal V}_T.$
\end{prop}
\begin{proof}
As we have explained in \ref{a1} $\dim\mathcal B_T=\dim {\mathcal V}_T$ so that it is enough to check that
$$X_T\in \nil_{2n}\cap w_T(\nil_{2n}).$$
By computations in \ref{a2}
 $X_{e_j-e_i}\in \nil_{2n}\cap w_T(\nil_{2n})$ for any $j\in T_2$ and $i\in T_1$ such that $i<j$. 
This is in particular true for any $X_\alpha$ in $X_T$ of this type.

By the same computations $X_{2e_j}\in \nil_{2n}\cap w_T(\nil_{2n})$ for any $j\in T_2$. 
This is in particular true for a horizontal dominoes of $T$, so that it is true for any $X_\alpha$ in $X_T$ of this type.

Thus it is true for $X_T$ which is a sum of such root vectors.
\end{proof}

\subsection{Orbital varieties closures}\label{a5} Now we can apply the results of \ref{5.2} to the maximal link patterns in order to obtain the results on the closures of link patterns. 
Given $T$ of shape $(2^k,1^{2n-2k})$ and let $h_1<\ldots<h_p$ be the entries of its horizonal dominoes (put $h_p=0$ if there are no horizontal dominoes), 
$a_1<\ldots<a_q$ the entries of the vertical dominoes of the first column and $b_1<\ldots<b_r$ the entries of the vertical dominoes of the second column so that $k=2r+p$ and $2q+p=2n-k$. 
%Let $T_1=\{a_1,\ldots,a_q\}$ and $T_2=\{b_1,\ldots,b_r\}\cup\{h_1,\ldots h_p\}$. 
For each $b_s$ let $i_s$ be the corresponding element of the first column such that $(i_s,b_s)\in L_T$. 
Recall that for $h_s$ one has $(-h_s,h_s)\in L_T.$
%Put $L^+(T)=\{i\ :\ i\in \left\langle T_1\right\rangle,\ i+1\in \left\langle T_2\right\rangle\}$. Put $L^-(T)=\{-1\}$ if $h_1=1$ and $L^-(T)=\emptyset$ otherwise. 
%Let $L(T)=L^+(T)\cup L^-(T).$ For example $L(T)=\{1,6\}$ for $T$ from \ref{a2}.

Recall $E'(L)$ from \ref 5.2. Translating \cite[4.3]{Me2} to the language of SDT we get 
$E'(L_T)=\{(-h_p,h_p)\ :\ h_p\ne 0\}\cup \{(i_l,b_l)\ :\ b_l>h_p\ {\rm and}\ (l=r\ {\rm or}\ b_s-b_l\geq 2(s-l)\ \forall s\ l<s\leq r\}$.

For $T$ of shape $(2^k,1^{2n-2k})$ and for $b_s> h_p$ let $T\langle b_s\rangle$ be SDT obtained 
from $T$ by moving the vertical domino $[b_s:b_s]$ from the second column to the first column. 
$T\langle b_s\rangle$ is of shape $(2^{k-2},1^{2n-2k+4}).$ 
If $h_p\ne 0$ let $T\langle h_p\rangle$ be SDT obtained from $T$ by putting domino $[h_p:h_p]$ vertically into the first column. 
$T\langle h_p\rangle$ is of shape $(2^{k-1},1^{2n-2k+2})$

For the example from \ref{a2} one has $E'(L_T)=\{(-4,4),\ (6,7)\}$ and
$$T\langle 4\rangle=\vcenter{\halign{& \hfill#\hfill\cr
 \ssod\multispan{5}{\hrulefill}\cr
 \ssod\vb &\ \, 1\   &\vb &\  2\ \,  &\vb\cr
 \ssa\multispan{4}{\hrulefill}\cr
 \ssal\vbs & &\, 3 & & \vbs\cr
 \ssal\multispan{4}{\hrulefill}\cr
 \ssod\vb &\ \,  4 &\vb &\  7\ \,& \vb\cr
 \ssa\multispan{4}{\hrulefill}\cr
 \ssa\vb &\ \, 5\  &\vb\cr
 \vsa\multispan{3}{\hrulefill}\cr
 \ssa\vb &\ \, 6\   &\vb\cr
 \ssa\multispan{3}{\hrulefill}\cr
 }}\quad{\rm and}\quad
 T\langle 7\rangle=\vcenter{\halign{& \hfill#\hfill\cr
 \ssod\multispan{4}{\hrulefill}\cr
 \ssod\vb &\ \, 1 &\vb &\  2\ \, &\vb\cr
 \ssa\multispan{4}{\hrulefill}\cr
 \ssal\vbs & &\, 3 & & \vbs\cr
 \vsal\multispan{4}{\hrulefill}\cr
 \ssa\vbs & & 4 & & \vbs\cr
 \ssal\multispan{4}{\hrulefill}\cr
 \ssod\vb &\ \ 5\  &\vb\cr
 \ssal\multispan{3}{\hrulefill}\cr
 \vb &\ \ 6\ &\vb\cr
 \ssa\multispan{3}{\hrulefill}\cr
 \vb &\ \ 7\ &\vb\cr
 \ssa\multispan{3}{\hrulefill}\cr
 }}
 $$
Applying \ref{5.2} to the maximal link patterns and translating this into the language of SDT's we get:
\begin{thm}
Let $T$ be of shape $(2^k,1^{2n-2k})$ and let $L:=L_T.$
One has
 $$\overline{\mathcal V}_T\cap{\mathcal O}_{(2^{k-1},1^{2n-2k+2})}=
				\left\{\begin{array}{ll}\bigcup\limits_{(i_l,b_l)\in E'(L_T)}\overline{\mathcal B}_{L^-_{(i_l,b_l)}(-b_l,b_l)}\cap{\mathcal O}_{(2^{k-1},1^{2n-2k+2})}&{\rm if}\ h_p=0;\\
  \mathcal{V}_{T\langle h_p\rangle}\bigcup\bigcup\limits_{(i_l,b_l)\in E'(L_T)}\overline {\mathcal B}_{L^-_{(i_l,b_l)}(-b_l,b_l)}\cap{\mathcal O}_{(2^{k-1},1^{2n-2k+2})}&{\rm otherwise};\\
  \end{array}
  \right.$$
In particular this intersection is irreducible iff $|E'(L_T)|=1$ and contains an orbital variety iff $h_p\ne 0.$ 
Moreover, this intersection is an orbital variety iff $b_r<h_p.$

One also has
$${\mathcal V}_T\cap{\mathcal O}_{(2^{k-2},1^{2n-2k+4})}=\left\{\begin{array}{ll}\bigcup\limits_{(i_l,b_l)\in E'(L_T)}\mathcal V_{T\langle b_l\rangle}&{\rm if}\ h_p=0;\\
 \bigcup\limits_{(i_l,b_l)\in E'(L_T)}\mathcal V_{T\langle b_l\rangle}\cup (\overline{\mathcal V}_{T\langle h_p\rangle}\cap{\mathcal O}_{(2^{k-2},1^{2n-2k+4})})&{\rm otherwise};\\
 \end{array}\right.$$
Moreover, this intersection is a union of orbital varieties iff either $h_p=0$ or 
$E'(L_T)=\{(-h_p,h_p)\}$ and $E'(L_{T\langle h_p\rangle})=\{(-h_p+1,h_p-1)\}$, that is if either there are no horizontal dominoes in SDT or both 
%$h_p$ and $h_p-1$ 
$[h_p:h_p]$ and $[h_p-1:h_p-1]$ 
are horizontal dominoes such that $b_q<h_p-1.$
\end{thm}

This means in particular, that all the intersections $\mathcal V_T\cap{\mathcal O}_{(2^{k-m},1^{2n-2k+2m})}$ 
are irreducible and consists of orbital varieties only iff $T_2=\{h_1,\ldots,h_p\}$ that is if $\mathcal V_T$ belongs to an abelian nilradical.

This also means that if $h_p=0,$ that is if $\mathcal V_T$ is induced from the corresponding orbital variety in $\mathfrak{sl}_n$ 
then all the intersections with ${\mathcal O}_{(2^{k-2m},1^{2n-2k+4m})}$ are a union of orbital varieties. 
And all the intersections with ${\mathcal O}_{(2^{k-2m+1},1^{2n-2k+4m-2})}$
do not contain orbital varieties and are non-equidimensional in general.

In all other cases the intersections are non-equidimensional and contain both orbital varieties and varieties of smaller dimensions.

In our example:
$$\begin{array}{ll}\overline{\mathcal V}_T\cap\mathcal O_{(2^5,1^4)}&=\mathcal{V}_{T\langle 4\rangle}\cup(\overline{\mathcal B}_{X_{e_2-e_1}+X_{2e_3}+X_{2e_4}+X_{2e_7}}\cap\mathcal O_{(2^5,1^4)};\\
\overline{\mathcal V}_T\cap\mathcal O_{(2^4,1^6)}&=\mathcal{V}_{T\langle 7\rangle}\cup\mathcal{V}_{T\langle 3,4 \rangle}\cup(\overline{\mathcal B}_{X_{e_2-e_1}+X_{2e_3}+X_{2e_7}}\cap\mathcal O_{(2^4,1^6)};\\
\overline{\mathcal V}_T\cap\mathcal O_{(2^3,1^8)}&=\mathcal V_{T\langle 4,7\rangle}\cup((\overline{\mathcal B}_{X_{e_2-e_1}+X_{2e_7}}\cup\overline{\mathcal B}_{X_{2e_2}+X_{e7-e_6}})\cap \mathcal O_{(2^3,1^8)});\\
\overline{\mathcal V}_T\cap\mathcal O_{(2^2,1^{10})}&=\mathcal V_{T\langle 3,4,7\rangle}\cup (\overline{\mathcal B}_{X_{2e_2}+X_{2e_7}}\cap\mathcal O_{(2^2,1^{10})})\\
\overline{\mathcal V}_T\cap\mathcal O_{(2,1^{12})}&=\overline{\mathcal B}_{X_{2e_2}}\cap\mathcal O_{(2,1^{12})}\\
\end{array}$$

\subsection{Orbital varieties intersections of codimension 1}\label{a6} Another question is to describe pairs of orbital varieties $\mathcal V_T,\mathcal V_S$ in
$\mathcal{O}_{ (2^k,1^{2n-2k})}\cap\nil_{2n}$ such that
$$\codim_{\mathcal{O}_{(2^k,1^{2n-2k})}\cap\nil_{2n}}\mathcal V_T\cap \mathcal V_S=1.$$

Unlike the case of the closure, here the picture is very similar to the picture in ${\mathfrak{sl}}_n$.
Exactly as there we consider all $\B_{2n}-$orbits $\mathcal B_L$ in $\mathcal O_{ (2^k,1^{2n-2k})}\cap\nil_{2n}$ of codimension 1 
in it and show that for any such $\mathcal B_L$ but one there are exactly two maximal $\B_{2n}-$orbit closures containing it 
and that the intersection of these two closures is exactly $\overline{\mathcal B}_L.$ And the only remaining $\mathcal B_L$ belongs to a unique maximal $\B_{2n}-$orbit closure.

One has $\codim_{{\mathcal O}_{ (2^k,1^{2n-2k})}\cap\nil_{2n}}\mathcal B_L=1$ in one of the following cases. 
In each case $L_0$ denotes the "maximal" part of $L$, so that $b(L_0)=c(L_0)=0$. 
We also illustrate each case with link patterns, where the part corresponding to $L_0$ is denoted by dots.
\begin{itemize}
\item[(i)] There are two symmetric fixed points either both under one bridge or under two separate bridges. 
Namely, $L=L_0(-j,j),$ or $L=L_0(j,k)(-j,-k)$. 
In both cases $c(L)=0$ and $b(L)=1$.\\
(a) The two points $-i,i$ are under the central arc $(-j,j)$. 
In this case the only maximal $B-$orbit with the closure containing $\mathcal B_L$ is $\mathcal B_{L'}$
where $L'=L_0(-i,i)$;\\
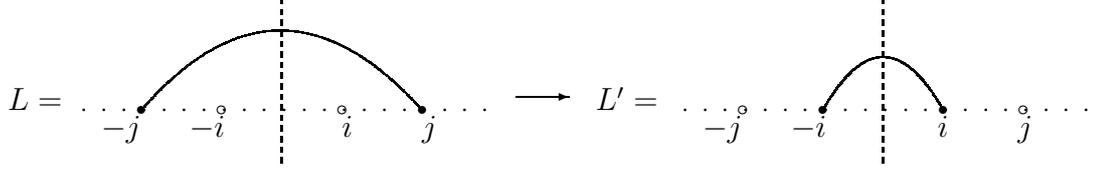
\begin{center}
\begin{picture}(100,70)
\put(-160,40){$L=$}
\multiput(-131.5,40)(7.5,0){21}%
{\circle*{1}}
 \put(-110,40){\circle*{3}}
 \put(-5,40){\circle*{3}}
 \put(-80,40){\circle{3}}
 \put(-35,40){\circle{3}}
  \put(-125,30){$-j$}
 \put(-92,30){$-i$}
 \multiput(-57.5,20)(0,4){16}{\line(0,1){2}}
\put(-35,30){$i$}
\put(-5,30){$j$}

 \qbezier(-110,40)(-58.5,100)(-5,40)

 \put(30,45){\vector(2,0){20}}
\put(60,40){$L'=$}

\multiput(93.5,40)(7.5,0){21}%
{\circle*{1}}
 \put(115,40){\circle{3}}
 \put(145,40){\circle*{3}}
 \put(190,40){\circle*{3}}
 \multiput(167.5,20)(0,4){16}{\line(0,1){2}}
 \put(220,40){\circle{3}}
 \put(100,30){$-j$}
 \put(133,30){$-i$}
 \put(188,30){$i$}
\put(218,30){$j$}
 \qbezier(145,40)(167.5,80)(190,40)
\end{picture}
\end{center}
(b) The two points $-i,i$ are under arcs $(-k,-j),\ (j,k)$ respectively. 
Note that since the codimension is one, one has $0<j<k$. 
The two maximal $\B-$orbits with the closures containing $\mathcal B_L$ are $\mathcal B_{L_1}$ where $L_1=L_0(i,k)(-i,-k)$ and 
$\mathcal B_{L_2}$ where $L_2=L_0(j,i)(-j,-i)$. 
Note that considering link patterns as link patterns in $\mathfrak{sl}_{2n}$ we get that 
$\overline{\mathfrak B}(Y_L)=\overline{\mathfrak B}(Y_{L_1})\cap\overline{\mathfrak B}(Y_{L_2})$ so that even in $\mathfrak{sl}_{2n}$ 
the corresponding intersection is irreducible (and of codimension 2). Thus, the intersection in $\mathfrak{sp}_{2n}$ is also irreducible.
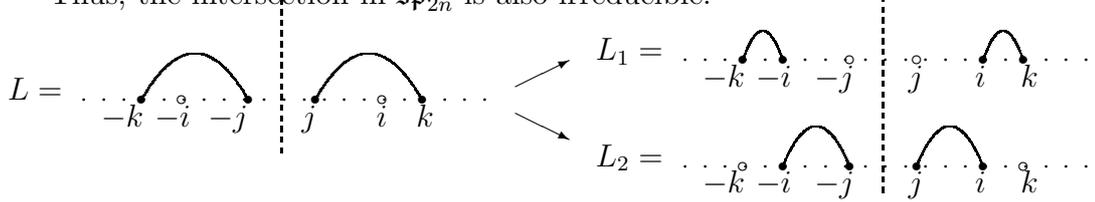
\begin{center}
\begin{picture}(100,70)
\put(-160,40){$L=$}
\multiput(-131.5,40)(7.5,0){21}%
{\circle*{1}}
 \put(-110,40){\circle*{3}}
 \put(-95,40){\circle{3}}
 \put (-70,40){\circle*{3}}
 \put(-125,30){$-k$}
 \put(-105,30){$-i$}
 \put (-85,30){$-j$}
 \multiput(-57.5,20)(0,4){16}{\line(0,1){2}}
 \put (-45,40){\circle*{3}}
 \put(-20,40){\circle{3}}
 \put(-5,40){\circle*{3}}
 \put(-50,30){$j$}
\put(-22,30){$i$}
\put(-7,30){$k$}
 \qbezier(-110,40)(-90,75)(-70,40)
\qbezier(-45,40)(-25,75)(-5,40)

 \put(30,45){\vector(2,1){20}}
\put(60,55){$L_1=$}

\multiput(93.5,55)(7.5,0){21}%
{\circle*{1}}
 \put(115,55){\circle*{3}}
 \put(130,55){\circle*{3}}
 \put(155,55){\circle{3}}
 \put(100,45){$-k$}
 \put(120,45){$-i$}
 \put(142,45){$-j$}
 \multiput(167.5,5)(0,4){22}{\line(0,1){2}}
 \put(180,55){\circle{3}}
 \put(205,55){\circle*{3}}
 \put(220,55){\circle*{3}}
 \put(177,45){$j$}
  \put(202,45){$i$}
 \put(219,45){$k$}

 \qbezier(115,55)(122.5,77)(130,55)
\qbezier(205,55)(212.5,77)(220,55)

 \put(30,35){\vector(2,-1){20}}
 \put(60,15){$L_2=$}

 \multiput(93.5,15)(7.5,0){21}%
{\circle*{1}}
\put(115,15){\circle{3}}
 \put(130,15){\circle*{3}}
 \put(155,15){\circle*{3}}
 \put(100,5){$-k$}
 \put(120,5){$-i$}
 \put(142,5){$-j$}
 %\multiput(167.5,5)(0,4){22}{\line(0,1){2}}
 \put(180,15){\circle*{3}}
 \put(205,15){\circle*{3}}
 \put(220,15){\circle{3}}
 \put(177,5){$j$}
  \put(202,5){$i$}
 \put(219,5){$k$}

 \qbezier(130,15)(142.5,45)(155,15)
\qbezier(180,15)(192.5,45)(205,15)
\end{picture}
\end{center}

\item[(ii)] There is a cross in the center, that is $L=L_0(i,-j)(-i,j)$ 
and $b(L)=0,$ $c(L)=1$. Let $L_1=L_0(i,j)(-i,-j)$ and $L_2=L_0(-i,i)(-j,j)$. 
Note again that considering link patterns as link patterns in $\mathfrak{sl}_{2n}$ we get that 
$\overline{\mathfrak B}(Y_L) = \overline{\mathfrak B}(Y_{L_1})\cap\overline{\mathfrak B}(Y_{L_2})$ so that even in $\mathfrak{sl}_{2n}$ 
the corresponding intersection is of codimension 1 and irreducible. Thus, the intersection is irreducible.

    
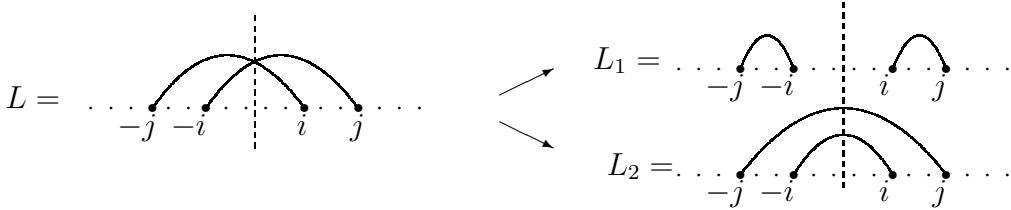
\begin{center}
\begin{picture}(100,90)
\put(-165,35){$L=$}
\multiput(-133,35)(7.5,0){3}%
{\circle*{1}}
 \put(-110,35){\circle*{3}}
 \put(-123,25){$-j$}
 \multiput(-104,35)(7.5,0){3}%
{\circle*{1}}
 \put(-90,35){\circle*{3}}
 \put(-103,25){$-i$}
 \multiput(-71.5,20)(0,4){13}{\line(0,1){2}}
\multiput(-82.5,35)(7.5,0){4}%
{\circle*{1}}
 \put(-53,35){\circle*{3}}
 \put(-56,25){$i$}
 \put(-33,35){\circle*{3}}
 \multiput(-47,35)(7.5,0){3}%
{\circle*{1}}
 \put(-36,25){$j$}
\multiput(-25,35)(7.5,0){3}%
{\circle*{1}}
 \qbezier(-110,35)(-82.5,75)(-53,35)
\qbezier(-90,35)(-62.5,75)(-33,35)

 \put(20,40){\vector(2,1){20}}
 \put(55,50){$L_1=$}

\multiput(87,50)(7.5,0){3}%
{\circle*{1}}
 \put(110,50){\circle*{3}}
 \put(97,40){$-j$}
 \put(130,50){\circle*{3}}
 \put(117,40){$-i$}
\multiput(116,50)(7.5,0){2}%
{\circle*{1}}
\multiput(137.5,50)(7.5,0){4}%
{\circle*{1}}
 \put(167,50){\circle*{3}}
 \put(162,40){$i$}
 \put(187,50){\circle*{3}}
 \put(182,40){$j$}
 \multiput(173,50)(7.5,0){2}%
{\circle*{1}}
\multiput(195,50)(7.5,0){3}%
{\circle*{1}}
 \qbezier(110,50)(120,75)(130,50)
\qbezier(167,50)(177,75)(187,50)

 \put(20,30){\vector(2,-1){20}}
\put(60,10){$L_2=$}
\multiput(87,10)(7.5,0){3}%
{\circle*{1}}
 \put(110,10){\circle*{3}}
 \put(97,0){$-j$}
 \put(130,10){\circle*{3}}
 \put(117,0){$-i$}
 \multiput(116,10)(7.5,0){2}%
{\circle*{1}}
\multiput(137.5,10)(7.5,0){4}%
{\circle*{1}}
\multiput(148.5,5)(0,4){18}{\line(0,1){2}}
 \put(167,10){\circle*{3}}
 \put(162,0){$i$}
 \put(187,10){\circle*{3}}
 \put(182,0){$j$}
 \multiput(173,10)(7.5,0){2}%
{\circle*{1}}
\multiput(195,10)(7.5,0){3}%
{\circle*{1}}
 \qbezier(110,10)(148.5,60)(187,10)
\qbezier(130,10)(148.5,40)(167,10)
\end{picture}
\end{center}

\item[(iii)] There are two symmetric crosses so that either\\ $L=L_0(i,k)(-i,-k)(j,l)(-j,-l)$ or $L=L_0(-j,j)(i,k)(-i,-k)$ where $i<j<k$, that is $b(L)=0$ and $c(L)=1.$ 
As in (i) we consider them separately:\\
(a) $L=L_0(i,k)(-i,-k)(j,l)(-j,-l)$ that is both roots are induced form $\mathfrak{sl}_n.$ 
In this case there are two maximal  $\B_{2n}-$orbits with the closures containing $\mathcal B_L$ namely
$\mathcal B_{L_1}$ where $L_1=L_0(i,l)(-i,-l)(j,k)(-j,-k)$ and   $\mathcal B_{L_2}$ where $L_2=L_0(i,j)(-i,-j)(k,l)(-k,-l)$. 
Exactly as in case (i-b), considering link patterns as link patterns in $\mathfrak{sl}_{2n}$ 
we get that $\overline{\mathfrak B}(Y_L)=\overline{\mathfrak B}(Y_{L_1})\cap\overline{\mathfrak B}(Y_{L_2})$ so that even in $\mathfrak{sl}_{2n}$ 
the corresponding intersection is of codimension 2 and irreducible. Thus, the intersection is irreducible.\\

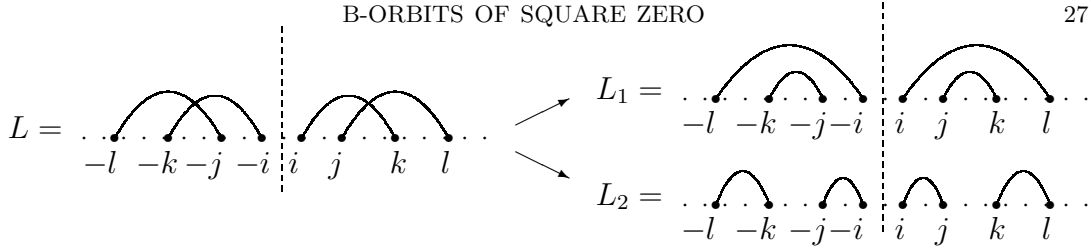
\begin{center}
\begin{picture}(100,70)
\put(-160,40){$L=$}
\multiput(-131.5,40)(7.5,0){21}%
{\circle*{1}}
 \put(-120,40){\circle*{3}}
 \put(-100,40){\circle*{3}}
 \put(-80,40){\circle*{3}}
 \put (-65,40){\circle*{3}}
 \put(-132,27){$-l$}
 \put(-112,27){$-k$}
 \put (-94,27){$-j$}
 \put (-75,27){$-i$}
 \qbezier(-120,40)(-100,75)(-80,40)
 \qbezier(-100,40)(-82.5,72)(-65,40)
 \multiput(-57.5,20)(0,4){16}{\line(0,1){2}}
 \put (-50,40){\circle*{3}}
 \put(-35,40){\circle*{3}}
 \put(-15,40){\circle*{3}}
 \put(5,40){\circle*{3}}
 \put(-55,27){$i$}
\put(-40,27){$j$}
\put(-17,27){$k$}
\put(2,27){$l$}
 \qbezier(-50,40)(-32.5,72)(-15,40)
 \qbezier(-35,40)(-15,75)(5,40)
%\qbezier(-45,40)(-25,75)(-5,40)

 \put(30,45){\vector(2,1){20}}
\put(60,55){$L_1=$}

\multiput(93.5,55)(7.5,0){21}%
{\circle*{1}}
\put(105,55){\circle*{3}}
 \put(125,55){\circle*{3}}
 \put(145,55){\circle*{3}}
 \put(160,55){\circle*{3}}
 \put(92,42){$-l$}
 \put(112,42){$-k$}
 \put(132,42){$-j$}
 \put(147,42){$-i$}
 \qbezier(105,55)(132.5,95)(160,55)
 \qbezier(125,55)(135,75)(145,55)
 \multiput(167.5,5)(0,4){22}{\line(0,1){2}}
 \put(175,55){\circle*{3}}
 \put(190,55){\circle*{3}}
 \put(210,55){\circle*{3}}
 \put(230,55){\circle*{3}}
 \put(172,42){$i$}
  \put(187,42){$j$}
 \put(207,42){$k$}
\put(227,42){$l$}

\qbezier(175,55)(202.5,95)(230,55)
\qbezier(190,55)(200,75)(210,55)

 \put(30,35){\vector(2,-1){20}}
 \put(60,15){$L_2=$}

 \multiput(93.5,15)(7.5,0){21}%
{\circle*{1}}
\put(105,15){\circle*{3}}
 \put(125,15){\circle*{3}}
 \put(145,15){\circle*{3}}
 \put(160,15){\circle*{3}}
 \put(92,2){$-l$}
 \put(112,2){$-k$}
 \put(132,2){$-j$}
 \put(147,2){$-i$}
 \qbezier(105,15)(115,40)(125,15)
\qbezier(145,15)(152.5,35)(160,15)
 %\multiput(167.5,5)(0,4){22}{\line(0,1){2}}
 \put(175,15){\circle*{3}}
 \put(190,15){\circle*{3}}
 \put(210,15){\circle*{3}}
 \put(230,15){\circle*{3}}
 \put(172,2){$i$}
  \put(187,2){$j$}
 \put(207,2){$k$}
\put(227,2){$l$}

 \qbezier(175,15)(182.5,35)(190,15)
\qbezier(210,15)(220,40)(230,15)
\end{picture}
\end{center}
(b) $L=L_0(-j,j)(i,k)(-i,-k)$. In this case there are two maximal $\B_{2n}-$orbits with the closures containing $\mathcal B_L$ 
namely $\mathcal B_{L_1}$ where $L_1=L_0(-i,i)(j,k)(-j,-k)$ and $\mathcal B_{L_2}$ where $L_2=L_0(i,j)(-i,-j)(-k,k)$. 
In this case $\overline{\mathfrak B}(Y_{L_1})\cap\overline{\mathfrak B}(Y_{L_2})=\overline{\mathfrak B}(Y_L)\cup\overline{\mathfrak B}(Y_{L'})\cup\overline{\mathfrak B}(Y_{L''})$ 
where $L'=L_0(-k,i)(-j,k)(-i,j)$ and $L''=L_0(-k,j)(-j,i)(-i,k)$. 
Note that all 3 components are of codimension 2 but only $\mathfrak B(Y_L)$ is symmetric. 
Moreover the maximal symmetric $L_S'\prec L'$ is $L_S'=L_0(i,-k)(-i,k)(-j,j)$ 
and the maximal symmetric element  $L_S''\prec L''$ is also $L_S'.$ 
Further note that $L_S'=L_S'' \prec L$.  
Thus, again the intersection $\overline{\mathcal B}_{L_1}\cap\overline{\mathcal B}_{L_2}=\overline{\mathcal B}_L$.\\
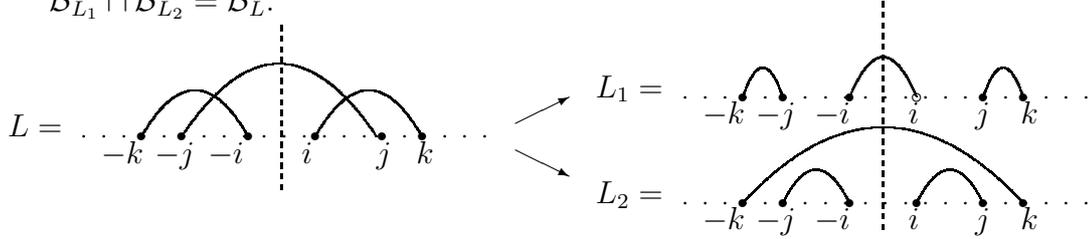
\begin{center}
\begin{picture}(100,70)
\put(-160,40){$L=$}
\multiput(-131.5,40)(7.5,0){21}%
{\circle*{1}}
 \put(-110,40){\circle*{3}}
 \put(-95,40){\circle*{3}}
 \put (-70,40){\circle*{3}}
 \put(-125,30){$-k$}
 \put(-105,30){$-j$}
 \put (-85,30){$-i$}
 \multiput(-57.5,20)(0,4){16}{\line(0,1){2}}
 \put (-45,40){\circle*{3}}
 \put(-20,40){\circle*{3}}
 \put(-5,40){\circle*{3}}
 \put(-50,30){$i$}
\put(-22,30){$j$}
\put(-7,30){$k$}
 \qbezier(-110,40)(-90,75)(-70,40)
 \qbezier(-95,40)(-58.5,95)(-22,40)

\qbezier(-45,40)(-25,75)(-5,40)

 \put(30,45){\vector(2,1){20}}
\put(60,55){$L_1=$}

\multiput(93.5,55)(7.5,0){21}%
{\circle*{1}}
 \put(115,55){\circle*{3}}
 \put(130,55){\circle*{3}}
 \put(155,55){\circle*{3}}
 \put(100,45){$-k$}
 \put(120,45){$-j$}
 \put(142,45){$-i$}
 \multiput(167.5,5)(0,4){22}{\line(0,1){2}}
 \put(180,55){\circle{3}}
 \put(205,55){\circle*{3}}
 \put(220,55){\circle*{3}}
 \put(177,45){$i$}
  \put(202,45){$j$}
 \put(219,45){$k$}

 \qbezier(115,55)(122.5,77)(130,55)
\qbezier(205,55)(212.5,77)(220,55)
\qbezier(155,55)(167.5,85)(180,55)

 \put(30,35){\vector(2,-1){20}}
 \put(60,15){$L_2=$}

 \multiput(93.5,15)(7.5,0){21}%
{\circle*{1}}
\put(115,15){\circle*{3}}
 \put(130,15){\circle*{3}}
 \put(155,15){\circle*{3}}
 \put(100,5){$-k$}
 \put(120,5){$-j$}
 \put(142,5){$-i$}
 %\multiput(167.5,5)(0,4){22}{\line(0,1){2}}
 \put(180,15){\circle*{3}}
 \put(205,15){\circle*{3}}
 \put(220,15){\circle*{3}}
 \put(177,5){$i$}
  \put(202,5){$j$}
 \put(219,5){$k$}

 \qbezier(130,15)(142.5,40)(155,15)
\qbezier(180,15)(192.5,40)(205,15)
\qbezier(115,15)(167.5,72)(220,15)
\end{picture}
\end{center}

The link patterns $L',\ L'',\ L_S'=L_S''$ are:
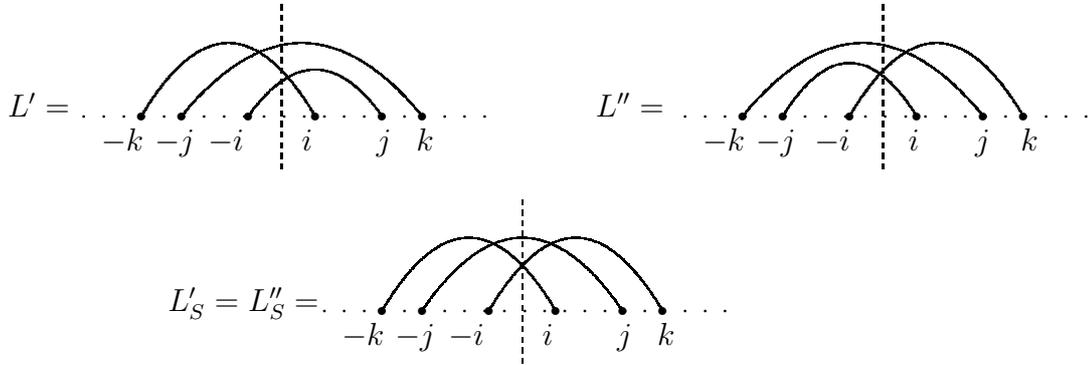
\begin{center}
\begin{picture}(100,80)(0,10)
\put(-160,40){$L'=$}
\multiput(-131.5,40)(7.5,0){21}%
{\circle*{1}}
 \put(-110,40){\circle*{3}}
 \put(-95,40){\circle*{3}}
 \put (-70,40){\circle*{3}}
 \put(-125,27){$-k$}
 \put(-105,27){$-j$}
 \put (-85,27){$-i$}
 \multiput(-57.5,20)(0,4){16}{\line(0,1){2}}
 \put (-45,40){\circle*{3}}
 \put(-20,40){\circle*{3}}
 \put(-5,40){\circle*{3}}
 \put(-50,27){$i$}
\put(-22,27){$j$}
\put(-7,27){$k$}
 \qbezier(-110,40)(-77.5,95)(-45,40)
 \qbezier(-95,40)(-50,95)(-5,40)
\qbezier(-70,40)(-45,75)(-20,40)

\put(60,40){$L''=$}

\multiput(93.5,40)(7.5,0){21}%
{\circle*{1}}
 \put(115,40){\circle*{3}}
 \put(130,40){\circle*{3}}
 \put(155,40){\circle*{3}}
 \put(100,27){$-k$}
 \put(120,27){$-j$}
 \put(142,27){$-i$}
 \multiput(167.5,20)(0,4){16}{\line(0,1){2}}
 \put(180,40){\circle*{3}}
 \put(205,40){\circle*{3}}
 \put(220,40){\circle*{3}}
 \put(177,27){$i$}
  \put(202,27){$j$}
 \put(219,27){$k$}

 \qbezier(115,40)(160,95)(205,40)
\qbezier(130,40)(155,80)(180,40)
\qbezier(155,40)(187.5,95)(220,40)

\end{picture}
\end{center}
\begin{center}
\begin{picture}(100,80)(-90,0)
\put(-190,40){$L_S'=L_S''=$}
\multiput(-131.5,40)(7.5,0){21}%
{\circle*{1}}
 \put(-110,40){\circle*{3}}
 \put(-95,40){\circle*{3}}
 \put (-70,40){\circle*{3}}
 \put(-125,27){$-k$}
 \put(-105,27){$-j$}
 \put (-85,27){$-i$}
 \multiput(-57.5,20)(0,4){16}{\line(0,1){2}}
 \put (-45,40){\circle*{3}}
 \put(-20,40){\circle*{3}}
 \put(-5,40){\circle*{3}}
 \put(-50,27){$i$}
\put(-22,27){$j$}
\put(-7,27){$k$}
 \qbezier(-110,40)(-77.5,95)(-45,40)
 \qbezier(-95,40)(-57.5,95)(-20,40)
\qbezier(-70,40)(-37.5,95)(-5,40)
\end{picture}
\end{center}
\end{itemize}

As a corollary we get that
\begin{thm} Let ${\mathcal V}_T,\ {\mathcal V}_S\subset {\mathcal O}_{(2^k,1^{2n-2k})}\cap\nil_{2n}$ be such that
${\rm codim}_{{\mathcal O}_{(2^k,1^{2n-2k})}\cap\nil_{2n}}{\mathcal V}_T\cap{\mathcal V}_S=1$  then ${\mathcal V}_T\cap{\mathcal V}_S$
 is irreducible.
\end{thm}
\begin{rem}
It is easy to give a combinatorial criteria for ${\mathcal V}_T,\ {\mathcal V}_S$ to intersect in codimension 1, it is very close to the criteria provided in
\cite[5.6]{Me2}, however it involves meanders, so we are not providing it here.
\end{rem}
%\newpage
\appendix
\section{ Combinatorics of link patterns: restriction to symmetric link patterns}
To construct $D(L)$ we consider $L^{\prime\prime}\in \widehat D(L)$ obtained by one of  elementary moves described in \ref{5.3}. We always assume that an arc and a fixed point (respectively a pair of arcs) involved in obtaining $L^{\prime\prime}$ satisfy the conditions of minimality.
In order to simplify the notation we set $L_0$ exactly as in \ref{5.3} to be a symmetric link pattern obtained from $L$ by removing the pairs of symmetric arcs
such that one arc of the pair was  involved (and removing a central arc if it was involved) in getting $L^{\prime\prime}$ .

All the non central arcs are written as $(\pm i,\pm j)$ where $0<i< j$. In the case of central arc we write $(-i,i)$ where $i>0.$

{\it $\rightleftharpoons$ Move of an end to a fixed point} (F).
We always choose a fixed point involved in getting $L^{\prime\prime}$ to be positive. Obviously, we always can make this choice by the symmetry.

\begin{itemize}
\item[(FC)] {\it The arc is central:}  The arc is $(-i,i)$ and $r_i>i$ is the corresponding fixed point.
Then $L^{\prime\prime}=L_{0}(-i, r_i)$ and $SR_{L^{\prime\prime}}=R_{L'}$ where $L'=L^-_{(-i,i)}(-r_i,r_i)$ and 
$\dim \mathcal B_{L'}=\dim\mathcal B_L-1\ \Rightarrow\ L'\in D(L).$ 
One has $X_{L'}=\lim\limits_{m\rightarrow\infty} T_i(\frac{1}{m}).U_{e_{r_i}-e_i}(1).X_L$.
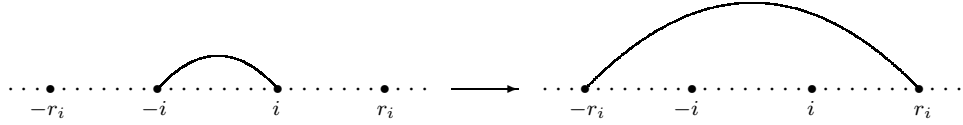
\begin{center}
\begin{picture}(100,45)
\multiput(-135,10)(5,0){32}%
{\circle*{1}}
\put(-120,10){\circle*{3}}
 \put(-128,2){${}_{-r_i}$}
 \put(-80,10){\circle*{3}}
 \put(-86,2){${}_{-i}$}
 \put(-35,10){\circle*{3}}
 \put(-37,2){${}_{i}$}
 \put(5,10){\circle*{3}}
 \put(2,2){${}_{r_i}$}

 \qbezier(-80,10)(-57.5,35)(-35,10)
 \put(30,10){\vector(1,0){25}}
 \multiput(65,10)(5,0){32}
 {\circle*{1}}
\put(80,10){\circle*{3}}
\put(74,2){${}_{-r_i}$}
 \put(120,10){\circle*{3}}
 \put(113,2){${}_{-i}$}
 \put(165,10){\circle*{3}}
 \put(163,2){${}_{i}$}
 \put(205,10){\circle*{3}}
  \put(203,2){${}_{r_i}$}

 \qbezier(80,10)(142.5,75)(205,10)
\end{picture}
\end{center}

\item[(FP)]{\it The arc is in the positive part:}
The arc is $(i,j)$  and a fixed point is either $l_i$ or $r_j.$ These two cases are different.
\begin{itemize}
\item{(a)} Let a fixed point be $l_i$.
Then $L^{\prime\prime}=L^-_{(i,j)}(l_i,j)$ and $SR_{L^{\prime\prime}}=R_{L'}$ where $L'=L_0(l_i,j)(-l_i,-j)$ and $\dim \mathcal B_{L'}=
\dim\mathcal B_L-1\ \Rightarrow\ L'\in D(L).$\\ 
One has $X_{L'}=\lim\limits_{m\rightarrow\infty}T_i(m) . U_{e_{i}-e_{l_i}}(-1).X_L$.
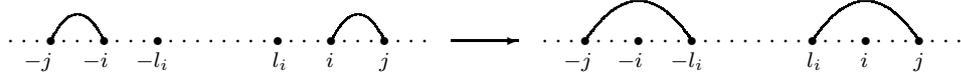
\begin{center}
\begin{picture}(100,45)(0,-10)
%\put(-20,35){$L_\sigma=$}
\multiput(-135,10)(5,0){32}%
{\circle*{1}}
\put(-120,10){\circle*{3}}
 \put(-130,2){${}_{-j}$}
 \put(-100,10){\circle*{3}}
 \put(-108,2){${}_{-i}$}
 \put(-80,10){\circle*{3}}
 \put(-88,2){${}_{-l_i}$}
 \put(-35,10){\circle*{3}}
 \put(-37,2){${}_{l_i}$}
 \put(-15,10){\circle*{3}}
 \put(-17,2){${}_{i}$}
 \put(5,10){\circle*{3}}
 \put(3,2){${}_{j}$}

 \qbezier(-120,10)(-110,30)(-100,10)
 \qbezier(-15,10)(-5,30)(5,10)

 \put(30,10){\vector(1,0){25}}
 \multiput(65,10)(5,0){32}
 {\circle*{1}}
\put(80,10){\circle*{3}}
\put(72,2){${}_{-j}$}
\put(100,10){\circle*{3}}
\put(92,2){${}_{-i}$}
 \put(120,10){\circle*{3}}
 \put(112,2){${}_{-l_i}$}
 \put(165,10){\circle*{3}}
 \put(163,2){${}_{l_i}$}
 \put(185,10){\circle*{3}}
 \put(183,2){${}_{i}$}
 \put(205,10){\circle*{3}}
  \put(203,2){${}_{j}$}

 \qbezier(80,10)(100,40)(120,10)
  \qbezier(165,10)(185,40)(205,10)

\end{picture}
\end{center}

\item{(b)} Let fixed point be $r_j$. Then $L^{\prime\prime}=L^-_{(i,j)}(i,r_j)$. There is no $L'$ such that $SR_{L^{\prime\prime}}$ is $R_{L'}.$
Two maximal SLPs $\{L'_1,\, L'_2\}$ such that $R_{L'_s}\preceq SR_{L^{\prime\prime}}$  are $L'_1=L_0(i,r_j)(-i,-r_j)$ and $L'_2=L_0(-j,j)(-r_j,r_j)$. Note that  $L\succ L_0(i,-j)(-i,j)\succ L'_2\ \Rightarrow\ L'_2\not\in D(L)$ and $\dim \mathcal B_{L'_1}=
\dim\mathcal B_L-1\ \Rightarrow\ L'_1\in D(L)$. 
One has $X_{L_1'}=\lim\limits_{m\rightarrow\infty}T_j(\frac{1}{m}). U_{e_{r_j}-e_j}(1).X_L$.
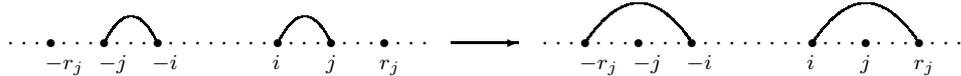
 \begin{center}
\begin{picture}(100,45)(0,-10)
\multiput(-135,10)(5,0){32}%
{\circle*{1}}
\put(-120,10){\circle*{3}}
 \put(-122,2){${}_{-r_j}$}
 \put(-100,10){\circle*{3}}
 \put(-102,2){${}_{-j}$}
 \put(-80,10){\circle*{3}}
 \put(-82,2){${}_{-i}$}
 \put(-35,10){\circle*{3}}
 \put(-37,2){${}_{i}$}
 \put(-15,10){\circle*{3}}
 \put(-17,2){${}_{j}$}
 \put(5,10){\circle*{3}}
 \put(3,2){${}_{r_j}$}

 \qbezier(-100,10)(-90,30)(-80,10)
 \qbezier(-35,10)(-25,30)(-15,10)

 \put(30,10){\vector(1,0){25}}
 \multiput(65,10)(5,0){32}
 {\circle*{1}}
\put(80,10){\circle*{3}}
\put(78,2){${}_{-r_j}$}
\put(100,10){\circle*{3}}
\put(98,2){${}_{-j}$}
 \put(120,10){\circle*{3}}
 \put(118,2){${}_{-i}$}
 \put(165,10){\circle*{3}}
 \put(163,2){${}_{i}$}
 \put(185,10){\circle*{3}}
 \put(183,2){${}_{j}$}
 \put(205,10){\circle*{3}}
  \put(203,2){${}_{r_j}$}

 \qbezier(80,10)(100,40)(120,10)
  \qbezier(165,10)(185,40)(205,10)

\end{picture}
\end{center}
\end{itemize}

\item[(FM)]{\it One end of the arc is in the negative part and one end in the positive:} The arc is either $(-i,j)$ or $(i,-j)$ and a fixed point on the right of an arc is $r_j$ in the first case and $r_i$ in the second one. We consider both cases:
    \begin{itemize}
\item{(a)}  $(-i,j)$ and  $r_j$. Then $L^{\prime\prime}=L^-_{(-i,j)}(-i,r_j)$. There is no $L'$ such that $SR_{L^{\prime\prime}}$ is $R_{L'}.$
Two maximal SLPs $\{L'_1,\, L'_2\}$ such that  $R_{L'_s}\preceq SR_{L^{\prime\prime}}$  are $L'_1=L_0(-i,r_j)(i,-r_j)$ and $L'_2=L_0(j,-j)(r_j,-r_j)$.\\ 
One has $\dim \mathcal B_{L'_1}= \dim\mathcal B_L-1\ \Rightarrow\ L_1'\in D(L)$.
For $L'_2$ one has $\dim \mathcal B_{L'_2}=\dim \mathcal B_L-1$ iff  $r_i=r_j$. Moreover, if there is a fixed point $i_r\in[i,j]$ 
then $L\succ L_0(-i_r, j)(i_r,-j)\succ L'_2\ \Rightarrow\ L'_2\in D(L)$ iff $r_i=r_j.$\\
One has $X_{L'_1}=\lim\limits_{m\rightarrow\infty}T_j(\frac{1}{m}). U_{e_{r_j}-e_j}(1).X_L$ and\\
$X_{L'_2}=\lim\limits_{m\rightarrow\infty}T_j(\sqrt{-1}).T_i(\frac{1}{m}) . U_{e_j-e_i}(-1).U_{e_{r_j}-e_j}(1).U_{e_j-e_i}(\frac{1}{2}).X_L$.
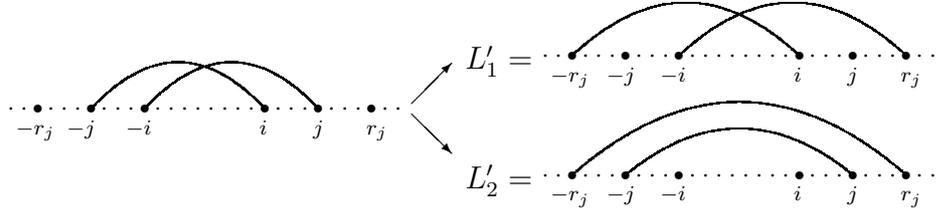
\begin{center}
\begin{picture}(100,100)(0,-40)
\multiput(-130,10)(5,0){30}%
{\circle*{1}}
\put(-120,10){\circle*{3}}
 \put(-128,2){${}_{-r_j}$}
 \put(-100,10){\circle*{3}}
 \put(-109,2){${}_{-j}$}
 \put(-80,10){\circle*{3}}
 \put(-87,2){${}_{-i}$}
 \put(-35,10){\circle*{3}}
 \put(-37,2){${}_{i}$}
 \put(-15,10){\circle*{3}}
 \put(-17,2){${}_{j}$}
 \put(5,10){\circle*{3}}
 \put(3,2){${}_{r_j}$}

 \qbezier(-100,10)(-67.5,45)(-35,10)
 \qbezier(-80,10)(-47.5,45)(-15,10)

 \put(20,12){\vector(1,1){15}}
 \put(20,8){\vector(1,-1){15}}
 \put(40,25){$L'_1=$}
 \multiput(70,30)(5,0){30}
 {\circle*{1}}
\put(80,30){\circle*{3}}
\put(72,22){${}_{-r_j}$}
\put(100,30){\circle*{3}}
\put(93,22){${}_{-j}$}
 \put(120,30){\circle*{3}}
 \put(113,22){${}_{-i}$}
 \put(165,30){\circle*{3}}
 \put(163,22){${}_{i}$}
 \put(185,30){\circle*{3}}
 \put(183,22){${}_{j}$}
 \put(205,30){\circle*{3}}
  \put(203,22){${}_{r_j}$}

 \qbezier(80,30)(122.5,70)(165,30)
  \qbezier(120,30)(162.5,70)(205,30)
\put(40,-20){$L'_2=$}
  \multiput(70,-15)(5,0){30}
 {\circle*{1}}
\put(80,-15){\circle*{3}}
\put(72,-23){${}_{-r_j}$}
\put(100,-15){\circle*{3}}
\put(93,-23){${}_{-j}$}
 \put(120,-15){\circle*{3}}
 \put(113,-23){${}_{-i}$}
 \put(165,-15){\circle*{3}}
 \put(163,-23){${}_{i}$}
 \put(185,-15){\circle*{3}}
 \put(183,-23){${}_{j}$}
 \put(205,-15){\circle*{3}}
  \put(203,-23){${}_{r_j}$}
\qbezier(80,-15)(142.5,40)(205,-15)
  \qbezier(100,-15)(142.5,20)(185,-15)

\end{picture}
\end{center}

\item{(b)} $(i,-j)$ and  $r_i$. One has $L^{\prime\prime}=L^-_{(i,-j)}(r_i,-j)$.
If $r_i=r_j$ then $SR_{L^{\prime\prime}}=R_{L_2'}$ where $L_2'$ is from (FM(a)) and  $L'_2\in D(L)$.
If $r_i\ne r_j$ then $r_i<j.$ Then $SR_{L^{\prime\prime}}=R_{L'}$  where
$L'=L_0(r_i,-j)(-r_i,j)$ and  $\dim \mathcal B_{L'}=\dim \mathcal B_L-1\ \Rightarrow\ L'\in D(L)$.
One has $X_{L'}=\lim\limits_{m\rightarrow\infty} T_i(\frac{1}{m}).U_{e_{r_i}-e_i}(1).X_{L}$.
We draw only a new case:
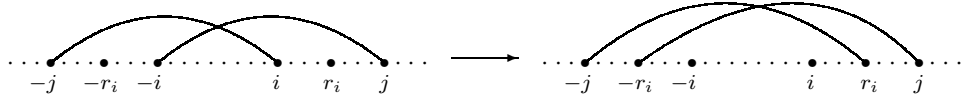
\begin{center}
\begin{picture}(100,50)(0,-10)
\multiput(-135,10)(5,0){32}%
{\circle*{1}}
\put(-120,10){\circle*{3}}
 \put(-128,2){${}_{-j}$}
 \put(-100,10){\circle*{3}}
 \put(-108,2){${}_{-r_i}$}
 \put(-80,10){\circle*{3}}
 \put(-88,2){${}_{-i}$}
 \put(-35,10){\circle*{3}}
 \put(-37,2){${}_{i}$}
 \put(-15,10){\circle*{3}}
 \put(-18,2){${}_{r_i}$}
 \put(5,10){\circle*{3}}
 \put(3,2){${}_{j}$}

 \qbezier(-120,10)(-77.5,45)(-35,10)
 \qbezier(-80,10)(-37.5,45)(5,10)

 \put(30,12){\vector(2,0){25}}

  \multiput(65,10)(5,0){32}
 {\circle*{1}}
\put(80,10){\circle*{3}}
\put(72,2){${}_{-j}$}
\put(100,10){\circle*{3}}
\put(92,2){${}_{-r_i}$}
 \put(120,10){\circle*{3}}
 \put(112,2){${}_{-i}$}
 \put(165,10){\circle*{3}}
 \put(163,2){${}_{i}$}
 \put(185,10){\circle*{3}}
 \put(183,2){${}_{r_i}$}
 \put(205,10){\circle*{3}}
  \put(203,2){${}_{j}$}
\qbezier(80,10)(132.5,55)(185,10)
\qbezier(100,10)(162.5,55)(205,10)

\end{picture}
\end{center}
\end{itemize}

\item[(FN)] {\it The arc is in the negative part:}  The arc is $(-i,-j)$  and a fixed point is $r_{-i}>0$.
 Since there is no other fixed points on $[-i,r_{-i}]$ and by the symmetry we get $r_{-i}=r_j>j$. In this case $RS_{L^-_{(-i,-j)}(-j,r_j)}=R_{L_0(-j,j)(-r_j,r_j)}$ and $L\succ L_0(i,-j),(-i,j)\succ L_0(-j,j)(-r_j,r_j)$ so that $L_0(-j,j)(-r_j,r_j)\not\in D(L)$. 
This is the case showing that for $L^{\prime\prime}\in \widehat D(L)$ even if $SR_{L^{\prime\prime}}=R_{L'}$ this does not imply $L'\in D(L)$.
\end{itemize}
\medskip

{\it $\rightleftharpoons$ Intersection of consecutive arcs} (L). We can choose the new cross of $L^{\prime\prime}$ to be in a non-negative part, that is $L^{\prime\prime}\in \widehat D(L)$ is obtained by crossing $(\pm i,\pm j)$ and a positive arc $(k,l)$ such that or at least one of $(\pm i, \pm j)$ is with positive sign and it is smaller than $k$ or arc $(-i,-j)$ satisfies   $i<k$. Obviously, we always can make this choice by the symmetry. In this part we always denote by $(k,l)$ the right arc  of $L$ changed in order to get in a new cross of $L^{\prime\prime}$.
\begin{itemize}
\item[(LS)] {\it symmetric intersection:} The left arc is $(-k,-l)$ so that $L^{\prime\prime}=L_0(k,-l)(-k,l)\in {\bf SLP}_{2n}\ \Rightarrow\ L^{\prime\prime}\in D(L).$ One has $ X_{L^{\prime\prime}}=\lim\limits_{m\rightarrow\infty} T_l(\frac{1}{m}).U_{2e_k}(-m).X_{L}$.
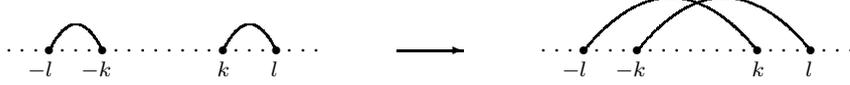
 \begin{center}
\begin{picture}(100,45)
\multiput(-115,10)(5,0){24}%
{\circle*{1}}
 \put(-100,10){\circle*{3}}
 \put(-108,2){${}_{-l}$}
 \put(-80,10){\circle*{3}}
 \put(-88,2){${}_{-k}$}
 \put(-35,10){\circle*{3}}
 \put(-37,2){${}_{k}$}
 \put(-15,10){\circle*{3}}
 \put(-17,2){${}_{l}$}

 \qbezier(-100,10)(-90,30)(-80,10)
 \qbezier(-35,10)(-25,30)(-15,10)

 \put(30,10){\vector(1,0){25}}
 \multiput(85,10)(5,0){24}
 {\circle*{1}}
\put(100,10){\circle*{3}}
\put(92,2){${}_{-l}$}
 \put(120,10){\circle*{3}}
 \put(112,2){${}_{-k}$}
 \put(165,10){\circle*{3}}
 \put(163,2){${}_{k}$}
 \put(185,10){\circle*{3}}
 \put(183,2){${}_{l}$}

 \qbezier(100,10)(132.5,50)(165,10)
  \qbezier(120,10)(152.5,50)(185,10)

\end{picture}
\end{center}

\item[(LC)] {\it Intersection with a central arc}: The left arc is $(-i,i)\in L$  so that $L^{\prime\prime}=L^-_{(-i,i)(k,l)}(-i,k)(i,l)$ and $SR_{L^{\prime\prime}}=R_{L'}$ where $L'=L_0(-i,-l)(-k,k)(i,l)$ and $\dim \mathcal B_{L'}=\dim \mathcal B_L -1\ \Rightarrow\ L'\in D(L)$.\\
One has $X_{L'}=\lim\limits_{m\rightarrow\infty} T_k(m).T_i(\frac{1}{m}) . U_{e_k-e_i}(-\frac{1}{m}).X_L$.
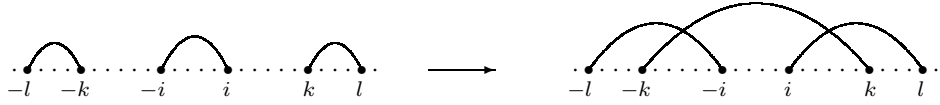
     \begin{center}
\begin{picture}(100,50)(0,-10)
%\put(-20,35){$L_\sigma=$}
\multiput(-125,10)(5,0){28}%
{\circle*{1}}
 \put(-120,10){\circle*{3}}
 \put(-128,2){${}_{-l}$}
 \put(-100,10){\circle*{3}}
  \put(-108,2){${}_{-k}$}
  \put(-70,10){\circle*{3}}
  \put(-78,2){${}_{-i}$}
  \put(-45,10){\circle*{3}}
  \put(-47,2){${}_{i}$}
 \put(-15,10){\circle*{3}}
 \put(-17,2){${}_{k}$}
 \put(5,10){\circle*{3}}
 \put(3,2){${}_{l}$}

 \qbezier(-120,10)(-110,30)(-100,10)
 \qbezier(-15,10)(-5,30)(5,10)
 \qbezier(-70,10)(-57.5,35)(-45,10)

 \put(30,10){\vector(1,0){25}}
 \multiput(85,10)(5,0){28}
 {\circle*{1}}
\put(90,10){\circle*{3}}
\put(82,2){${}_{-l}$}
 \put(110,10){\circle*{3}}
 \put(102,2){${}_{-k}$}
 \put(140,10){\circle*{3}}
 \put(132,2){${}_{-i}$}
 \put(165,10){\circle*{3}}
 \put(163,2){${}_{i}$}
 \put(195,10){\circle*{3}}
 \put(193,2){${}_{k}$}
 \put(215,10){\circle*{3}}
 \put(213,2){${}_{l}$}

 \qbezier(90,10)(115,45)(140,10)
 \qbezier(110,10)(152.5,60)(195,10)
 \qbezier(165,10)(190,45)(215,10)
 \end{picture}
\end{center}

\item [(LP)] {\it Intersection with a positive arc}: The left arc is $(i,j)\in L$ where $j<k$ so that $L^{\prime\prime}=L_{(i,j)(k,l)}^-(i,k)(j,l).$ There is no such $L'$ that $SR_{L^{\prime\prime}}$ is $R_{L'}$. Two maximal SLPs $\{L'_1, L'_2\}$ such that $R_{L'_s}\prec SR_{L^{\prime\prime}}$
    are $L'_1=L_0(i,k)(j,l)(-i,-k)(-j,-l)$ and $L_2'=L_0(j,-j)(k,-k)(i,l)(-i,-l)$.
    Note that\\ $L\succ L^-_{(i,j)(-i,-j)}(i,-j)(-i,j)\succ L_2'$ $\Rightarrow$ $L'_2\not\in D(L)$.\\ 
		As for $L'_1$:
         $\dim \mathcal B_{L_1'}=\dim \mathcal B_L-1\ \Rightarrow\ L_1'\in D(L)$.\\
				One has $ X_{L'_1}=\lim\limits_{m\rightarrow\infty}T_l(-\frac{1}{m}).T_i(m).U_{e_k-e_j}(m).X_L.$

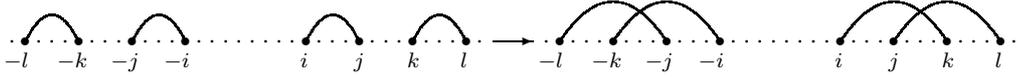
 \begin{center}
\begin{picture}(100,45)(0,-10)
\multiput(-145,10)(5,0){36}%
{\circle*{1}}
\put(-140,10){\circle*{3}}
 \put(-148,2){${}_{-l}$}
\put(-120,10){\circle*{3}}
 \put(-128,2){${}_{-k}$}
 \put(-100,10){\circle*{3}}
 \put(-108,2){${}_{-j}$}
 \put(-80,10){\circle*{3}}
 \put(-88,2){${}_{-i}$}
 \put(-35,10){\circle*{3}}
 \put(-37,2){${}_{i}$}
 \put(-15,10){\circle*{3}}
 \put(-17,2){${}_{j}$}
 \put(5,10){\circle*{3}}
 \put(3,2){${}_{k}$}
 \put(25,10){\circle*{3}}
 \put(23,2){${}_{l}$}

 \qbezier(-100,10)(-90,30)(-80,10)
 \qbezier(-35,10)(-25,30)(-15,10)
\qbezier(-140,10)(-130,30)(-120,10)
\qbezier(5,10)(15,30)(25,10)
 \put(35,10){\vector(1,0){15}}

 \multiput(55,10)(5,0){36}
 {\circle*{1}}
 \put(60,10){\circle*{3}}
\put(52,2){${}_{-l}$}
\put(80,10){\circle*{3}}
\put(72,2){${}_{-k}$}
\put(100,10){\circle*{3}}
\put(92,2){${}_{-j}$}
 \put(120,10){\circle*{3}}
 \put(112,2){${}_{-i}$}
 \put(165,10){\circle*{3}}
 \put(163,2){${}_{i}$}
 \put(185,10){\circle*{3}}
 \put(183,2){${}_{j}$}
 \put(205,10){\circle*{3}}
  \put(203,2){${}_{k}$}
  \put(225,10){\circle*{3}}
  \put(223,2){${}_{l}$}

 \qbezier(80,10)(100,40)(120,10)
  \qbezier(165,10)(185,40)(205,10)
  \qbezier(60,10)(80,40)(100,10)
  \qbezier(185,10)(205,40)(225,10)

\end{picture}
\end{center}

\item[(LM)] {\it Intersection with a mixed arc}: The left arc has one positive and one negative end. Here we have to distinguish case $(-i,j)$ and $(i,-j)$ :
\begin{itemize}
\item{(a)}  $(-i,j)$:   $L^{\prime\prime}=L_{(-i,j)(k,l)}^-(-i,k)(j,l)$. There is no  $L'$ such that $SR_{L^{\prime\prime}}$ is $R_{L'}$. 
Two maximal SLPs $\{L'_1, L'_2\}$ such that $R_{L'_s}\prec SR_{L^{\prime\prime}}$  are $L'_1=L_0(-i,k)(j,l)(i,-k)(-j,-l)$\\ and $L_2'=L_0(j,-j)(k,-k)(i,l)(-i,-l)$;
One has $\dim \mathcal B_{L'_1}=\dim\mathcal B_L-1\ \Rightarrow\ L'_1\in D(L)$. 
As for $L'_2$ one has $\dim \mathcal B_{L'_2}=\dim\mathcal B_L-1\ \Rightarrow\ L'_2\in D(L)$ 
iff there is no fixed point $r_i\in[i,j]$. Otherwise $L\succ L_{(i,-j)(-i,j)}(r_i,-j)(-r_i,j)\succ L_2\ \Rightarrow\ L'_2\not\in D(L)$.\\
One has $X_{L'_1}=\lim\limits_{m\rightarrow\infty}T_k(m).T_j(-\frac{1}{m}).U_{e_k-e_j}(\frac{1}{m}).X_L$ and \\
$X_{L'_2}=\lim\limits_{m\rightarrow\infty}T_l(-\frac{1}{m}).T_j(\sqrt{-1}).T_i(\frac{1}{m}).U_{e_k-e_i}(1).U_{e_k-e_j}(1).U_{e_j-e_i}(-\frac{1}{2}).X_{L}$
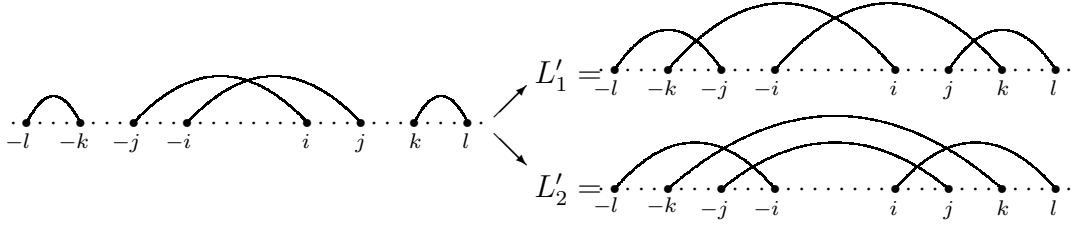
\begin{center}
\begin{picture}(100,100)(0,-30)
\multiput(-165,10)(5,0){36}%
{\circle*{1}}
\put(-160,10){\circle*{3}}
 \put(-168,2){${}_{-l}$}
\put(-140,10){\circle*{3}}
 \put(-148,2){${}_{-k}$}
 \put(-120,10){\circle*{3}}
 \put(-128,2){${}_{-j}$}
 \put(-100,10){\circle*{3}}
 \put(-108,2){${}_{-i}$}
 \put(-55,10){\circle*{3}}
 \put(-57,2){${}_{i}$}
 \put(-35,10){\circle*{3}}
 \put(-37,2){${}_{j}$}
 \put(-15,10){\circle*{3}}
 \put(-17,2){${}_{k}$}
 \put(5,10){\circle*{3}}
 \put(3,2){${}_{l}$}

 \qbezier(-120,10)(-87.5,45)(-55,10)
 \qbezier(-100,10)(-67.5,45)(-35,10)
 \qbezier(-160,10)(-150,30)(-140,10)
\qbezier(-15,10)(-5,30)(5,10)

 \put(15,12){\vector(1,1){12}}
 \put(15,6){\vector(1,-1){12}}
\put(30,25){$L'_1=$}
 \multiput(55,30)(5,0){36}
 {\circle*{1}}
 \put(60,30){\circle*{3}}
\put(52,22){${}_{-l}$}
\put(80,30){\circle*{3}}
\put(72,22){${}_{-k}$}
\put(100,30){\circle*{3}}
\put(92,22){${}_{-j}$}
 \put(120,30){\circle*{3}}
 \put(112,22){${}_{-i}$}
 \put(165,30){\circle*{3}}
 \put(163,22){${}_{i}$}
 \put(185,30){\circle*{3}}
 \put(183,22){${}_{j}$}
 \put(205,30){\circle*{3}}
    \put(203,22){${}_{k}$}
  \put(225,30){\circle*{3}}
  \put(223,22){${}_{l}$}

 \qbezier(80,30)(122.5,80)(165,30)
  \qbezier(120,30)(162.5,80)(205,30)
  \qbezier(60,30)(80,60)(100,30)
  \qbezier(185,30)(205,60)(225,30)
\put(30,-18){$L_2'=$}
\multiput(55,-15)(5,0){36}
 {\circle*{1}}
 \put(60,-15){\circle*{3}}
\put(52,-22){${}_{-l}$}
\put(80,-15){\circle*{3}}
\put(72,-22){${}_{-k}$}
\put(100,-15){\circle*{3}}
\put(92,-23){${}_{-j}$}
 \put(120,-15){\circle*{3}}
 \put(112,-23){${}_{-i}$}
 \put(165,-15){\circle*{3}}
 \put(163,-23){${}_{i}$}
 \put(185,-15){\circle*{3}}
 \put(183,-23){${}_{j}$}
 \put(205,-15){\circle*{3}}
  \put(203,-23){${}_{k}$}
  \put(225,-15){\circle*{3}}
  \put(223,-23){${}_{l}$}
\qbezier(80,-15)(142.5,40)(205,-15)
  \qbezier(100,-15)(142.5,20)(185,-15)
\qbezier(60,-15)(90,20)(120,-15)
\qbezier(165,-15)(195,20)(225,-15)
\end{picture}
\end{center}

\item{(b)} $(i,-j)$ where $i<k.$ Here $j$ can be any point of $[i,n]$. 
$L^{\prime\prime}=L_{(i,-j)(k,l)}^-(k,-j)(i,l)$ and $SR_{L^{\prime\prime}}=R_{L'}$ where
$$L'=\left\{\begin{array}{ll}L_0(j,-j)(k,-k)(i,l)(-i,-l)&{\rm if}\ j<k;\\
                             L_0(k,-j)(-k,j)(i,l)(-i,-l)&{\rm otherwise};\\
                             \end{array}\right.$$
In both cases one has $\dim\mathcal B_{L'}=\dim\mathcal B_L-1\ \Rightarrow\ L'\in D(L)$.  If $j<k$ then $L'=L'_2$ from (LM(a)) so we omit it.\\
One has $X_{L'}=\lim\limits_{m\rightarrow\infty}T_l(\frac{1}{m}).T_j(-\frac{1}{m}).U_{e_k-e_i}(-m).X_L$.

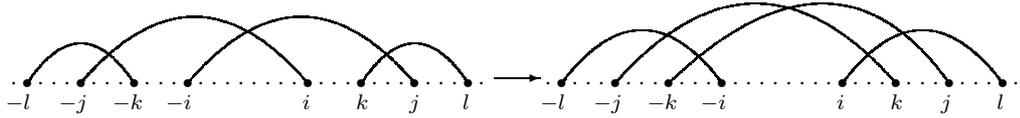
\begin{center}
\begin{picture}(100,60)(0,-10)
\multiput(-145,10)(5,0){36}%
{\circle*{1}}
\put(-140,10){\circle*{3}}
 \put(-148,2){${}_{-l}$}
\put(-120,10){\circle*{3}}
 \put(-128,2){${}_{-j}$}
 \put(-100,10){\circle*{3}}
 \put(-108,2){${}_{-k}$}
 \put(-80,10){\circle*{3}}
 \put(-88,2){${}_{-i}$}
 \put(-35,10){\circle*{3}}
 \put(-37,2){${}_{i}$}
 \put(-15,10){\circle*{3}}
 \put(-17,2){${}_{k}$}
 \put(5,10){\circle*{3}}
 \put(3,2){${}_{j}$}
\put(25,10){\circle*{3}}
 \put(23,2){${}_{l}$}

 \qbezier(-120,10)(-77.5,60)(-35,10)
 \qbezier(-80,10)(-37.5,60)(5,10)
 \qbezier(-140,10)(-120,40)(-100,10)
 \qbezier(-15,10)(5,40)(25,10)

 \put(35,12){\vector(2,0){17}}

\multiput(55,10)(5,0){36}
 {\circle*{1}}
 \put(60,10){\circle*{3}}
\put(52,2){${}_{-l}$}

\put(80,10){\circle*{3}}
\put(72,2){${}_{-j}$}
\put(100,10){\circle*{3}}
\put(92,2){${}_{-k}$}
 \put(120,10){\circle*{3}}
 \put(112,2){${}_{-i}$}
 \put(165,10){\circle*{3}}
 \put(163,2){${}_{i}$}
 \put(185,10){\circle*{3}}
 \put(183,2){${}_{k}$}
 \put(205,10){\circle*{3}}
  \put(202,2){${}_{j}$}
 \put(225,10){\circle*{3}}
  \put(223,2){${}_{l}$}

\qbezier(80,10)(132.5,70)(185,10)
\qbezier(100,10)(162.5,70)(205,10)
\qbezier(60,10)(90,50)(120,10)
\qbezier(165,10)(195,50)(225,10)

\end{picture}
\end{center}
If $j>l$ then the picture is:
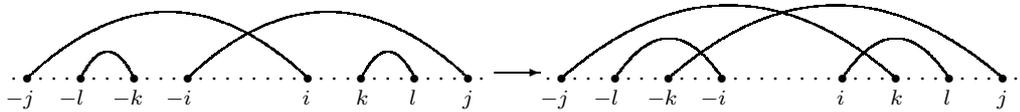
\begin{center}
\begin{picture}(100,60)(0,-10)
\multiput(-145,10)(5,0){36}%
{\circle*{1}}
\put(-140,10){\circle*{3}}
 \put(-148,2){${}_{-j}$}
\put(-120,10){\circle*{3}}
 \put(-128,2){${}_{-l}$}
 \put(-100,10){\circle*{3}}
 \put(-108,2){${}_{-k}$}
 \put(-80,10){\circle*{3}}
 \put(-88,2){${}_{-i}$}
 \put(-35,10){\circle*{3}}
 \put(-37,2){${}_{i}$}
 \put(-15,10){\circle*{3}}
 \put(-17,2){${}_{k}$}
 \put(5,10){\circle*{3}}
 \put(3,2){${}_{l}$}
\put(25,10){\circle*{3}}
 \put(23,2){${}_{j}$}

 \qbezier(-120,10)(-110,30)(-100,10)
\qbezier(-15,10)(-5,30)(5,10)
\qbezier(-140,10)(-87.5,60)(-35,10)
\qbezier(-80,10)(-27.5,60)(25,10)

 \put(35,12){\vector(2,0){17}}

\multiput(55,10)(5,0){36}
 {\circle*{1}}
 \put(60,10){\circle*{3}}
\put(52,2){${}_{-j}$}

\put(80,10){\circle*{3}}
\put(72,2){${}_{-l}$}
\put(100,10){\circle*{3}}
\put(92,2){${}_{-k}$}
 \put(120,10){\circle*{3}}
 \put(112,2){${}_{-i}$}
 \put(165,10){\circle*{3}}
 \put(163,2){${}_{i}$}
 \put(185,10){\circle*{3}}
 \put(183,2){${}_{k}$}
 \put(205,10){\circle*{3}}
  \put(203,2){${}_{l}$}
 \put(225,10){\circle*{3}}
  \put(223,2){${}_{j}$}

\qbezier(80,10)(100,40)(120,10)
\qbezier(100,10)(162.5,65)(225,10)
\qbezier(60,10)(122.5,65)(185,10)
\qbezier(165,10)(185,40)(205,10)
\end{picture}
\end{center}
\end{itemize}

\item[(LN)] {\it Intersection with a negative arc}: The left arc is $(-i,-j)$ where $i<k$ (since the cross is in a nonnegative part). Note that
$j<l$, otherwise $(k,l)$ is under $(i,j)$ so that $(-i,-j)$ and $(k,l)$ do not satisfy the condition of minimality.
In this case $L^{\prime\prime}=L_{(-i,-j)(k,l)}^-(-j,k)(-i,l)$ so that $SR_{L^{\prime\prime}}=R_{L'}$ where
$$L'=\left\{\begin{array}{ll}L_0(j,-j)(k,-k)(i,l)(-i,-l)&{\rm if}\ j<k;\\
                             L_0(j,-k)(k,-j)(i,l)(-i,-j)&{\rm otherwise};\\
                             \end{array}\right.$$
In both cases $L\succ L^-_{(i,j)(-i,-j)}(i,-j)(-i,j)\succ L'$ so that $L'\not\in D(L).$

\end{itemize}
\medskip
{\it $\rightleftharpoons$ Intersection of concentric arcs} (C). Exactly as in the previous case we can consider only $L^{\prime\prime}$ obtained by a  crossing of concentric arcs in the non negative part. Again let us fix an external arc $(i,\pm l)$ or $(-i,l)$ where $0<i<l$ then the conditions for a new cross to be non-negative are: internal arc $(\pm j,k)$ (that is such that $\pm j,k\in [i,\pm l]$ or resp. $\pm j,k\in [-i, l]$) satisfies $k>0$ and $j\leq k$.
Note that in the case when the external arc is $(i,-l)$ the new crossing will move the cross of $(-i,l)$ and $(i,-l)$ to the negative part.

\begin{itemize}
\item[(CS)] {\it Intersection of central arcs}: Both  external and internal arcs of $L$ are  central: $(-l,l)$ and $(-j,j)$  where $0<j<l$ so that $L^{\prime\prime}=L^-_{(-l,l)(-j,j)}(j,-l)(-j,l)\in {\bf SLP}_{2n}\ \Rightarrow\ L^{\prime\prime}\in D(L).$
    One has $X_{L^{\prime\prime}}=\lim\limits_{m\rightarrow\infty} T_l(\frac{m}{\sqrt{-1}}).T_j(\frac{1}{m}). U_{e_l-e_j}(\sqrt{-1}).X_L$.
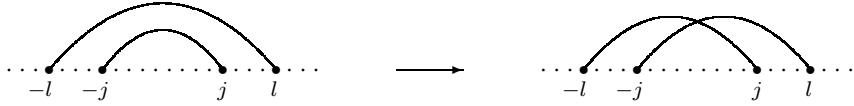
\begin{center}
\begin{picture}(100,45)
\multiput(-115,10)(5,0){24}%
{\circle*{1}}
 \put(-100,10){\circle*{3}}
 \put(-108,2){${}_{-l}$}
 \put(-80,10){\circle*{3}}
 \put(-88,2){${}_{-j}$}
 \put(-35,10){\circle*{3}}
 \put(-37,2){${}_{j}$}
 \put(-15,10){\circle*{3}}
 \put(-17,2){${}_{l}$}

 \qbezier(-100,10)(-57.5,60)(-15,10)
 \qbezier(-80,10)(-57.5,40)(-35,10)
 \put(30,10){\vector(1,0){25}}
 \multiput(85,10)(5,0){24}
 {\circle*{1}}
\put(100,10){\circle*{3}}
\put(92,2){${}_{-l}$}
 \put(120,10){\circle*{3}}
 \put(112,2){${}_{-j}$}
 \put(165,10){\circle*{3}}
 \put(163,2){${}_{j}$}
 \put(185,10){\circle*{3}}
 \put(183,2){${}_{l}$}

 \qbezier(100,10)(132.5,50)(165,10)
  \qbezier(120,10)(152.5,50)(185,10)

\end{picture}
\end{center}

\item[(CCP)] {\it Intersection of external central arc and positive internal}: The external arc is  $(-l,l)$ and the internal arc is  $(j,k)$.
 Then $L^{\prime\prime}=L_{(-l,l)(j,k)}^-(k,-l)(j,l)$ and $SR_{L^{\prime\prime}}=R_{L'}$ where $L'=L_0(k,-k)(j,l)(-j,-l)$, and
 $\dim \mathcal B_{L'}=\dim\mathcal B_L-1\ \Rightarrow\ L'\in D(L).$\\
One has $X_{L'}=\lim\limits_{m\rightarrow\infty}T_l(m).T_j(m).U_{e_l+e_j}(\frac{1}{2}).U_{e_k+e_j}(-\frac{1}{2}).U_{e_l-e_k}(1).X_L$.
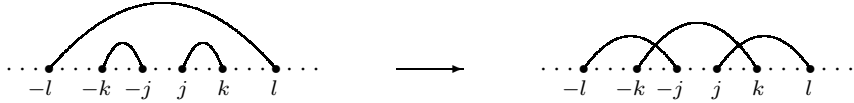
\begin{center}
\begin{picture}(100,45)
\multiput(-115,10)(5,0){24}%
{\circle*{1}}
 \put(-100,10){\circle*{3}}
 \put(-108,2){${}_{-l}$}
 \put(-80,10){\circle*{3}}
 \put(-88,2){${}_{-k}$}
 \put(-65,10){\circle*{3}}
 \put(-72,2){${}_{-j}$}
 \put(-50,10){\circle*{3}}
 \put(-52,2){${}_{j}$}
 \put(-35,10){\circle*{3}}
 \put(-37,2){${}_{k}$}
 \put(-15,10){\circle*{3}}
 \put(-17,2){${}_{l}$}

 \qbezier(-100,10)(-57.5,60)(-15,10)
 \qbezier(-80,10)(-72.5,30)(-65,10)
\qbezier(-50,10)(-42.5,30)(-35,10)

 \put(30,10){\vector(1,0){25}}
 \multiput(85,10)(5,0){24}
 {\circle*{1}}
\put(100,10){\circle*{3}}
\put(92,2){${}_{-l}$}
 \put(120,10){\circle*{3}}
 \put(112,2){${}_{-k}$}
 \put(135,10){\circle*{3}}
 \put(127,2){${}_{-j}$}
 \put(150,10){\circle*{3}}
 \put(148,2){${}_{j}$}
 \put(165,10){\circle*{3}}
 \put(163,2){${}_{k}$}
 \put(185,10){\circle*{3}}
 \put(183,2){${}_{l}$}

 \qbezier(100,10)(117.5,35)(135,10)
 \qbezier(150,10)(167.5,35)(185,10)
  \qbezier(120,10)(142.5,45)(165,10)

\end{picture}
\end{center}

\item[(CCM)] {\it Intersection of external central arc and mixed internal arc}: The external arc is  $(-l,l)$ and the internal arc is $(-j,k)$. Then $L^{\prime\prime}=L_{(-l,l)(-j,k)}^-(k,-l)(-j,l)$ and $SR_{L^{\prime\prime}}=R_{L'}$ where $L'=L_0(j,-j)(-i,k)(i,-k)$ and
    $\dim \mathcal B_{L'}=\dim\mathcal B_L-1\ \Rightarrow\ L'\in D(L).$ \\
		One has $X_{L'}=\lim\limits_{n\rightarrow\infty}T_l(m).T_j(\frac{1}{m}).U_{e_l-e_j}(-\frac{1}{2}).U_{e_k-e_j}(\frac{1}{2}).U_{e_l-e_k}(1). X_L$.
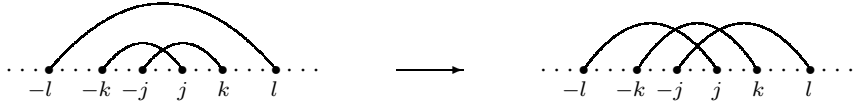
\begin{center}
\begin{picture}(100,45)
\multiput(-115,10)(5,0){24}%
{\circle*{1}}
 \put(-100,10){\circle*{3}}
 \put(-108,2){${}_{-l}$}
 \put(-80,10){\circle*{3}}
 \put(-88,2){${}_{-k}$}
 \put(-65,10){\circle*{3}}
 \put(-73,2){${}_{-j}$}
 \put(-50,10){\circle*{3}}
 \put(-52,2){${}_{j}$}
 \put(-35,10){\circle*{3}}
 \put(-37,2){${}_{k}$}
 \put(-15,10){\circle*{3}}
 \put(-17,2){${}_{l}$}

 \qbezier(-100,10)(-57.5,60)(-15,10)
 \qbezier(-80,10)(-65,30)(-50,10)
\qbezier(-65,10)(-50,30)(-35,10)

 \put(30,10){\vector(1,0){25}}
 \multiput(85,10)(5,0){24}
 {\circle*{1}}
\put(100,10){\circle*{3}}
\put(92,2){${}_{-l}$}
 \put(120,10){\circle*{3}}
 \put(112,2){${}_{-k}$}
 \put(135,10){\circle*{3}}
 \put(127,2){${}_{-j}$}
 \put(150,10){\circle*{3}}
 \put(148,2){${}_{j}$}
 \put(165,10){\circle*{3}}
 \put(163,2){${}_{k}$}
 \put(185,10){\circle*{3}}
 \put(183,2){${}_{l}$}

 \qbezier(100,10)(125,45)(150,10)
 \qbezier(135,10)(160,45)(185,10)
  \qbezier(120,10)(142.5,45)(165,10)

\end{picture}
\end{center}

\item[(CPP)] {\it Both external and internal arcs are positive}: If the external arc is $(i,l)$ and the internal $(j,k)$ where $i<j<k<l$. 
In this case $L^{\prime\prime}=L_{(i,l)(j,k)}^-(i,k)(j,l)$ and $SR_{L^{\prime\prime}}=R_{L'}$ where $L'=L_0(i,k)(-i,-k)(j,l)(-j,-l)$ and\\
$\dim \mathcal B_{L'}=\dim\mathcal B_L-1\ \Rightarrow\ L'\in D(L).$\\
One has $X_{L'}=\lim\limits_{m\rightarrow\infty}T_l(-m).T_j(m).U_{e_j-e_i}(-1).U_{e_l-e_k}(-1).X_L$.
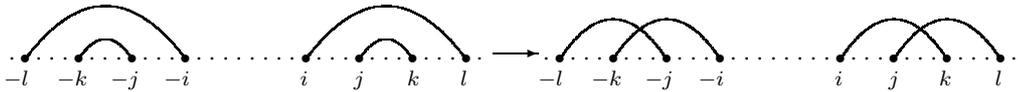
\begin{center}
\begin{picture}(100,50)(0,-10)
\multiput(-145,10)(5,0){36}%
{\circle*{1}}
\put(-140,10){\circle*{3}}
 \put(-148,2){${}_{-l}$}
\put(-120,10){\circle*{3}}
 \put(-128,2){${}_{-k}$}
 \put(-100,10){\circle*{3}}
 \put(-108,2){${}_{-j}$}
 \put(-80,10){\circle*{3}}
 \put(-88,2){${}_{-i}$}
 \put(-35,10){\circle*{3}}
 \put(-37,2){${}_{i}$}
 \put(-15,10){\circle*{3}}
 \put(-17,2){${}_{j}$}
 \put(5,10){\circle*{3}}
 \put(3,2){${}_{k}$}
\put(25,10){\circle*{3}}
 \put(23,2){${}_{l}$}

 \qbezier(-140,10)(-110,50)(-80,10)
 \qbezier(-120,10)(-110,25)(-100,10)
 \qbezier(-15,10)(-5,25)(5,10)
 \qbezier(-35,10)(-5,50)(25,10)

 \put(35,12){\vector(2,0){17}}

\multiput(55,10)(5,0){36}
 {\circle*{1}}
 \put(60,10){\circle*{3}}
\put(52,2){${}_{-l}$}
\put(80,10){\circle*{3}}
\put(72,2){${}_{-k}$}
\put(100,10){\circle*{3}}
\put(92,2){${}_{-j}$}
 \put(120,10){\circle*{3}}
 \put(112,2){${}_{-i}$}
 \put(165,10){\circle*{3}}
 \put(163,2){${}_{i}$}
 \put(185,10){\circle*{3}}
 \put(183,2){${}_{j}$}
 \put(205,10){\circle*{3}}
  \put(202,2){${}_{k}$}
 \put(225,10){\circle*{3}}
  \put(223,2){${}_{l}$}

\qbezier(60,10)(80,40)(100,10)
\qbezier(80,10)(100,40)(120,10)
\qbezier(165,10)(185,40)(205,10)
\qbezier(185,10)(205,40)(225,10)

\end{picture}
\end{center}

\item[(CMC)] {\it The external arc is mixed and the internal arc is central}: The external arc is  $(-i,l)$ and the internal arc $(-j,j)$. Note that
 $j<i$ and the crossing of $(-j,j)$ with $(-i,l)$ or with $(i,-l)$ are symmetric so we can choose $L^{\prime\prime}$ to be obtained by crossing with $(-i,l)$ so that $L^{\prime\prime}=L^-_{(-i,l)(-j,j)}(-j,l)(j,-i)$
and $SR_{L^{\prime\prime}}=R_{L'}$ where $L'=L_0(-j,l)(j,-l)(i,-i)$.
One has $\dim \mathcal B_{L'}=\dim\mathcal B_L-1\ \Rightarrow\ L'\in D(L).$ 
And $X_{L'}=\lim\limits_{m\rightarrow\infty}T_l(-m).T_j(\frac{1}{n}).U_{e_i-e_j}(\frac{1}{2}).U_{e_l-e_i}(-2).U_{e_i-e_j}(\frac{1}{2}).X_L.$
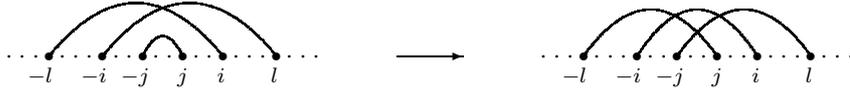
\begin{center}
\begin{picture}(100,45)
%\put(-20,35){$L_\sigma=$}
\multiput(-115,10)(5,0){24}%
{\circle*{1}}
 \put(-100,10){\circle*{3}}
 \put(-108,2){${}_{-l}$}
 \put(-80,10){\circle*{3}}
 \put(-88,2){${}_{-i}$}
 \put(-65,10){\circle*{3}}
 \put(-73,2){${}_{-j}$}
 \put(-50,10){\circle*{3}}
 \put(-52,2){${}_{j}$}
 \put(-35,10){\circle*{3}}
 \put(-37,2){${}_{i}$}
 \put(-15,10){\circle*{3}}
 \put(-17,2){${}_{l}$}

 \qbezier(-100,10)(-67.5,50)(-35,10)
 \qbezier(-80,10)(-47.5,50)(-15,10)
\qbezier(-65,10)(-57.5,25)(-50,10)

 \put(30,10){\vector(1,0){25}}
 \multiput(85,10)(5,0){24}
 {\circle*{1}}
\put(100,10){\circle*{3}}
\put(92,2){${}_{-l}$}
 \put(120,10){\circle*{3}}
 \put(112,2){${}_{-i}$}
 \put(135,10){\circle*{3}}
 \put(127,2){${}_{-j}$}
 \put(150,10){\circle*{3}}
 \put(148,2){${}_{j}$}
 \put(165,10){\circle*{3}}
 \put(163,2){${}_{i}$}
 \put(185,10){\circle*{3}}
 \put(183,2){${}_{l}$}

 \qbezier(100,10)(125,45)(150,10)
 \qbezier(135,10)(160,45)(185,10)
  \qbezier(120,10)(142.5,45)(165,10)

\end{picture}
\end{center}

\item[(CMP)] {\it The external arc is mixed and the internal arc is positive}: The external arc is either $(-i,l)$ or $(i,-l)$ 
and the internal arc is $(j,k)$ where $k<l$ or resp. $k<i$. Here we have to distinguish these two cases:\\
    (a) The external arc is $(-i,l)$. Then $L^{\prime\prime}=L_{(j,k)(-i,l)}^-(j,l)(-i,k)$ and we have to distinguish two subcases:
    \begin{itemize}
    \item {(i)} If $k>i$ then $SR_{L^{\prime\prime}}=R_{L'}$ where $L'= L_0(j,l)(-j,-l)(-i,k)(i,-k)$ and
    $\dim \mathcal B_{L'}=\dim\mathcal B_L-1\ \Rightarrow\ L'\in D(L).$\\
    One has $X_{L'}=\lim\limits_{m\rightarrow\infty}T_k(\frac{1}{m}).T_i(-m).U_{e_l-e_k}(1).U_{e_j+e_i}(1).X_L$.

If $i<j$ the picture is:
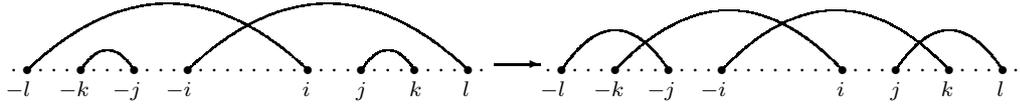
\begin{center}
\begin{picture}(100,50)(0,-10)
\multiput(-145,10)(5,0){36}%
{\circle*{1}}
\put(-140,10){\circle*{3}}
 \put(-148,2){${}_{-l}$}
\put(-120,10){\circle*{3}}
 \put(-128,2){${}_{-k}$}
 \put(-100,10){\circle*{3}}
 \put(-108,2){${}_{-j}$}
 \put(-80,10){\circle*{3}}
 \put(-88,2){${}_{-i}$}
 \put(-35,10){\circle*{3}}
 \put(-37,2){${}_{i}$}
 \put(-15,10){\circle*{3}}
 \put(-17,2){${}_{j}$}
 \put(5,10){\circle*{3}}
 \put(3,2){${}_{k}$}
\put(25,10){\circle*{3}}
 \put(23,2){${}_{l}$}

 \qbezier(-140,10)(-87.5,60)(-35,10)
 \qbezier(-120,10)(-110,25)(-100,10)
 \qbezier(-15,10)(-5,25)(5,10)
 \qbezier(-80,10)(-27.5,60)(25,10)

 \put(35,12){\vector(2,0){17}}

\multiput(55,10)(5,0){36}
 {\circle*{1}}
 \put(60,10){\circle*{3}}
\put(52,2){${}_{-l}$}
\put(80,10){\circle*{3}}
\put(72,2){${}_{-k}$}
\put(100,10){\circle*{3}}
\put(92,2){${}_{-j}$}
 \put(120,10){\circle*{3}}
 \put(112,2){${}_{-i}$}
 \put(165,10){\circle*{3}}
 \put(163,2){${}_{i}$}
 \put(185,10){\circle*{3}}
 \put(183,2){${}_{j}$}
 \put(205,10){\circle*{3}}
  \put(202,2){${}_{k}$}
 \put(225,10){\circle*{3}}
  \put(223,2){${}_{l}$}

\qbezier(60,10)(80,40)(100,10)
\qbezier(80,10)(122.5,55)(165,10)
\qbezier(120,10)(162.5,55)(205,10)
\qbezier(185,10)(205,40)(225,10)

\end{picture}
\end{center}
If $i>j$ the picture is:
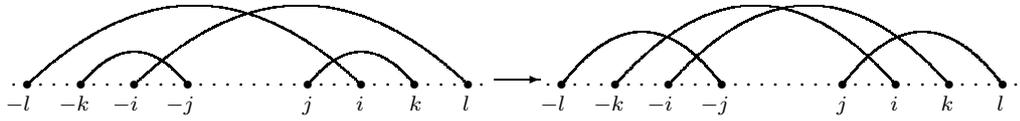
\begin{center}
\begin{picture}(100,60)(0,-10)
%\put(-20,35){$L_\sigma=$}
\multiput(-145,10)(5,0){36}%
{\circle*{1}}
\put(-140,10){\circle*{3}}
 \put(-148,2){${}_{-l}$}
\put(-120,10){\circle*{3}}
 \put(-128,2){${}_{-k}$}
 \put(-100,10){\circle*{3}}
 \put(-108,2){${}_{-i}$}
 \put(-80,10){\circle*{3}}
 \put(-88,2){${}_{-j}$}
 \put(-35,10){\circle*{3}}
 \put(-37,2){${}_{j}$}
 \put(-15,10){\circle*{3}}
 \put(-17,2){${}_{i}$}
 \put(5,10){\circle*{3}}
 \put(3,2){${}_{k}$}
\put(25,10){\circle*{3}}
 \put(23,2){${}_{l}$}

 \qbezier(-140,10)(-77.5,70)(-15,10)
 \qbezier(-120,10)(-100,35)(-80,10)
 \qbezier(-35,10)(-15,35)(5,10)
 \qbezier(-100,10)(-37.5,70)(25,10)

 \put(35,12){\vector(2,0){17}}

\multiput(55,10)(5,0){36}
 {\circle*{1}}
 \put(60,10){\circle*{3}}
\put(52,2){${}_{-l}$}
\put(80,10){\circle*{3}}
\put(72,2){${}_{-k}$}
\put(100,10){\circle*{3}}
\put(92,2){${}_{-i}$}
 \put(120,10){\circle*{3}}
 \put(112,2){${}_{-j}$}
 \put(165,10){\circle*{3}}
 \put(163,2){${}_{j}$}
 \put(185,10){\circle*{3}}
 \put(183,2){${}_{i}$}
 \put(205,10){\circle*{3}}
  \put(202,2){${}_{k}$}
 \put(225,10){\circle*{3}}
  \put(223,2){${}_{l}$}

\qbezier(60,10)(90,50)(120,10)
\qbezier(80,10)(132.5,70)(185,10)
\qbezier(100,10)(152.5,70)(205,10)
\qbezier(165,10)(195,50)(225,10)

\end{picture}
\end{center}

\item{(ii)} The case  $k<i$ was considered in detail in \ref{5.3}. We provide the picture for the completeness
  
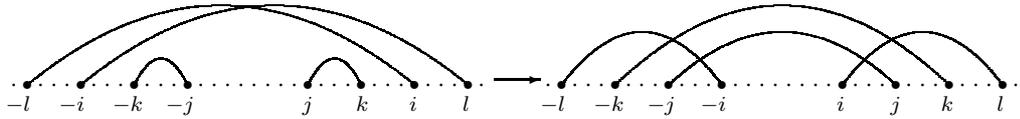
\begin{center}
\begin{picture}(100,60)(0,-10)
\multiput(-145,10)(5,0){36}%
{\circle*{1}}
\put(-140,10){\circle*{3}}
 \put(-148,2){${}_{-l}$}
\put(-120,10){\circle*{3}}
 \put(-128,2){${}_{-i}$}
 \put(-100,10){\circle*{3}}
 \put(-108,2){${}_{-k}$}
 \put(-80,10){\circle*{3}}
 \put(-88,2){${}_{-j}$}
 \put(-35,10){\circle*{3}}
 \put(-37,2){${}_{j}$}
 \put(-15,10){\circle*{3}}
 \put(-17,2){${}_{k}$}
 \put(5,10){\circle*{3}}
 \put(3,2){${}_{i}$}
\put(25,10){\circle*{3}}
 \put(23,2){${}_{l}$}

 \qbezier(-140,10)(-67.5,70)(5,10)
 \qbezier(-120,10)(-47.5,70)(25,10)
 \qbezier(-100,10)(-90,30)(-80,10)
 \qbezier(-35,10)(-25,30)(-15,10)

 \put(35,12){\vector(2,0){17}}

\multiput(55,10)(5,0){36}
 {\circle*{1}}
 \put(60,10){\circle*{3}}
\put(52,2){${}_{-l}$}
\put(80,10){\circle*{3}}
\put(72,2){${}_{-k}$}
\put(100,10){\circle*{3}}
\put(92,2){${}_{-j}$}
 \put(120,10){\circle*{3}}
 \put(112,2){${}_{-i}$}
 \put(165,10){\circle*{3}}
 \put(163,2){${}_{i}$}
 \put(185,10){\circle*{3}}
 \put(183,2){${}_{j}$}
 \put(205,10){\circle*{3}}
  \put(202,2){${}_{k}$}
 \put(225,10){\circle*{3}}
  \put(223,2){${}_{l}$}

\qbezier(60,10)(90,50)(120,10)
\qbezier(165,10)(195,50)(225,10)
\qbezier(80,10)(142.5,70)(205,10)
\qbezier(100,10)(142.5,50)(185,10)

\end{picture}
\end{center}
\end{itemize}

\noindent
 (b) The external arc is $(i,-l)$. This case was also considered in detail in \ref{5.3}. For completeness we show  the new case.
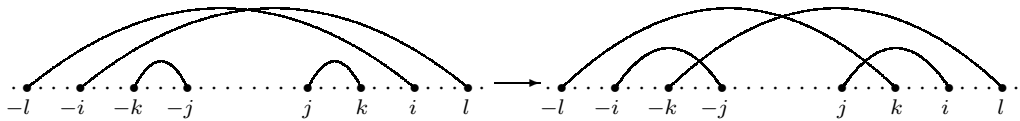
\begin{center}
\begin{picture}(100,70)(0,-10)
\multiput(-145,10)(5,0){36}%
{\circle*{1}}
\put(-140,10){\circle*{3}}
 \put(-148,2){${}_{-l}$}
\put(-120,10){\circle*{3}}
 \put(-128,2){${}_{-i}$}
 \put(-100,10){\circle*{3}}
 \put(-108,2){${}_{-k}$}
 \put(-80,10){\circle*{3}}
 \put(-88,2){${}_{-j}$}
 \put(-35,10){\circle*{3}}
 \put(-37,2){${}_{j}$}
 \put(-15,10){\circle*{3}}
 \put(-17,2){${}_{k}$}
 \put(5,10){\circle*{3}}
 \put(3,2){${}_{i}$}
\put(25,10){\circle*{3}}
 \put(23,2){${}_{l}$}

 \qbezier(-140,10)(-67.5,70)(5,10)
 \qbezier(-120,10)(-47.5,70)(25,10)
 \qbezier(-100,10)(-90,30)(-80,10)
 \qbezier(-35,10)(-25,30)(-15,10)

 \put(35,12){\vector(2,0){17}}

\multiput(55,10)(5,0){36}
 {\circle*{1}}
 \put(60,10){\circle*{3}}
\put(52,2){${}_{-l}$}
\put(80,10){\circle*{3}}
\put(72,2){${}_{-i}$}
\put(100,10){\circle*{3}}
\put(92,2){${}_{-k}$}
 \put(120,10){\circle*{3}}
 \put(112,2){${}_{-j}$}
 \put(165,10){\circle*{3}}
 \put(163,2){${}_{j}$}
 \put(185,10){\circle*{3}}
 \put(183,2){${}_{k}$}
 \put(205,10){\circle*{3}}
  \put(202,2){${}_{i}$}
 \put(225,10){\circle*{3}}
  \put(223,2){${}_{l}$}

\qbezier(60,10)(122.5,70)(185,10)
\qbezier(100,10)(162.5,70)(225,10)
\qbezier(80,10)(100,40)(120,10)
\qbezier(165,10)(185,40)(205,10)

\end{picture}
\end{center}

\item[(CMM)] {\it Both external and internal arcs are mixed}: The external arc is either  $(-i,l)$ or $(i,-l)$  and the internal arc is $(-j,k)$ where $k<l$
 or resp. $k<i$. \\
 (a) If the external arc is $(-i,l)$ one has $L^{\prime\prime}=L_{(-j,k)(-i,l)}^-(-j,l)(-i,j)$ and we have to distinguish two subcases:
\begin{itemize}
\item{(i)} If $k>i$ then $SR_{L^{\prime\prime}}=R_{L'}$ where $L'=L_0(j,-l)(-j,l)(-i,k)(i,-k)$ and $\dim \mathcal B_{L'}=\dim\mathcal B_L-1\ \Rightarrow\ L'\in D(L).$\\
     One has $X_{L'}=\lim\limits_{m\rightarrow\infty}T_k(\frac{1}{m}).T_i(-m).U_{e_i-e_j}(-1).U_{e_l-e_k}(1).X_L$

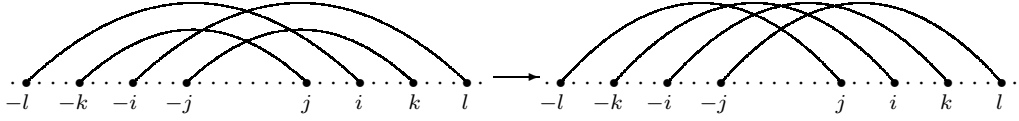
\begin{center}
\begin{picture}(100,60)(0,-10)
\multiput(-145,10)(5,0){36}%
{\circle*{1}}
\put(-140,10){\circle*{3}}
 \put(-148,2){${}_{-l}$}
\put(-120,10){\circle*{3}}
 \put(-128,2){${}_{-k}$}
 \put(-100,10){\circle*{3}}
 \put(-108,2){${}_{-i}$}
 \put(-80,10){\circle*{3}}
 \put(-88,2){${}_{-j}$}
 \put(-35,10){\circle*{3}}
 \put(-37,2){${}_{j}$}
 \put(-15,10){\circle*{3}}
 \put(-17,2){${}_{i}$}
 \put(5,10){\circle*{3}}
 \put(3,2){${}_{k}$}
\put(25,10){\circle*{3}}
 \put(23,2){${}_{l}$}

 \qbezier(-140,10)(-77.5,70)(-15,10)
 \qbezier(-120,10)(-77.5,50)(-35,10)
 \qbezier(-80,10)(-37.5,50)(5,10)
 \qbezier(-100,10)(-37.5,70)(25,10)

 \put(35,12){\vector(2,0){17}}

\multiput(55,10)(5,0){36}
 {\circle*{1}}
 \put(60,10){\circle*{3}}
\put(52,2){${}_{-l}$}
\put(80,10){\circle*{3}}
\put(72,2){${}_{-k}$}
\put(100,10){\circle*{3}}
\put(92,2){${}_{-i}$}
 \put(120,10){\circle*{3}}
 \put(112,2){${}_{-j}$}
 \put(165,10){\circle*{3}}
 \put(163,2){${}_{j}$}
 \put(185,10){\circle*{3}}
 \put(183,2){${}_{i}$}
 \put(205,10){\circle*{3}}
  \put(202,2){${}_{k}$}
 \put(225,10){\circle*{3}}
  \put(223,2){${}_{l}$}

\qbezier(60,10)(112.5,70)(165,10)
\qbezier(80,10)(132.5,70)(185,10)
\qbezier(100,10)(152.5,70)(205,10)
\qbezier(120,10)(172.5,70)(225,10)

\end{picture}
\end{center}

\item{(ii)} If $k<i$ then $SR_{L^{\prime\prime}}=R_{L'}$ where $L'=L_0(j,-l)(-j,l)(-k,k)(-i,i)$,
and $\dim \mathcal B_{L'}=\dim\mathcal B_L-1\ \Rightarrow\ L'\in D(L).$
Set $T_m:=T_l(m).T_j(m)$ and 
$U_m:=U_{e_i-e_k}(1).U_{e_l-e_j}(-1).U_{e_i-e_j}(-1).U_{e_l-e_k}(1).U_{e_k-e_j}(\frac{1}{2}).U_{e_l-e_i}(\frac{1}{2})$ then 
$X_{L'}=\lim\limits_{m\rightarrow\infty}T_m.U_m.X_L.$

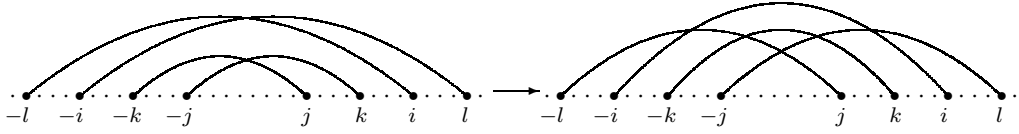
 \begin{center}
\begin{picture}(100,70)(0,-10)
%\put(-20,35){$L_\sigma=$}
\multiput(-145,10)(5,0){36}%
{\circle*{1}}
\put(-140,10){\circle*{3}}
 \put(-148,2){${}_{-l}$}
\put(-120,10){\circle*{3}}
 \put(-128,2){${}_{-i}$}
 \put(-100,10){\circle*{3}}
 \put(-108,2){${}_{-k}$}
 \put(-80,10){\circle*{3}}
 \put(-88,2){${}_{-j}$}
 \put(-35,10){\circle*{3}}
 \put(-37,2){${}_{j}$}
 \put(-15,10){\circle*{3}}
 \put(-17,2){${}_{k}$}
 \put(5,10){\circle*{3}}
 \put(3,2){${}_{i}$}
\put(25,10){\circle*{3}}
 \put(23,2){${}_{l}$}

 \qbezier(-140,10)(-67.5,70)(5,10)
 \qbezier(-120,10)(-47.5,70)(25,10)
 \qbezier(-100,10)(-67.5,40)(-35,10)
 \qbezier(-80,10)(-47.7,40)(-15,10)

 \put(35,12){\vector(2,0){17}}

\multiput(55,10)(5,0){36}
 {\circle*{1}}
 \put(60,10){\circle*{3}}
\put(52,2){${}_{-l}$}
\put(80,10){\circle*{3}}
\put(72,2){${}_{-i}$}
\put(100,10){\circle*{3}}
\put(92,2){${}_{-k}$}
 \put(120,10){\circle*{3}}
 \put(112,2){${}_{-j}$}
 \put(165,10){\circle*{3}}
 \put(163,2){${}_{j}$}
 \put(185,10){\circle*{3}}
 \put(183,2){${}_{k}$}
 \put(205,10){\circle*{3}}
  \put(202,2){${}_{i}$}
 \put(225,10){\circle*{3}}
  \put(223,2){${}_{l}$}

\qbezier(60,10)(112.5,60)(165,10)
\qbezier(120,10)(172.5,60)(225,10)
\qbezier(80,10)(142.5,80)(205,10)
\qbezier(100,10)(142.5,60)(185,10)

\end{picture}
\end{center}
\end{itemize}

 \noindent
 (b) The external arc is $(i,-l)$. Then $L^{\prime\prime}=L^-_{(-j,k),(i,-l)}(-j,i)(k,-l)$ and
there is no $L'$ such that $SR_{L^{\prime\prime}}$ is $R_{L'}$. Two maximal SLPs $\{L'_1, L'_2\}$ 
such that $R_{L'_s}\prec SR_{L^{\prime\prime}}$  are $L'_1=L_0(j,-i)(-j,i)(-k,l)(k,-l)$ and $L_2'=L'$ from (ii) of (a), and $s=1,2$ one has
$\dim \mathcal B_{L'_s}=\dim\mathcal B_L-1\ \Rightarrow\ L_1',L_2'\in D(L).$\\ 
Here $ X_{L_1'}= \lim\limits_{m\rightarrow\infty}T_l(-m).T_k(\frac{1}{m}).U_{e_l-e_j}(-1).U_{e_i-e_k}(1).X_l$.\\
Picture for  the new case:

\begin{center}
\begin{picture}(100,60)(0,-10)
%\put(-20,35){$L_\sigma=$}
\multiput(-145,10)(5,0){36}%
{\circle*{1}}
\put(-140,10){\circle*{3}}
 \put(-148,2){${}_{-l}$}
\put(-120,10){\circle*{3}}
 \put(-128,2){${}_{-i}$}
 \put(-100,10){\circle*{3}}
 \put(-108,2){${}_{-k}$}
 \put(-80,10){\circle*{3}}
 \put(-88,2){${}_{-j}$}
 \put(-35,10){\circle*{3}}
 \put(-37,2){${}_{j}$}
 \put(-15,10){\circle*{3}}
 \put(-17,2){${}_{k}$}
 \put(5,10){\circle*{3}}
 \put(3,2){${}_{i}$}
\put(25,10){\circle*{3}}
 \put(23,2){${}_{l}$}

 \qbezier(-140,10)(-67.5,70)(5,10)
 \qbezier(-120,10)(-47.5,70)(25,10)
 \qbezier(-100,10)(-67.5,40)(-35,10)
 \qbezier(-80,10)(-47.7,40)(-15,10)

 \put(35,12){\vector(2,0){17}}

\multiput(55,10)(5,0){36}
 {\circle*{1}}
 \put(60,10){\circle*{3}}
\put(52,2){${}_{-l}$}
\put(80,10){\circle*{3}}
\put(72,2){${}_{-i}$}
\put(100,10){\circle*{3}}
\put(92,2){${}_{-k}$}
 \put(120,10){\circle*{3}}
 \put(112,2){${}_{-j}$}
 \put(165,10){\circle*{3}}
 \put(163,2){${}_{j}$}
 \put(185,10){\circle*{3}}
 \put(183,2){${}_{k}$}
 \put(205,10){\circle*{3}}
  \put(202,2){${}_{i}$}
 \put(225,10){\circle*{3}}
  \put(223,2){${}_{l}$}

\qbezier(60,10)(122.5,70)(185,10)
\qbezier(120,10)(162.5,50)(205,10)
\qbezier(80,10)(122.5,50)(165,10)
\qbezier(100,10)(162.5,70)(225,10)

\end{picture}
\end{center}
\end{itemize}
\newpage
 \centerline{Index of Notation}
\bigskip

\begin{tabular}{lllll}
\ref{1.1}&$\mathcal O_x,\, \nil$&\quad&\ref{2.9a}&$\pi;$\\

\ref{1.2}&$\mathcal B_x$&\quad&\ref{4.2}&$ b_L(f),\, b(L),\, c_L(\left\langle i,j\right\rangle),\,c(L),$\\

\ref{1.3}&$\B_{SL_n},\, \mathfrak n_{\Sl_n},\, R,\, R^+,\, \Pi,\, Y_{(i,j)}, 
$&\quad&&$d(k,n),\, f_L([s,t]);$\\

&$l(L),\, {\bf LP}_n,\, {\bf LP}_n(k),\,Y_L,\, \mathfrak B(Y_L),$&&
\ref{4.1a}&$A(L),\, C(L),\, D(L),\, N(L),\, \widehat A(L),$\\ 

&$|S|,\,R_L,\,\preceq,\, \mathcal Y_n^{(2)},\, \overline{\mathcal A};$&&&$\widehat C(L),\, \widehat D(L),\, 
\widehat N(L),\, SR_L,\, s',\,t'$\\

\ref{1.4}&$\Phi,\, \Phi^+,\,\Delta,\,{\bf SLP}_{2n},\,{\bf SLP}_{2n}(k),$
&&&$L^-_{(i_1,j_1)\ldots(i_k,j_k)},\,(i,j)L,\,L(i,j);$\\
&$\B_{2n},\,X_\alpha,\,\nil_{2n},\, X_L,\, \mathcal X_{2n}^{(2)},\, \mathcal B_L;$&&\ref{5.2}&$E(L),\, E'(L)$\\
\ref{1.5}&$\mathcal V,\, \mathcal V_T;$&&\ref{5.3}&$l_i,\, r_j;$\\
\ref{2.3}&${\bf T}_{2n},\,{\bf U}_{2n},\, T_i(a),\, U_{\alpha}(a);$&&\ref{a1}&$[i:i],\,{\rm SDT};$\\
\ref{2.5}&$\lambda,\, \mathcal P(n),\, \mathcal P_1(2n),\, \mathcal O_\lambda,\, s_i,\, r_i,\, \lambda\geq \mu;$&&\ref{a2}&$w(\nil),\, \mathcal V_w,\, T_1,\, T_2,\, w_T,\ \mathcal V_T;$\\
\ref{2.8}&$(s,t)\in L,\, Ep(L),\, Fp(L),\,\left\langle i,j\right\rangle,$&&\ref{a3}&$L_T;$\\
&$(\pm i,\pm j),\, X_{(\pm i,j)},\, X_{(-i,i)},$&&\ref{a4}&$X_T,\,\mathcal B_T;$\\
&$Ep^+(L),\, Fp^+(L);$&&\\ 
\end{tabular}


\bigskip

\begin{thebibliography}{m}
\bibitem{Br}  M. Brion, {\it Quelques proprietes des espaces homogenes
spheriques}, Manuscripta Math. 55 (1986), pp. 191-198.
\bibitem{Ca}  G. Carnovale, {\it A classification of spherical
conjugacy classes in good characteristic}. Pacific Journal of
Mathematics 245 (2010), pp. 25-45.
\bibitem{Co-Mc}  David H. Collingwood, William M. McGovern, Nilpotent
orbits in semisimple Lie algebras, New York: Van Nostrand Reinhold
(1993).
\bibitem{Fr-Me}  L. Fresse and A. Melnikov, {\it Smooth orbital
varieties and orbital varieties with a dense B-orbit}, IRMN, to
appear.
\bibitem{Ga}  D. Garfinkle, {\it On the classification of primitive
ideals for complex classical lie algebras}, I. Comp. Math. 75
(1990), pp. 135-169.
\bibitem{Ge} M. Gerstenhaber, {\it On dominance and varieties of commuting matrices},
Ann. of Math. (2) 73 (1961), pp. 324-348.
\bibitem{Joseph-1984} A. Joseph: ``On the variety of a highest weight module''.
J. Algebra \textbf{88} (1984), 238--278.
\bibitem{Joseph-1997} A. Joseph, {\it Orbital varieties,
Goldie rank polynomials and unitary highest weight modules},
in: Algebraic and analytic methods in representation theory (S{\o}nderborg, 1994),
Perspect. Math., vol. 17, pp. 53--98, Academic Press, San Diego, CA, 1997.
\bibitem{Kr-Pr} H. Kraft, C. Procesi, {\it On the geometry of conjugacy classes in classical groups},
Comment. Math. Helv. 57 (1982), no. 4, pp. 539-602
\bibitem{Ma-We-Ze}  P. Magyar, J. Weyman and A. Zelevinsky,
{\it Symplectic multiple flag varieties of finite type},
Journal of Algebra 230 (2000), pp. 245-265.
\bibitem{Mc1}  W. M. McGovern, {\it Left cells and domino tableaux
in classical Weyl groups}, Composition Math 101 (1996), pp. 77-98.
\bibitem{Mc2}  W. M. McGovern, {\it On the Spaltenstein-Steinberg map
for classical Lie algebras}, Communications in Algebra 27 (1999),
pp. 2979-2993.
\bibitem{Me0} A. Melnikov, {\it $B$-orbits in solutions to the equation $X^2=0$ in triangular matrices},
J. Algebra \textbf{223} (2000), pp. 101--108.
\bibitem{Me1}  A. Melnikov, {\it Description of $B-$orbit closures of order 2 in
upper-triangular matrices}, Transform. Groups 11 (2) (2006), pp. 217-247.
\bibitem{Me2}  A. Melnikov, {\it B-orbits of nilpotency order 2 and
link patterns}, Indag. Math., NS 24, (2013), pp. 443-473.
\bibitem{Pa1}  D. Panyushev, {\it Complexity and nilpotent orbits},
Manuscripta Math. 83 (1994), pp. 223-237.
\bibitem{Pa2}  D. Panyushev, {\it On the orbits of a borel subgroup in abeliam ideals}, preprint 
April (2014).
\bibitem{Pi}  T. Pietraho, {\it Components of the Springer Fiber and
Domino Tableaux}, Journal of Algebra 272 (2004), pp. 711-729.
\bibitem{Pr} R. A. Proctor, {\it Classical Bruhat orders and lexicographic Shellability},J. Algebra 77 (1982), pp. 104-126.
\bibitem{Sp}  N. Spaltenstein, {\it On the fixed point set of a unipotent
element on the variety of Borel subgroup}, Topology 16 (1977), pp.
203-204.
\bibitem{St} R. Steinberg, {\it On the desingularization of the unipotent variety},
Invent. Math. (1976), pp. 209-224.
\bibitem{Tim}D. A. Timashev, {\it Generalization of the Bruhat decomposition}, Russian Acad. Sci. Izv. Math. 45 (1995), pp. 339-352.
\bibitem{Vi}  E. Vinberg, {\it Complexity of action of reductive groups},
Func. Anal. Appl. 20 (1986) pp. 1-11.
\end{thebibliography}
\end{document}